\documentclass[11pt]{amsart}
\usepackage{fullpage,xcolor,graphicx}
\usepackage[OT2,T1]{fontenc}
\usepackage{hyperref}
\usepackage{pifont}
\usepackage{amsthm}
\usepackage{amsfonts}
\usepackage{amssymb}
\usepackage[mathscr]{euscript}
\usepackage[all]{xy}
\usepackage{amsmath}
\usepackage{epsfig}
\usepackage{latexsym}
\usepackage{stackengine}
\usepackage{multirow}
\usepackage{hhline}
\usepackage{MnSymbol}

\usepackage[OT2,OT1]{fontenc} \newcommand\cyr
{
\renewcommand\rmdefault{wncyr} \renewcommand\sfdefault{wncyss} \renewcommand\encodingdefault{OT2} \normalfont
\selectfont
}
\DeclareTextFontCommand{\textcyr}{\cyr} 

\newcommand*\wbar[1]{%        % widebar 
  \hbox{ \kern-0.2em%
    \vbox{%
      \hrule height 0.5pt  % The actual bar
      \kern0.25ex%         % Distance between bar and symbol
      \hbox{%
        \kern-0.15em%       % Shortening on the left side
        \ensuremath{#1}%
        \kern-0.05em%       % Shortening on the right side
      }%
    }%
  \kern0.05em}%
} 

\newcommand*\wbarnew[1]{%        % widebar 
  \hbox{ \kern-0.2em%
    \vbox{%
      \hrule height 0.5pt  % The actual bar
      \kern0.25ex%         % Distance between bar and symbol
      \hbox{%
        \kern-0.35em%       % Shortening on the left side
        \ensuremath{#1}%
        \kern-0.05em%       % Shortening on the right side
      }%
    }%
  \kern0.05em}%
}

\newcommand{\Zh}{\hbox{\hspace{-5.8mm} \textcyr{Zh}}} %Use -6.4mm for 11 pt font and -5.8mm for 10pt font
\newcommand{\zh}{\hbox{\hspace{-5.8mm} \textcyr{zh}}} %Use -6.0mm for 11 pt font and -5.8mm for 10pt font

\DeclareMathOperator{\id}{id}

\DeclareMathOperator{\End}{End}

\newtheorem{theorem}{Theorem}[subsection]
\newtheorem{lemma}[theorem]{Lemma}

\newtheorem{corollary}[theorem]{Corollary}
\newtheorem{thm}{Theorem}

\theoremstyle{definition}
\newtheorem{definition}[theorem]{Definition}
\newtheorem{example}[theorem]{Example}

\newtheorem{conjecture}[theorem]{Conjecture}
\newtheorem*{theorem*}{Theorem}
\theoremstyle{remark}
\newtheorem{remark}[theorem]{Remark}

\numberwithin{equation}{subsection}

\begin{document}

\title[Lie Superalgebras generalization of the JKS-invariant]{Lie Superalgebra invariants and almost classical knots}

\author{Micah Chrisman}
\address{Department of Mathematics, The Ohio State University at Marion, Marion, Ohio, USA}
\email{chrisman.76@osu.edu}
\thanks{}

\author{Anup Poudel}
\address{Department of Mathematics, The Ohio State University, Columbus, Ohio, USA}
\email{poudel.33@osu.edu}
\thanks{}


\begin{abstract} A virtual link is said to be almost classical (AC) if it has a homologically trivial representative in some thickened surface $\Sigma \times [0,1]$, where $\Sigma$ is a closed orientable surface. Since AC links bound Seifert surfaces, they provide a useful window for observing the geometric topology of virtual knots. Here we take a different approach and look at AC links through the lens of quantum topology.  Two adjustments are needed to the existing theory. First, it is necessary to generalize the definition of AC to include virtual tangles and, in particular, virtual braids. Secondly, to distinguish AC and non-AC tangles, the additional structure of quantum supergroups is required. For each Lie superalgebra $\mathfrak{gl}(m|n)$, we define a pair of $U_q(\mathfrak{gl}(m|n))$ Reshetikhin-Turaev functors $Q^{m|n}$, $\widetilde{Q}^{m|n} \circ \Zh$ on framed virtual tangles. Here $\,\Zh$ denotes the Bar-Natan $\Zh$ construction. These functors unify the Alexander polynomial (AP) of AC links and the generalized Alexander polynomial (GAP) of all virtual links into a single quantum model: $Q^{1|1}$ recovers the AP of an AC link and for any virtual link $K$, $\widetilde{Q}^{1|1}\circ \Zh(K)$ is the 2-variable GAP. However, when $(m,n) \ne (1,1)$, these invariants are generally distinct from the AP and GAP. Furthermore, in contrast to the classical case, they are not determined by $m-n$. For example, there are virtual knots with trivial GAP but nontrivial $U_q(\mathfrak{gl}(2|2))$ and $U_q(\mathfrak{gl}(3|3))$ invariants.

Silver and Williams proved that the GAP vanishes on all AC links. Our main result is a generalization of this theorem to almost classical tangles and the $U_q(\mathfrak{gl}(m|n))$ Reshetikhin-Turaev functors. We prove that if $T$ is an almost classical tangle, then $\widetilde{Q}^{m|n}\circ \Zh(T)$ is conjugate to $Q^{m|n}(T)$, with conjugation determined by an Alexander numbering of $T$. The Silver-Williams theorem is then obtained for an AC link $T$ by setting $m=n=1$.

\end{abstract}


\maketitle

\section{Introduction}

\subsection{Background} \label{sec_back} Given any semisimple Lie algebra $\mathfrak{g}$, Reshetikhin and Turaev showed that for every irreducible representation of its quantum group $U_q(\mathfrak{g})$,  an invariant of framed links in $\mathbb{S}^3$ can be derived \cite{reshetikhin_turaev}.
For example, the Jones polynomial is recovered as the $U_q(\mathfrak{sl}(2))$ Reshetikhin-Turaev invariant. Kauffman and Saleur identified the Alexander polynomial (AP) with the quantum invariant arising from the vector representation of the non-semisimple Lie superalgebra $\mathfrak{gl}(1|1)$ \cite{kauffman_saleur_91}. Links, Gould, and Zhang extended and unified these approaches by defining Reshetikhin-Turaev invariants for each irreducible representation of $U_q(\mathfrak{gl}(m|n))$ \cite{links_gould_zhang} (see also Zhang \cite{zhang}). 

In Kauffman-Saleur \cite{kauffman_saleur_92} and Jaeger-Kauffman-Saleur \cite{jaeger_kauffman_saleur_94}, the $\mathfrak{gl}(1|1)$ quantum invariant was extended to a 2-variable polynomial invariant of links in thickened surfaces $\Sigma \times [0,1]$, where $\Sigma$ is closed and oriented. Soon afterwards, Kauffman discovered \emph{virtual knot theory}, which provides a systematic way of studying links in thickened surfaces using diagrammatic methods \cite{kauffman_vkt}. In particular, the JKS-polynomial is  a virtual link invariant (Sawollek \cite{sawollek_01}). Silver and Williams furthermore showed that the JKS-polynomial is equivalent to the \emph{generalized Alexander polynomial} (GAP) \cite{silver_williams_03}, which is instead constructed from the elementary ideal theory of a certain extension of the fundamental group \cite{silver_williams_00,silver_williams_06_II}. Since then, the JKS-polynomial has been shown to arise in many different guises (see e.g Kauffman-Radford \cite{kauffman_radford_00}, Manturov \cite{manturov_02}, Crans-Henrich-Nelson \cite{crans_henrich_nelson}, Boden et al. \cite{BDGGHN_15}). Henceforth, all equivalent versions will be referred to as the GAP.

The most well-known application of the GAP, proved by Silver-Williams \cite{silver_williams_06_II}, is that it vanishes on any \emph{almost classical (AC) link}. Recall that a virtual link is said to be almost classical if it can be represented by a homologically trivial link in some thickened surface $\Sigma \times [0,1]$. Equivalently, a virtual link is AC if it has a diagram with an Alexander numbering. Note that every classical link, viewed as a link in $S^2 \times [0,1]$, is almost classical. The Silver-Williams Theorem can then be viewed as a generalization of the result of Jaeger-Kauffman-Saleur that the JKS-polynomial vanishes  on all classical links \cite{jaeger_kauffman_saleur_94}. More recently, the GAP was shown to be trivial on any virtual link that is merely concordant to an almost classical link  \cite{boden_chrisman_21}. This implies, in particular, that the GAP is a slice obstruction for virtual knots. 

AC links has been studied independently of the GAP in many different ways. Since a homologically trivial link in $\Sigma \times [0,1]$ bounds a Seifert surface, it is natural to study them using techniques from geometric and low-dimensional topology. The standard method is to take the Seifert surface definition of a classical link invariant and apply this to homologically trivial links in $\Sigma \times [0,1]$. Importantly, AC links have an Alexander polynomial that can be defined in this way (Nakamura et. al \cite{NNST_12}, Boden et al. \cite{BGHNW_17}). Other examples include Levine-Tristam signatures, the Gordon-Litherland pairing, the Arf invariant, and the algebraic concordance group \cite{bcg2,BCK,chrisman_mukherjee}. In parallel with the classical case, every concordance class of AC knots contains both a prime hyperbolic AC knot and a prime satellite AC knot with the same Alexander polynomial \cite{chrisman_hyperbolic}. Non-geometric methods have also been used to study AC links; some examples are quandles (see e.g. Carter et al. \cite{cessw}) and skein-theoretic invariants (see e.g. Miller \cite{miller}). 

Although the AP and the GAP have been extensively studied using topological and group-theoretic methods, a common origin in quantum topology has not been fully worked out. Furthermore, the behavior of quantum invariants on AC links is generally unknown. The aim of this paper is to give a unified quantum model for the AP and the GAP, and extend this model to the family of quantum supergroups $U_q(\mathfrak{gl}(m|n))$. Using this unified and expanded model, the AC property can be studied all at once for a large family of quantum invariants. In addition, this approach will allow almost classicality to be explored in two new directions. First, we will extend the definition of AC to include virtual tangles and braids. Secondly, we obtain infinitely many generalizations of the GAP, one for every quantum supergroup $U_q(\mathfrak{gl}(m|n))$. Each such generalized GAP obstructs almost classicality in the appropriate category of virtual tangles.

\subsection{A pair of Reshitikhin-Turaev functors} \label{sec_mr_ext_invar} Our main tool for studying AC links, and eventually AC tangles, will be the Bar-Natan $\Zh$-construction. In order to describe our results, we begin by briefly recollecting its essential features. Full details are given in Section \ref{sec_zh_construction} ahead. Let $L$ be a virtual link diagram having $r$ components. The Bar-Natan $\Zh$-construction associates to $L$ an $r+1$ component virtual link diagram $\Zh(L)=L \cup \omega$, where $\omega$ is an over-crossing unknotted component. The new link $L \cup \omega$ is well-defined up to a relation called \emph{semi-welded equivalence}. If $L$ is almost classical, then $\Zh(L)$ is semi-welded equivalent to the split link $L \sqcup \bigcirc$. Moreover, the splitting of $\Zh(L)$ characterizes the AC property (\cite{chrisman_todd_23}, Corollary 3.3.6).  

A clue that the $\Zh$-construction is the right tool for our purposes comes from the GAP itself. Indeed, the GAP of a virtual knot can be defined as the usual multi-variable Alexander polynomial of $\Zh(L)$ \cite{boden_chrisman_21}. Our quantum invariants are likewise defined by factoring through the $\Zh$-construction. The initial step is to define a $U_q(\mathfrak{gl}(m|n))$ Reshetikhin-Turaev functor $Q^{m|n}:\mathcal{VT}^{\mathit{fr},\mathit{rot}} \to \textbf{Vect}_{\mathbb{C}(q)}$ from the category $\mathcal{VT}^{\mathit{fr},\mathit{rot}}$ of virtual tangles to the category of vector spaces over the field $\mathbb{C}(q)$. This \emph{virtual $U_q(\mathfrak{gl}(m|n))$ functor} is analogous to the $U_q(\mathfrak{gl}(m|n))$ Reshetikhin-Turaev functor of classical framed tangles. Note that in the virtual case, there is always a necessary additional framing corresponding to virtual curls (Kauffman \cite{kauffman_vkt}, Brochier \cite{brochier}). We denote this with the superscript \emph{rot} for \emph{rotational}. Next, to incorporate the $\Zh$-construction, $Q^{m|n}$ is extended to the category $\mathcal{SWT}^{\mathit{fr},\mathit{rot}}$ of \emph{semi-welded tangles}. This gives a functor $\widetilde{Q}^{m|n}:\mathcal{SWT}^{\mathit{fr},\mathit{rot}} \to \textbf{Vect}_{\mathbb{C}(q,w)}$. The extra variable $w$ in the field $\mathbb{C}(q,w)$ corresponds to the extra component $\omega$ from the $\Zh$-construction. Lastly, the $\Zh$-construction is made into a functor $\Zh:\mathcal{VT}^{\mathit{fr},\mathit{rot}} \to \widetilde{\mathcal{SWT}}|_{\alpha}^{\mathit{fr},\mathit{rot}}$. The target category $\widetilde{\mathcal{SWT}}|_{\alpha}^{\mathit{fr},\mathit{rot}}$ is a quotient of $\mathcal{SWT}^{\mathit{fr},\mathit{rot}}$ which is needed to make $\Zh$ into a well-defined monoidal functor. Our  \emph{extended $U_q(\mathfrak{gl}(m|n))$ functor} is the composite:
\[
\xymatrix{\mathcal{VT}^{\mathit{fr},\mathit{rot}} \ar[r]^-{\Zh} & \widetilde{\mathcal{SWT}}|_{\alpha}^{\mathit{fr},\mathit{rot}} \ar[r]^{\widetilde{Q}^{m|n}} & \textbf{Vect}_{\mathbb{C}(q,w)}}.
\]
Applying this to virtual link diagrams and deframing in the usual manner, we thereby obtain a \emph{virtual $U_q(\mathfrak{gl}(m|n))$ polynomial $f^{m|n}_L(q)$} (from $Q^{m|n}$) and a \emph{generalized $U_q(\mathfrak{gl}(m|n))$ polynomial $\widetilde{f}^{\, m|n}_L(q,w)$} (from $\widetilde{Q}^{m|n} \circ \Zh$). As will be shown, $\widetilde{f}^{\, m|n}_L(q,w)$ generalizes $f^{m|n}_L(q)$ in the same way that the GAP generalizes the AP.

\subsection{AC tangles $\&$ the Silver-Williams Theorem} In Section \ref{sec_ac_tangles}, we give a definition of AC tangles which restricts to the usual definition for virtual links $T$ when $T$ is considered as a morphism $\varnothing \to \varnothing$. The definition is also compatible with tensor products and composition: if $T_1,T_2$ are AC, then $T_1 \otimes T_2$ is AC and $T_1 \circ T_2$ is AC whenever the composition is sensible (more precisely, see Theorem \ref{thm_ac_makes_sense}). In Section \ref{sec_characterize}, we give a characterization theorem of the $\Zh$ functor in terms of AC tangles. This generalizes the characterization of the $\Zh$-construction for virtual links from \cite{chrisman_todd_23}. When applied to the pair of functors $\widetilde{Q}^{m|n}\circ \Zh(T)$ and $Q^{m|n}(T)$, we obtain the following result.

\begin{thm} \label{thm_A}  If $T:a \to a$ is an almost classical tangle diagram, then $\widetilde{Q}^{m|n}\circ \Zh(T)$ is conjugate to $Q^{m|n}(T)$, with conjugation depending only on a Alexander numbering of $T$.      
\end{thm}

As will be explained in the next subsection, this recovers to the Silver-Williams Theorem when $m=n=1$ and $T:\varnothing \to \varnothing$ is an AC link. The general case, however, provides new information. Consider, for example, the virtual braid group $\mathit{VB}_N$ on $N$ strands. For each $m,n$, $\widetilde{Q}^{m|n}\circ \Zh$ defines a representation $\widetilde{\rho}^{\,m|n}_N:\mathit{VB}_N \to \mathit{GL}((m+n)^N, \mathbb{C}(q,w))$. By the above theorem, if $\widetilde{Q}^{m|n}\circ \Zh(\beta)$ and $Q^{m|n}(\beta)$ are not conjugate for some $\beta \in \mathit{VB}_N$, then $\beta$ is not AC. Further applications of Lie superalgebra invariants to AC links are given in a separate paper \cite{chrisman_davis_poudel}, joint with K. Davis. For instance, a generalization of the extended $U_q(\mathfrak{gl}(2|1))$ Reshetikhin-Turaev functor is used to prove that the Seifert genus of AC knots is not additive under the connect sum operation.

\subsection{Application: the AP and the GAP} In \cite{sawollek_01}, Sawollek proved that the GAP satisfies a skein relation analogous to that of the classical AP. Furthermore, it was proved that for the set of \emph{all} virtual links, the AP defined from the fundamental group does not satisfy any linear skein relation. Much later, Boden et al. \cite{BGHNW_17} proved that when restricting to the family of AC links, the Alexander polynomial $\Delta_L(t)$ does in fact satisfy the expected skein relation. A unified quantum model for the AP and the GAP must reconcile these two skein relations for all virtual links. 

The key to this reconciliation lies in a well-known fact about Lie superalgebra quantum invariants of classical links. For any classical link, the $U_q(\mathfrak{gl}(1|1))$ quantum invariant is always trivial. To obtain the Alexander polynomial, one of the components must be cut so that the link becomes a 1-1 tangle (see Kauffman-Saleur \cite{kauffman_saleur_91} or Sartori \cite{sartori_15}). This is the \emph{modified trace argument} (Geer--Patureau-Mirand \cite{GPM}). Since our $U_q(\mathfrak{gl}(1|1))$ Reshetikhin-Turaev functors apply to all virtual tangles, this argument can now be generalized to virtual links. For the AP we show that if $L$ is AC and $T$ is a 1-1 AC tangle whose closure is $L$, then $\widetilde{Q}^{1|1}_T(q,w)=Q^{1|1}_T(q)$ is equal to $\Delta_L(q^{-2}) \cdot I_2$ (up to units), where $I_2$ is the $2 \times 2$ identity matrix (Theorem \ref{thm_gen_ac_eq_AP}). Moreover, the virtual $U_q(\mathfrak{gl}(1|1))$ invariant defines an AP for \emph{all} 1-1 virtual tangles. For the GAP, $G_L(s,t)$, of any virtual link $L$, the generalized $U_q(\mathfrak{gl}(1|1))$ polynomial $\widetilde{f}^{\,\,1,1}_L(q,w)$ is equal to $G_L(q^{-2}w^{-1},w)$ (Theorem \ref{thm_recover_GAP}). 
     
Given these identifications, Theorem \ref{thm_A} can now be seen as a generalization of the Silver-Williams Theorem to all $U_q(\mathfrak{gl}(m|n))$ and all AC tangles. Indeed, for any virtual link $L$, $f^{\,\,1|1}_L(q)=0$ for purely algebraic reasons. Theorem \ref{thm_A} implies that for any AC link, $\widetilde{f}^{\,\,1|1}_L(q,w)$ and $f^{1|1}_L(q)$ are conjugate. But $\widetilde{f}^{\,\,1|1}_L(q,w) \in \mathbb{C}(q,w)$, so that $\widetilde{f}^{\,\,1|1}_L(q,w)=f^{1|1}_L(q)=0$ when $L$ is AC. This is exactly the Silver-Williams Theorem.

The unification of the AP and GAP is facilitated by two features of our $U_q(\mathfrak{gl}(m|n))$ Reshetikhin-Turaev functors which are new to quantum invariants of virtual links. First, recall that the GAP is typically constructed from the generalized Burau representation of a virtual braid (see e.g. Kauffman-Radford \cite{kauffman_radford_00}, G\"{u}g\"{u}mc\"{u}-Kauffman \cite{MR4324388}). Here, we instead use the $\Zh$-construction, which is characterized by its behavior on AC links.  In our case, the generalized Burau representation is not a definition but a theorem which arises as an easy consequence of the $\Zh$-construction (see Lemma \ref{lemma_exterior_for_gen_func}). The second new feature is our virtual braiding. Instead of the usual switch map $x \otimes y \to y \otimes x$, we use a deformation which is more natural with respect to the classical Burau representation. This is needed so that for any 1-1 virtual tangle, the virtual $U_q(\mathfrak{gl}(1|1))$ invariant is a scalar multiple of the identity matrix $I_2$. This is not true with the usual switch map.

\subsection{Examples $\&$ conjectures} The skein relation for the $U_q(\mathfrak{gl}(m|n))$ polynomial depends only on $m-n$ (see Theorem \ref{thm_skein_relation}). This implies that $f_L^{m_1|n_1}(q)=f_L^{m_2|n_2}(q)$ whenever $L$ is classical and $m_1-n_1=m_2-n_2$. This is a well-known fact about Lie superalgebra invariants of classical links (see e.g. Queffelec \cite{queffelec_19}). However, this is not true in general when $L$ is non-classical. In Section \ref{sec_calc}, examples are given of virtual knots having the same generalized $U_q(\mathfrak{gl}(1|1))$ polynomial (and hence the same GAP), but with different generalized $U_q(\mathfrak{gl}(2|2))$ and $U_q(\mathfrak{gl}(3|3))$ polynomials. In particular, it will be shown the polynomials $\widetilde{f}^{\, m|n}_L(q,w)$ can detect non-AC rotational links in cases where the GAP is trivial.  Examples of virtual links having the same $U_q(\mathfrak{gl}(2|0))$ polynomials but different $U_q(\mathfrak{gl}(3|1))$ polynomials are also given.  

As previously mentioned, the generalized $U_q(\mathfrak{gl}(1|1))$ polynomial is a slice obstruction for virtual knots \cite{boden_chrisman_21}. It is reasonable to hope that this is also holds true of the generalized $U_q(\mathfrak{gl}(m|m))$ polynomials for all $m \ge 1$. Some computational evidence will be given in favor of this conjecture in Section \ref{sec_slice}. We show that for the thirteen slice virtual knots up to 4 classical crossings, the polynomials $\widetilde{f}^{\, m|m}_L(q,w)$ are trivial for $1 \le m \le 5$.  Assuming the conjecture is true for $m=2$, we show that the virtual knots 6.31445 and 6.62002 from Green's table \cite{green} are not slice. These virtual knots have vanishing graded genus, Rasmussen invariant, GAP, and extended Milnor invariants up to high order \cite{chrisman_22}. To our knowledge, no prior slice obstructions for these knots have been reported.

\subsection{Organization} Section \ref{sec_diag_prelim} reviews virtual tangles, almost classical links and the $\Zh$-construction. The categories $\mathcal{VT}$ of virtual tangles and $\mathcal{SWT}$ of semi-welded tangles are defined in Section \ref{sec_cats_defn}. The virtual $U_q(\mathfrak{gl}(m|n))$, semi-welded $U_q(\mathfrak{gl}(m|n))$, and $\Zh$ functors are defined in Section \ref{prelim}. Almost classical tangles are studied in Section \ref{sec_ac_tangles}. With these preparations, the generalized $U_q(\mathfrak{gl}(m|n))$ polynomials are finally defined in Section \ref{sec_invariants}. Theorem \ref{thm_A} is proved in Section \ref{sec_compose}. Recovery of the Alexander and generalized Alexander polynomial takes place in Section \ref{sec_alex_gen_alex}. Calculations and examples appear in Section \ref{sec_calc}.

In this paper, tangles in figures are oriented from bottom to top. To save vertical space in some figures, the tangles have been rotated clockwise by $\pi/2$. Figures where this has been done are marked with a $\circlearrowleft$ to indicate the corrective action that should be taken by the reader.

\section{Diagrammatic Preliminaries} \label{sec_diag_prelim}

\subsection{Virtual links, braids $\&$ tangles} \label{sec_virt_defns} A \emph{virtual link diagram} is a link diagram which, in addition to the usual positive or negative crossings, may also have a virtual crossing (see Figure \ref{fig_cross}). A virtual link diagram with no virtual crossings is said to be a \emph{classical link diagram}. All virtual link diagrams will be oriented. Virtual link diagrams are said to be equivalent if they are related by a finite sequence of extended Reidemeister moves (see Figure \ref{figvreidmoves}). An equivalence class of virtual link diagrams will be called a \emph{virtual link type}. If $L_1,L_2$ are diagrams having the same virtual link type, we write $L_1 \squigarrowleftright L_2$. A virtual link diagram having the same type as a classical link diagram is said to be \emph{classical}. By a \emph{framed virtual link type}, we mean an equivalence class generated by the set of relations shown in Figure 
\ref{figvreidmoves}, where the $R_1$ move is replaced with the move $R_1^{fr}$ (see Figure \ref{fig_framing}).

\begin{figure}[htb]
\begin{tabular}{|c|cccc||c|c||c|c|} \hline
\multirow[l]{2}*{\rotatebox{90}{ $\stackrel{\text{classical}}{ \text{crossings}}$ \hspace{.25cm}}}  & & & & & & \multirow[c]{2}*{\rotatebox{-90}{\hspace{-.09cm} $\stackrel{\text{virtual}}{\text{crossings}}$ }} & & \multirow[c]{2}*{\rotatebox{-90}{ \hspace{-.37cm} $\stackrel{\text{identity}}{\text{morphisms}}$}} \\
& \begin{tabular}{c} 
\def\svgwidth{.45in}
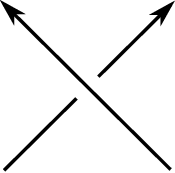 \end{tabular} & & \begin{tabular}{c} \def\svgwidth{.45in}
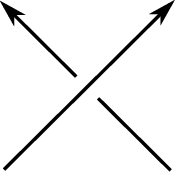 \end{tabular} & & \begin{tabular}{c} \def\svgwidth{.45in}
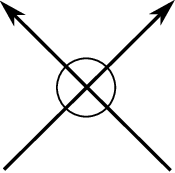 \end{tabular} & & \begin{tabular}{c} \def\svgwidth{.45in}
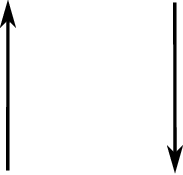 \end{tabular} & \\
& $\ominus$ & & $\oplus$ & & & & & \\\hline 
\multirow[l]{2}*{\rotatebox{90}{\hspace{.09cm} caps }} & \multicolumn{3}{c}{\begin{tabular}{ccc} \def\svgwidth{.45in}
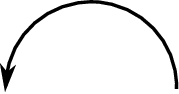 & & \def\svgwidth{.45in}
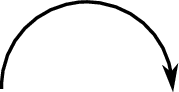 \end{tabular}} & & \multicolumn{3}{c|}{\begin{tabular}{ccc} \def\svgwidth{.45in}
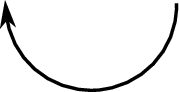 & & \def\svgwidth{.45in}
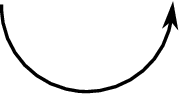  \end{tabular}} & \multirow[c]{2}*{\rotatebox{-90}{\hspace{-.37cm} cups }} \\ & & & & & \multicolumn{3}{c|}{} & \\ \hline
\end{tabular}
\caption{Elementary virtual tangles.} \label{fig_cross}
\end{figure}

A \emph{virtual tangle diagram} is likewise a tangle diagram which may also have virtual crossings. Given the usual  operations of tensor product $\otimes$ (horizontal juxtaposition) and composition $\circ$ (vertical stacking), virtual tangle diagrams are generated by the elementary virtual tangles shown in Figure \ref{fig_cross}. The equivalence relation on virtual tangles is generated by the classical tangle moves (see Figure \ref{fig_classical_tangle_moves} ) and the virtual tangle moves (see Figure \ref{fig_virtual_tangle_moves}). Every virtual link diagram may be represented as a composition of elementary virtual tangles and virtual link diagrams having the same type yield equivalent virtual tangle diagrams (Petit \cite{petit}).

\begin{figure}[htb]
\begin{tabular}{|cccccc|} \hline 
\multirow[l]{5}*{\rotatebox{90}{\tiny $\leftarrow$ Extended Reidemeister Moves $\rightarrow$ \hspace{.1cm}}} & \multicolumn{1}{|c}{\multirow[l]{2}*{\rotatebox{90}{ \tiny Classical Moves\hspace{.2cm}}}} & & & & \\
\multicolumn{1}{|c|}{} & & \multicolumn{4}{c|}{\begin{tabular}{cccc}  \begin{tabular}{c} \def\svgwidth{.72in}
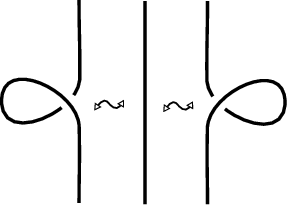 \\ $ R_1$ \end{tabular} & \begin{tabular}{c} \def\svgwidth{.45in}
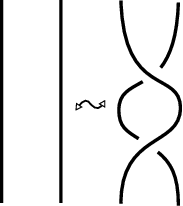 \\ $R_2$ \end{tabular} & \begin{tabular}{c} \def\svgwidth{.99in}
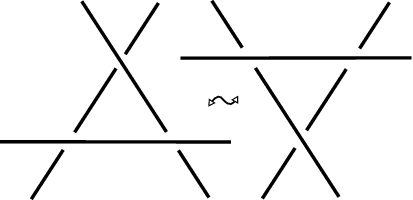 \\ $R_3$ \end{tabular} & \\  \end{tabular}}  \\  \hhline{~-----}
 & & & & & \\
& & \begin{tabular}{c} \def\svgwidth{.325in}
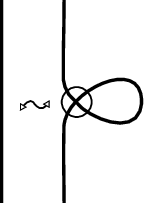 \\ $\mathit{VR}_1$ \end{tabular} & \begin{tabular}{c} \def\svgwidth{.36in}
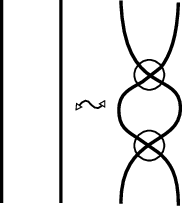 \\ $\mathit{VR}_2$\end{tabular} & \begin{tabular}{c} \def\svgwidth{.9in}
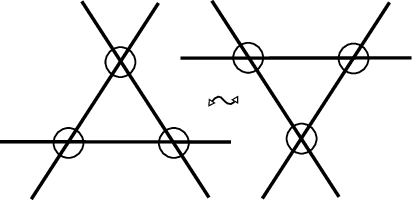 \\ $\mathit{VR}_3$ \end{tabular} & \begin{tabular}{c} \def\svgwidth{.9in}
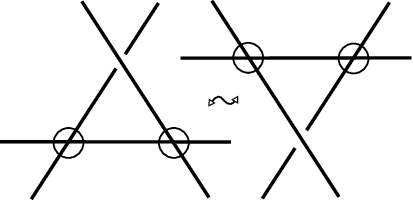 \\ $\mathit{VR}_4$ \end{tabular} \\ \hline
\end{tabular}
\caption{The classical and extended Reidemeister moves.} \label{figvreidmoves}
\end{figure}

A \emph{virtual braid} on $N$ strands is a braid on $N$ strands where, in addition to the classical crossing generators $\sigma_i^{\pm}$, $1 \le i \le N-1$, there are generators $\chi_i$, $1 \le i \le N-1$, in which the $i$-th and $(i+1)$-st strand have a single virtual crossing. See Figure \ref{fig_braid_gen}. If $\beta_1,\beta_2$ are $N$-strand virtual braids, $\beta_2 \circ \beta_1$ denotes the stacking of $\beta_2$ atop $\beta_1$. The virtual braid relations are then given by:
\begin{align} 
\label{rel_vbr1} \text{(Classical Relations)} \quad  & \left\{\begin{array}{cl} \sigma_i \sigma_i^{-1}=\sigma_i^{-1}\sigma_i=1 & \\
\sigma_i \sigma_j=\sigma_j \sigma_i &  \text{if } |i-j|>1 \\
\sigma_i \sigma_{i+1} \sigma_i=\sigma_{i+1} \sigma_i \sigma_{i+1}
  \end{array} \right. \\
\label{rel_vbr2} \text{(Virtual Relations)} \quad  & \left\{\begin{array}{cl} \chi_i^2=1 & \\
\chi_i \chi_j=\chi_j \chi_i &  \text{if } |i-j|>1 \\
\chi_i \chi_{i+1} \chi_i=\chi_{i+1} \chi_i \chi_{i+1}
  \end{array} \right. \\ 
\label{rel_vbr3} \text{(Mixed Relations)} \quad  & \left\{\begin{array}{cl}
\sigma_i \chi_j=\chi_j \sigma_i &  \text{if } |i-j|>1 \\
\sigma_i \chi_{i+1} \chi_i=\chi_{i+1} \chi_i \sigma_{i+1}
  \end{array} \right. 
\end{align}
The virtual braid group on $N$ strands will be denoted $\mathit{VB}_N$. In figures, virtual braids will be oriented upwards and the strands are numbered left to right from $1$ to $N$. For $\beta \in \mathit{VB}_N$ as in the left of Figure \ref{fig_closures}, its \emph{closure} is the virtual link shown in Figure \ref{fig_closures}, right. The \emph{partial closure} $\beta'$ of beta is the $1$-$1$ tangle shown in Figure \ref{fig_closures}, center.  Note that every virtual link type can be represented as the closure $\widehat{\beta}$ of some virtual braid $\beta$ (S. Kamada \cite{kamada_v_braid}).

\begin{figure}[htb]
\begin{tabular}{|c||c||c|} \hline
\begin{tabular}{c}\\ \def\svgwidth{1.2in}
%% Creator: Inkscape 1.0.2-2 (e86c870879, 2021-01-15), www.inkscape.org
%% Accompanies image file 'sigma_i_inverse.eps' (pdf, eps, ps)
\begingroup%
  \makeatletter%
  \providecommand\color[2][]{%
    \errmessage{(Inkscape) Color is used for the text in Inkscape, but the package 'color.sty' is not loaded}%
    \renewcommand\color[2][]{}%
  }%
  \providecommand\transparent[1]{%
    \errmessage{(Inkscape) Transparency is used (non-zero) for the text in Inkscape, but the package 'transparent.sty' is not loaded}%
    \renewcommand\transparent[1]{}%
  }%
  \providecommand\rotatebox[2]{#2}%
  \newcommand*\fsize{\dimexpr\f@size pt\relax}%
  \newcommand*\lineheight[1]{\fontsize{\fsize}{#1\fsize}\selectfont}%
  \ifx\svgwidth\undefined%
    \setlength{\unitlength}{291.82369523bp}%
    \ifx\svgscale\undefined%
      \relax%
    \else%
      \setlength{\unitlength}{\unitlength * \real{\svgscale}}%
    \fi%
  \else%
    \setlength{\unitlength}{\svgwidth}%
  \fi%
  \global\let\svgwidth\undefined%
  \global\let\svgscale\undefined%
  \makeatother%
  \begin{picture}(1,0.35400402)%
    \lineheight{1}%
    \setlength\tabcolsep{0pt}%
    \put(0,0){\includegraphics[width=\unitlength]{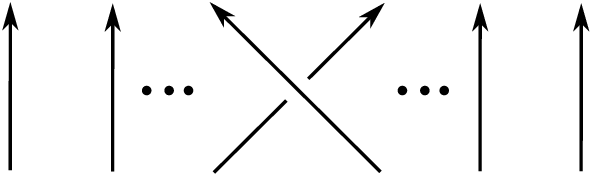}}%
    \put(0.93743895,0.32293176){\color[rgb]{0,0,0}\makebox(0,0)[lt]{\lineheight{40.54999924}\smash{\begin{tabular}[t]{l}$n$\end{tabular}}}}%
    \put(0.52,0.32125272){\color[rgb]{0,0,0}\makebox(0,0)[lt]{\lineheight{40.54999924}\smash{\begin{tabular}[t]{l}$i+1$\end{tabular}}}}%
    \put(0.29099413,0.32628995){\color[rgb]{0,0,0}\makebox(0,0)[lt]{\lineheight{40.54999924}\smash{\begin{tabular}[t]{l}$i$\end{tabular}}}}%
    \put(-0.00284444,0.32796899){\color[rgb]{0,0,0}\makebox(0,0)[lt]{\lineheight{40.54999924}\smash{\begin{tabular}[t]{l}$1$\end{tabular}}}}%
  \end{picture}%
\endgroup%
 \\ $\sigma_i^{-1}$\\ 
\end{tabular} & \begin{tabular}{c} \\ \def\svgwidth{1.2in}
%% Creator: Inkscape 1.0.2-2 (e86c870879, 2021-01-15), www.inkscape.org
%% Accompanies image file 'sigma_i.eps' (pdf, eps, ps)
\begingroup%
  \makeatletter%
  \providecommand\color[2][]{%
    \errmessage{(Inkscape) Color is used for the text in Inkscape, but the package 'color.sty' is not loaded}%
    \renewcommand\color[2][]{}%
  }%
  \providecommand\transparent[1]{%
    \errmessage{(Inkscape) Transparency is used (non-zero) for the text in Inkscape, but the package 'transparent.sty' is not loaded}%
    \renewcommand\transparent[1]{}%
  }%
  \providecommand\rotatebox[2]{#2}%
  \newcommand*\fsize{\dimexpr\f@size pt\relax}%
  \newcommand*\lineheight[1]{\fontsize{\fsize}{#1\fsize}\selectfont}%
  \ifx\svgwidth\undefined%
    \setlength{\unitlength}{291.82369523bp}%
    \ifx\svgscale\undefined%
      \relax%
    \else%
      \setlength{\unitlength}{\unitlength * \real{\svgscale}}%
    \fi%
  \else%
    \setlength{\unitlength}{\svgwidth}%
  \fi%
  \global\let\svgwidth\undefined%
  \global\let\svgscale\undefined%
  \makeatother%
  \begin{picture}(1,0.35400402)%
    \lineheight{1}%
    \setlength\tabcolsep{0pt}%
    \put(0,0){\includegraphics[width=\unitlength]{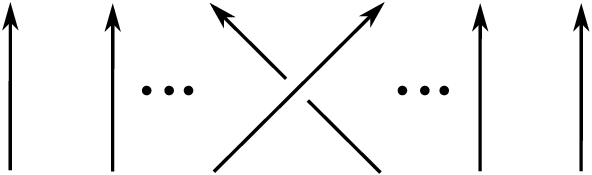}}%
    \put(0.93743895,0.32293176){\color[rgb]{0,0,0}\makebox(0,0)[lt]{\lineheight{40.54999924}\smash{\begin{tabular}[t]{l}$n$\end{tabular}}}}%
    \put(0.52,0.32125272){\color[rgb]{0,0,0}\makebox(0,0)[lt]{\lineheight{40.54999924}\smash{\begin{tabular}[t]{l}$i+1$\end{tabular}}}}%
    \put(0.29099413,0.32628995){\color[rgb]{0,0,0}\makebox(0,0)[lt]{\lineheight{40.54999924}\smash{\begin{tabular}[t]{l}$i$\end{tabular}}}}%
    \put(-0.00284444,0.32796899){\color[rgb]{0,0,0}\makebox(0,0)[lt]{\lineheight{40.54999924}\smash{\begin{tabular}[t]{l}$1$\end{tabular}}}}%
  \end{picture}%
\endgroup%
 \\ $\sigma_i$ \\ \end{tabular} & \begin{tabular}{c} \\  \def\svgwidth{1.2in}
%% Creator: Inkscape 1.0.2-2 (e86c870879, 2021-01-15), www.inkscape.org
%% Accompanies image file 'chi_i.eps' (pdf, eps, ps)
\begingroup%
  \makeatletter%
  \providecommand\color[2][]{%
    \errmessage{(Inkscape) Color is used for the text in Inkscape, but the package 'color.sty' is not loaded}%
    \renewcommand\color[2][]{}%
  }%
  \providecommand\transparent[1]{%
    \errmessage{(Inkscape) Transparency is used (non-zero) for the text in Inkscape, but the package 'transparent.sty' is not loaded}%
    \renewcommand\transparent[1]{}%
  }%
  \providecommand\rotatebox[2]{#2}%
  \newcommand*\fsize{\dimexpr\f@size pt\relax}%
  \newcommand*\lineheight[1]{\fontsize{\fsize}{#1\fsize}\selectfont}%
  \ifx\svgwidth\undefined%
    \setlength{\unitlength}{291.82369523bp}%
    \ifx\svgscale\undefined%
      \relax%
    \else%
      \setlength{\unitlength}{\unitlength * \real{\svgscale}}%
    \fi%
  \else%
    \setlength{\unitlength}{\svgwidth}%
  \fi%
  \global\let\svgwidth\undefined%
  \global\let\svgscale\undefined%
  \makeatother%
  \begin{picture}(1,0.35400402)%
    \lineheight{1}%
    \setlength\tabcolsep{0pt}%
    \put(0,0){\includegraphics[width=\unitlength]{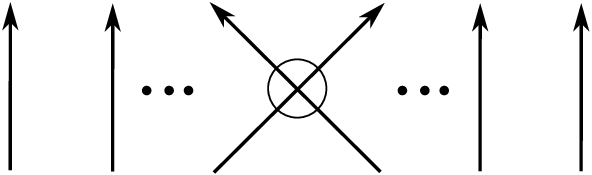}}%
    \put(0.93743895,0.32293176){\color[rgb]{0,0,0}\makebox(0,0)[lt]{\lineheight{40.54999924}\smash{\begin{tabular}[t]{l}$n$\end{tabular}}}}%
    \put(0.52,0.32125272){\color[rgb]{0,0,0}\makebox(0,0)[lt]{\lineheight{40.54999924}\smash{\begin{tabular}[t]{l}$i+1$\end{tabular}}}}%
    \put(0.29099413,0.32628995){\color[rgb]{0,0,0}\makebox(0,0)[lt]{\lineheight{40.54999924}\smash{\begin{tabular}[t]{l}$i$\end{tabular}}}}%
    \put(-0.00284444,0.32796899){\color[rgb]{0,0,0}\makebox(0,0)[lt]{\lineheight{40.54999924}\smash{\begin{tabular}[t]{l}$1$\end{tabular}}}}%
  \end{picture}%
\endgroup%
 \\ $\chi_i$  \\ \end{tabular} \\ \hline
\end{tabular}
\caption{The generators of the virtual braid group $\mathit{VB}_n$.} \label{fig_braid_gen}
\end{figure}
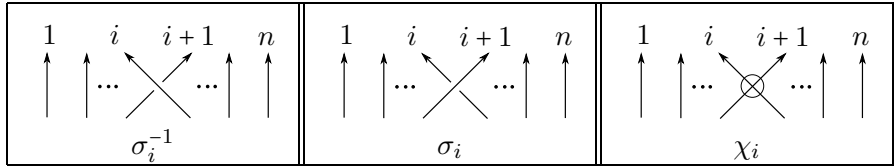

\begin{figure}[htb]
\begin{tabular}{|c||c|} \hline & \\
\begin{tabular}{c} $\underline{R_1^{fr}}$:  \end{tabular}   &  \begin{tabular}{c} \def\svgwidth{2.5in}
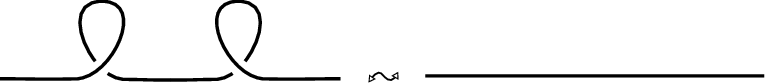 \\ \\ \end{tabular} \\ \hline & \\
\begin{tabular}{c} $\underline{\mathit{VR}_1^{\mathit{rot}}}$:  \end{tabular}   &  \begin{tabular}{c} \def\svgwidth{2.5in}
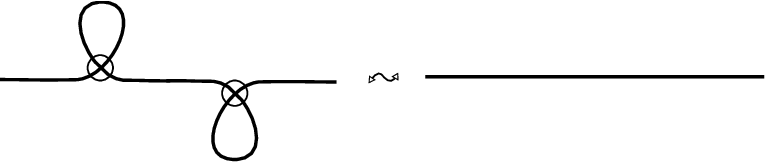 \\ \\ \end{tabular} \\ \hline
\end{tabular}
\caption{The framing relation and rotational relation.} \label{fig_framing}
\end{figure}

For quantum invariants, it is also necessary to study the category of \emph{rotational virtual tangles} \cite{kauffman_vkt, kauffman_rot}. The rotational equivalence relation is generated by replacing the move $\mathit{VR}_1$ with the move $\mathit{VR}_1^{\mathit{rot}}$ (see Figure \ref{fig_framing}). The $\mathit{VR}_1^{\mathit{rot}}$ move should be interpreted as a virtual framing relation. The virtual framing relation was given a geometric interpretation by Brochier \cite{brochier}, who showed that there is a one-to-one correspondence between virtual link diagrams and link diagrams on framed surfaces embedded into $\mathbb{R}^3$. If, in addition to the replacement of $\mathit{VR}_1$ with $\mathit{VR}_1^{\mathit{rot}}$, we replace $R_1$ with $R_1^{fr}$, we have the category of \emph{framed rotational virtual tangles}.

\begin{figure}[htb]
\begin{tabular}{|ccccccc|}\hline & & & & & & \\ 
& \begin{tabular}{c}\includegraphics[width=.8in]{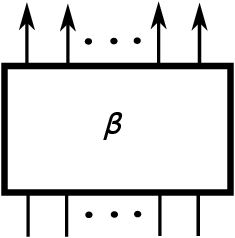} \\\end{tabular} & & \begin{tabular}{c}\includegraphics[width=1.5in]{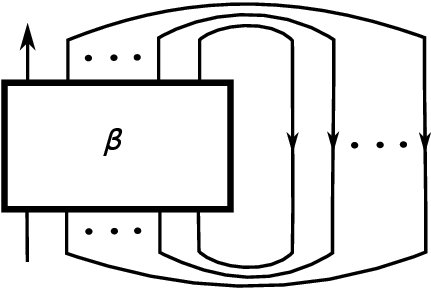} \end{tabular} & & \begin{tabular}{c} \includegraphics[width=1.6in]{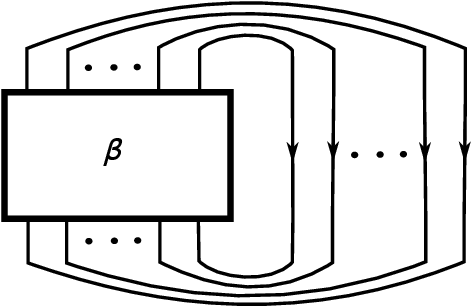} \end{tabular} & \\ & A virtual braid $\beta \ldots$ & & $\ldots$ its partial closure $\beta' \ldots$  & & $\ldots$ and its closure $\widehat{\beta}$.& \\ & & & & & & \\ \hline
\end{tabular}
\caption{The closure and partial closure of a virtual braid $\beta$.} \label{fig_closures}
\end{figure}

\subsection{Almost classical links} \label{sec_ac} For our purposes, the most useful definition of an AC link is the original one (Silver-Williams \cite{silver_williams_00}). Let $L$ be a virtual link diagram. An \emph{arc} of $L$ is a path on the diagram between one classical under-crossing and the next. A \emph{short arc} of $L$ is a path between two adjacent classical crossings. Note that virtual crossings are ignored when determining the arcs and short arcs of a virtual link diagram. If $\mathcal{S}$ is the set of short arcs of $L$, an \emph{Alexander numbering} of $L$ is a function $\Gamma:\mathcal{S} \to \mathbb{Z}$ such that the assignment obeys the rule shown in Figure \ref{fig_alex_numberings} at every classical crossing of $L$. If $L$ has an Alexander numbering, then $L$ is said to be \emph{Alexander numerable}.

\begin{figure}[htb]
    \begin{tabular}{|cc|} \hline & \\
    \multicolumn{2}{|c|}{\begin{tabular}{cc} $\begin{array}{c} \underline{T_0} \end{array}$ & \def\svgwidth{3in}%% Creator: Inkscape 1.0.2-2 (e86c870879, 2021-01-15), www.inkscape.org
%% Accompanies image file 'far_away_commutes.eps' (pdf, eps, ps)
\begingroup%
  \makeatletter%
  \providecommand\color[2][]{%
    \errmessage{(Inkscape) Color is used for the text in Inkscape, but the package 'color.sty' is not loaded}%
    \renewcommand\color[2][]{}%
  }%
  \providecommand\transparent[1]{%
    \errmessage{(Inkscape) Transparency is used (non-zero) for the text in Inkscape, but the package 'transparent.sty' is not loaded}%
    \renewcommand\transparent[1]{}%
  }%
  \providecommand\rotatebox[2]{#2}%
  \newcommand*\fsize{\dimexpr\f@size pt\relax}%
  \newcommand*\lineheight[1]{\fontsize{\fsize}{#1\fsize}\selectfont}%
  \ifx\svgwidth\undefined%
    \setlength{\unitlength}{515.42114203bp}%
    \ifx\svgscale\undefined%
      \relax%
    \else%
      \setlength{\unitlength}{\unitlength * \real{\svgscale}}%
    \fi%
  \else%
    \setlength{\unitlength}{\svgwidth}%
  \fi%
  \global\let\svgwidth\undefined%
  \global\let\svgscale\undefined%
  \makeatother%
  \begin{picture}(1,0.29997762)%
    \lineheight{1}%
    \setlength\tabcolsep{0pt}%
    \put(0,0){\includegraphics[width=\unitlength]{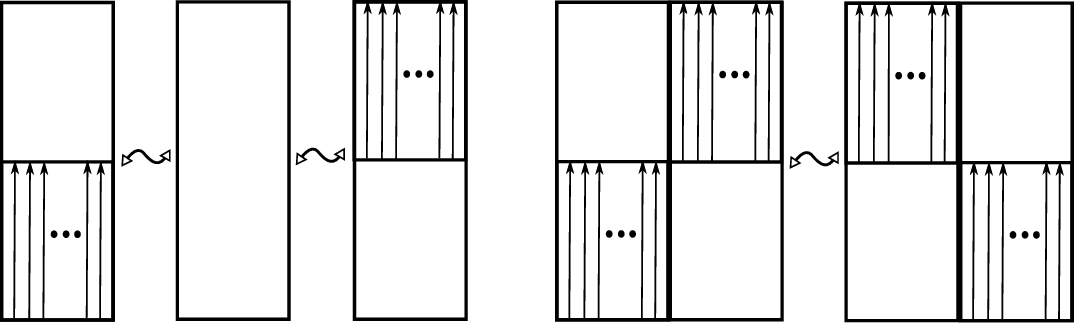}}%
    \put(0.03998963,0.21346042){\color[rgb]{0,0,0}\makebox(0,0)[lt]{\lineheight{40.54999924}\smash{\begin{tabular}[t]{l}$T$\end{tabular}}}}%
    \put(0.20341424,0.15007032){\color[rgb]{0,0,0}\makebox(0,0)[lt]{\lineheight{40.54999924}\smash{\begin{tabular}[t]{l}$T$\end{tabular}}}}%
    \put(0.36826914,0.07067913){\color[rgb]{0,0,0}\makebox(0,0)[lt]{\lineheight{40.54999924}\smash{\begin{tabular}[t]{l}$T$\end{tabular}}}}%
    \put(0.55608725,0.22140315){\color[rgb]{0,0,0}\makebox(0,0)[lt]{\lineheight{40.54999924}\smash{\begin{tabular}[t]{l}$S$\end{tabular}}}}%
    \put(0.66245947,0.07292614){\color[rgb]{0,0,0}\makebox(0,0)[lt]{\lineheight{40.54999924}\smash{\begin{tabular}[t]{l}$T$\end{tabular}}}}%
    \put(0.82567039,0.07175815){\color[rgb]{0,0,0}\makebox(0,0)[lt]{\lineheight{40.54999924}\smash{\begin{tabular}[t]{l}$S$\end{tabular}}}}%
    \put(0.93303553,0.21827519){\color[rgb]{0,0,0}\makebox(0,0)[lt]{\lineheight{40.54999924}\smash{\begin{tabular}[t]{l}$T$\end{tabular}}}}%
  \end{picture}%
\endgroup%
 \end{tabular}} \\ & \\ \hline & \\
\multicolumn{2}{|c|}{\begin{tabular}{ccc} \begin{tabular}{c}  \def\svgwidth{1.1in}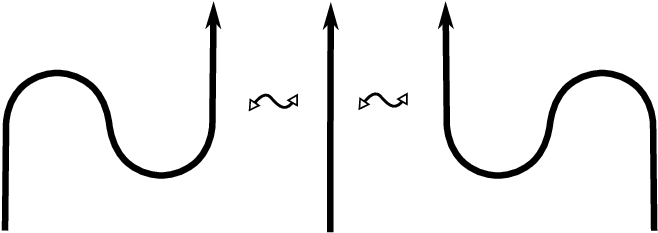 \\ \underline{$T_1$}  \end{tabular} & \begin{tabular}{c} \def\svgwidth{1.1in}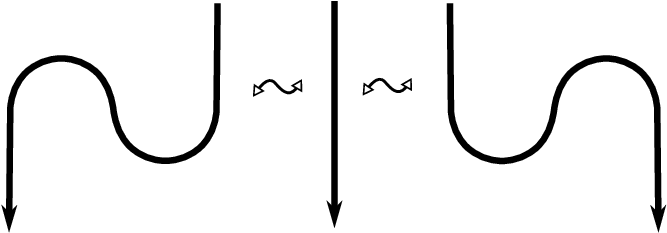 \\ \underline{$T_2$} \end{tabular}  & \begin{tabular}{c} \def\svgwidth{.9in}
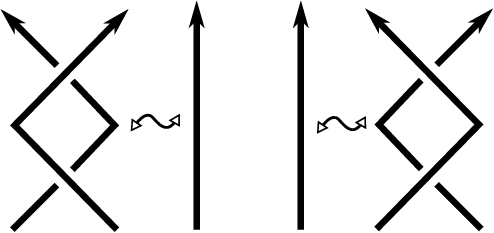 \\ \underline{$T_3$} \end{tabular} \end{tabular}} \\ & \\ \hline & \\
\multicolumn{2}{|c|}{\begin{tabular}{ccc} \begin{tabular}{c} \def\svgwidth{1.1in}
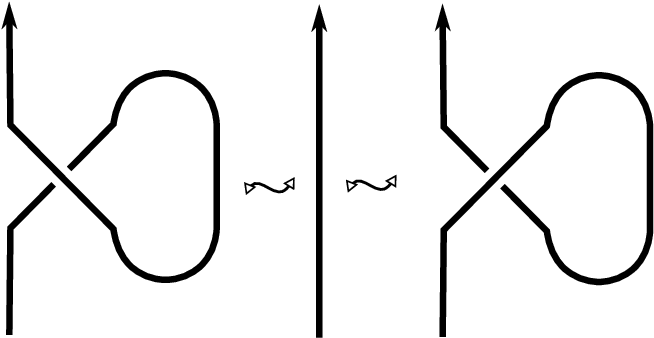 \\ \underline{$T_4$} \end{tabular} & & \begin{tabular}{c} \def\svgwidth{.85in}
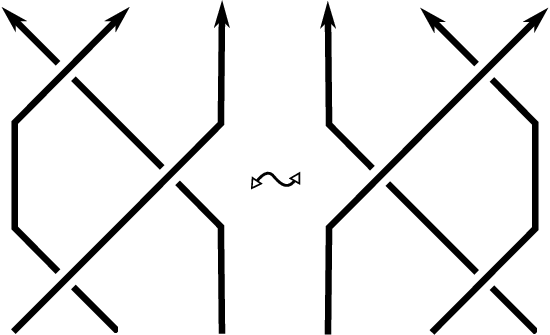 \\ \underline{$T_5$} \end{tabular} \end{tabular}} \\ & \\ \hline & \\
\multicolumn{2}{|c|}{\begin{tabular}{ccc} \begin{tabular}{c} \def\svgwidth{.85in}
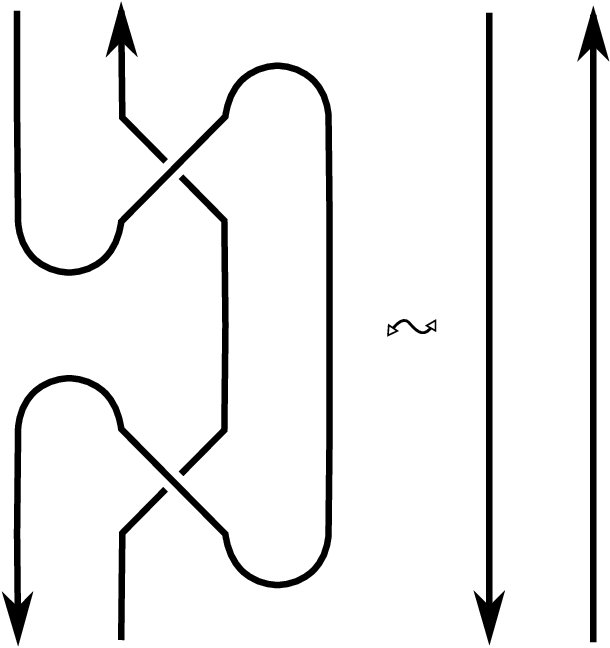 \\ \underline{$T_6$} \end{tabular} & & \begin{tabular}{c} \def\svgwidth{.85in}
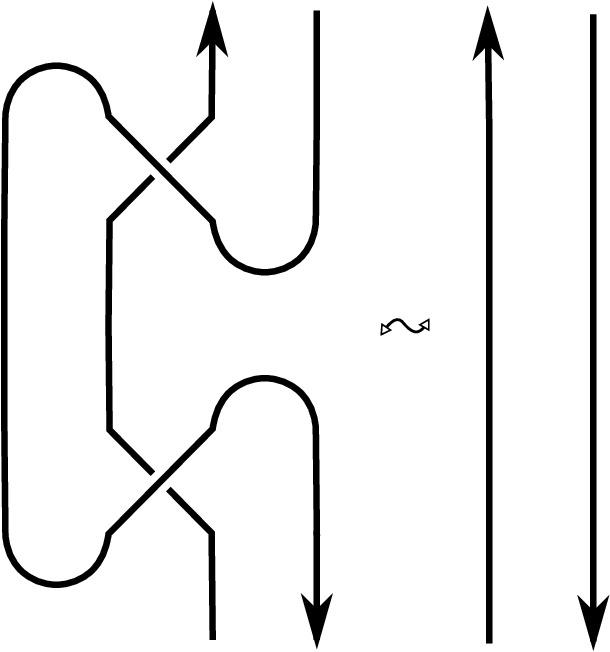 \\ \underline{$T_7$} \end{tabular} \end{tabular}} \\ & \\ \hline & \\ \multicolumn{2}{|c|}{\begin{tabular}{cc} \begin{tabular}{c} \def\svgwidth{1.6in}
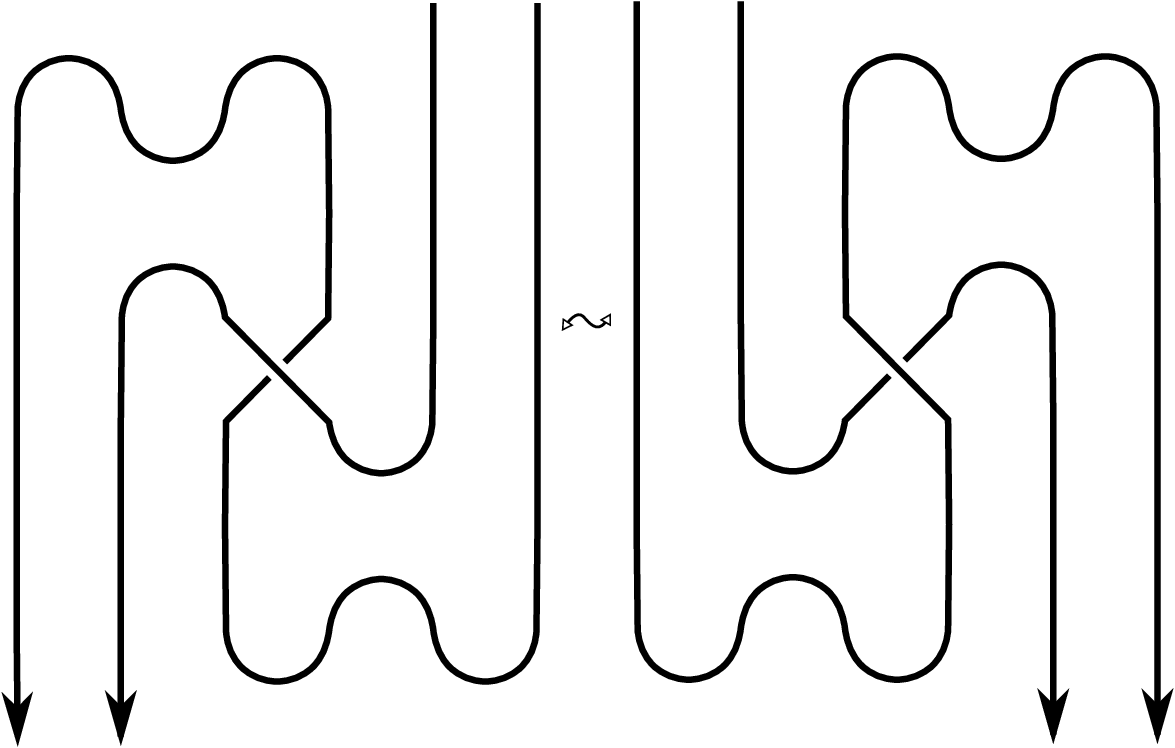 \\ \underline{$T_8$} \end{tabular} & \begin{tabular}{c} \def\svgwidth{1.6in}
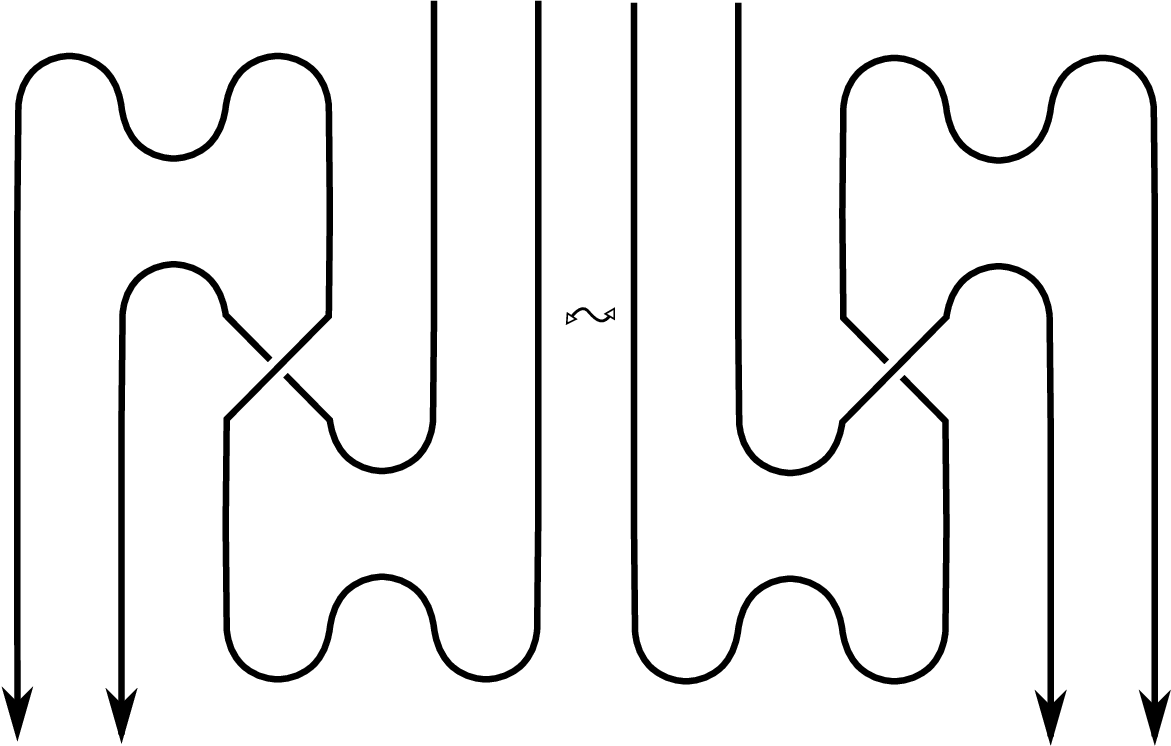 \\ \underline{$T_9$} \end{tabular} \end{tabular}} \\ & \\ \hline
\end{tabular}
\caption{Classical tangle moves.} \label{fig_classical_tangle_moves}
\end{figure}
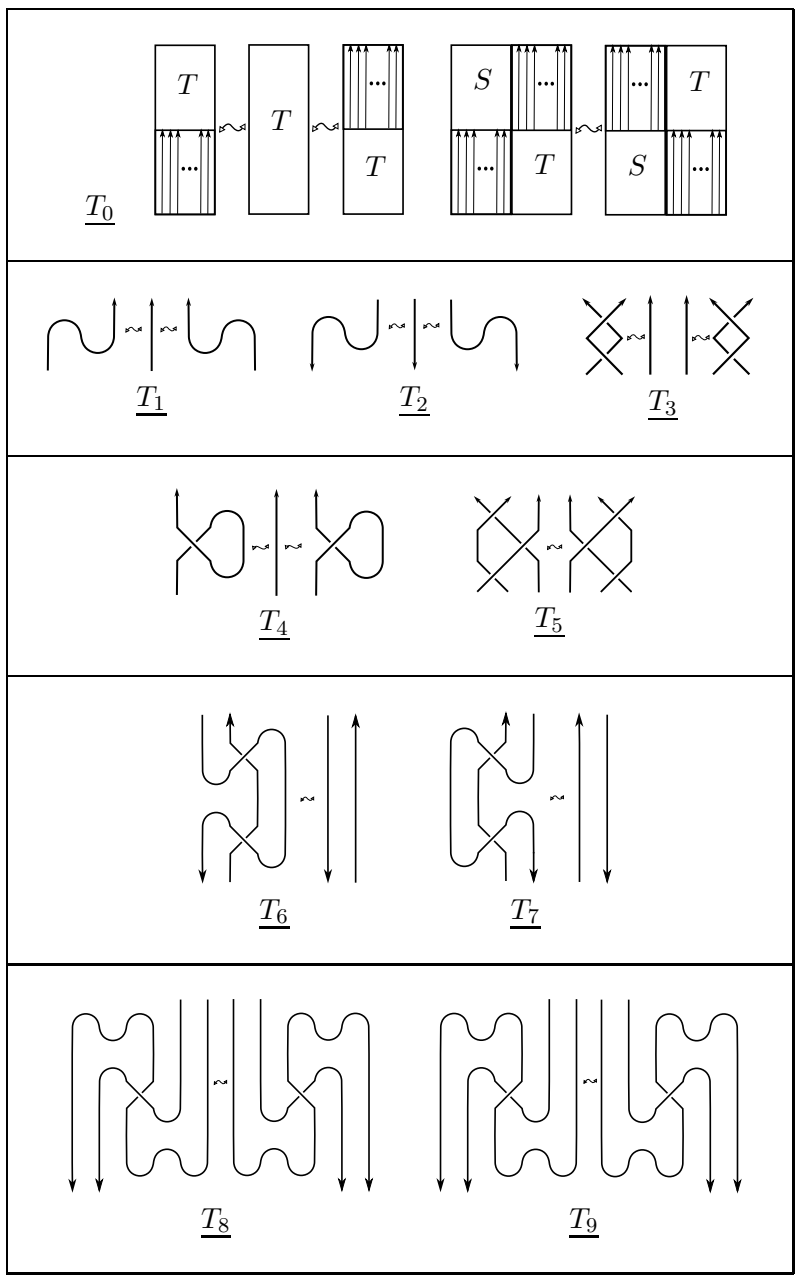

\begin{figure}[htb]
    \begin{tabular}{|cc|} \hline & \\ 
\multicolumn{2}{|c|}{\begin{tabular}{cccc} \begin{tabular}{c} \def\svgwidth{1.1in}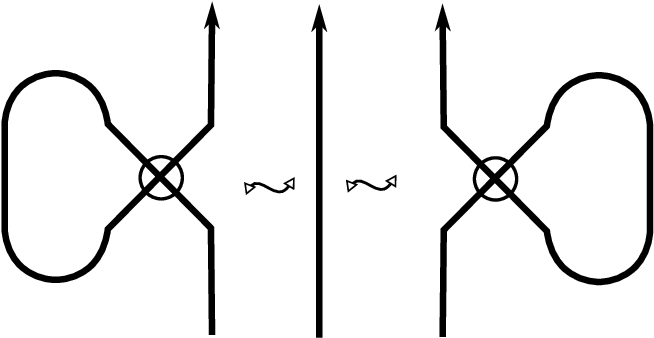 \\ \underline{$VT_1$} \end{tabular} & \begin{tabular}{c} \def\svgwidth{.67in}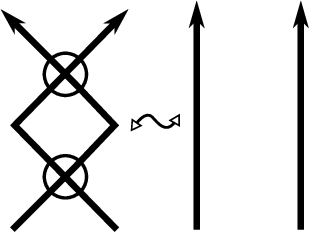 \\ \underline{$VT_2$} \end{tabular} & \begin{tabular}{c} \def\svgwidth{.9in}
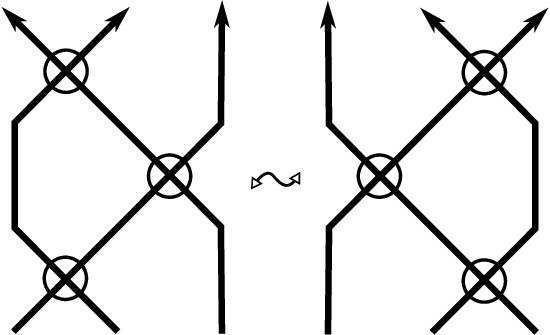 \\ \underline{$VT_3$} \end{tabular} & \begin{tabular}{c} \def\svgwidth{.9in}
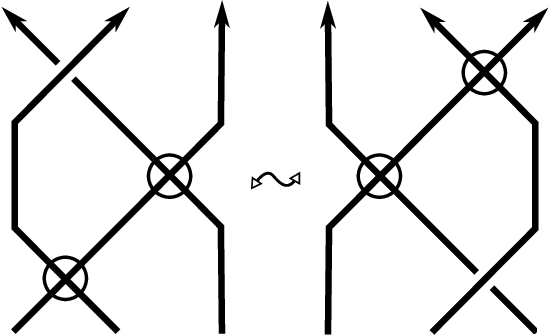 \\ \underline{$VT_4$} \end{tabular} \end{tabular}} \\ & \\ \hline & \\
\multicolumn{2}{|c|}{\begin{tabular}{ccc} \begin{tabular}{c} \def\svgwidth{.85in}
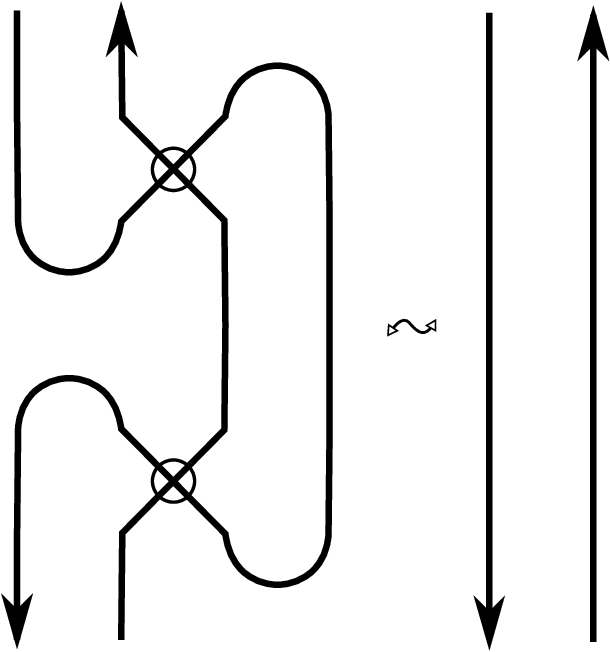 \\ \underline{$VT_5$} \end{tabular} & \begin{tabular}{c} \def\svgwidth{.85in}
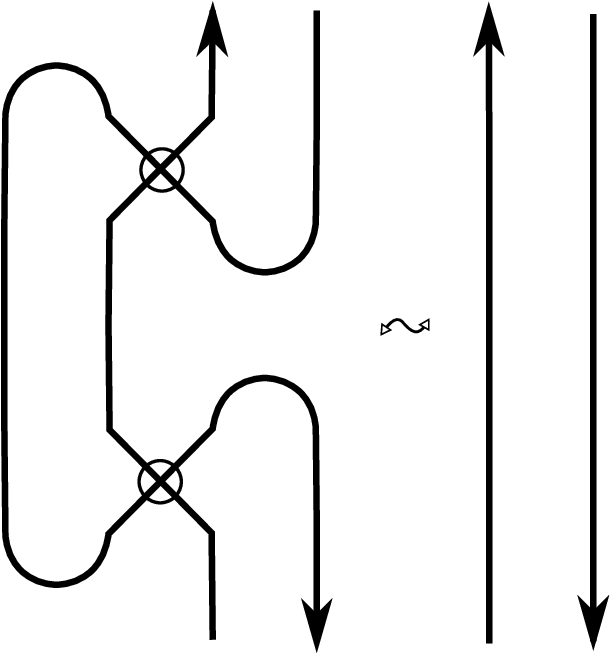  \\ \underline{$VT_6$} \end{tabular} & \begin{tabular}{c} \def\svgwidth{1.6in}
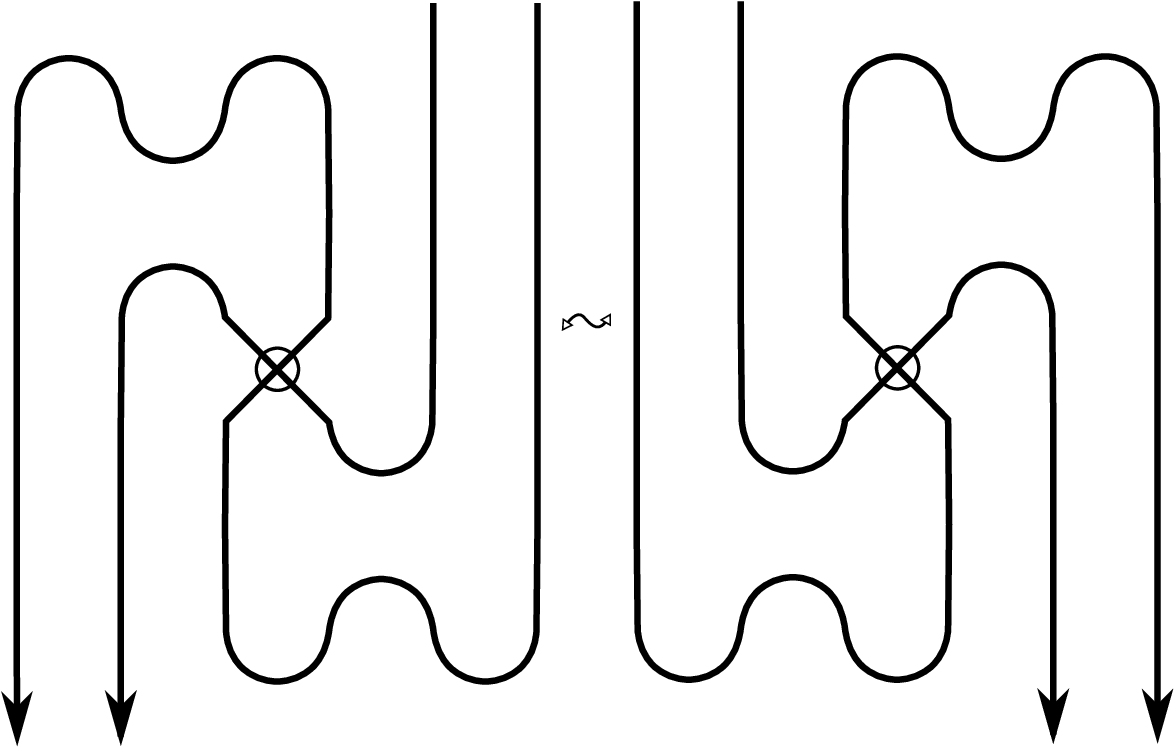 \\ \underline{$VT_7$} \end{tabular} \end{tabular}} \\ & \\ \hline
\end{tabular}
\caption{Virtual tangle moves.} \label{fig_virtual_tangle_moves}
\end{figure}

\begin{definition}[Almost classical link] A virtual link type $L$ is said to be \emph{almost classical (AC)} if $L$ has an Alexander numerable diagram. \end{definition}

\begin{figure}[htb]
\begin{tabular}{|ccc||c|} \hline & & & \\  \begin{tabular}{c} \includegraphics[scale=.35]{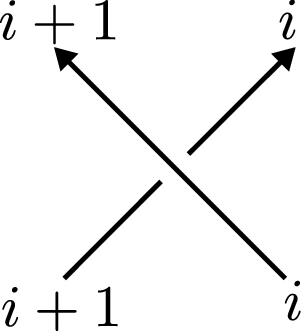} \end{tabular} & & \begin{tabular}{c}\includegraphics[scale=.35]{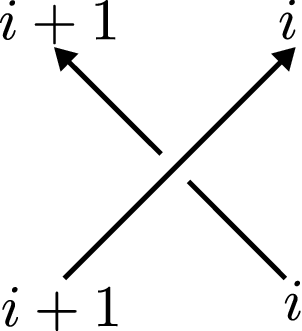} \end{tabular} & \begin{tabular}{c} \includegraphics[scale=.35]{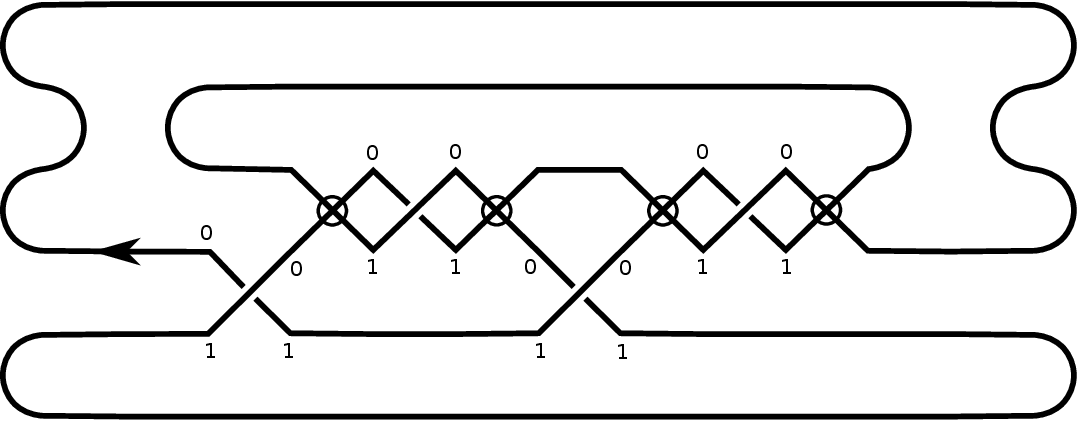} \end{tabular} \\ \multicolumn{3}{|c||}{\underline{Alexander numberings}}  & \underline{An Alexander numerable diagram}  \\ & & & \\ \hline \end{tabular}
\caption{ (Left) In an Alexander numbering, the labels of the short arcs on the right of the crossing are one less than those on the left. (Right) An example of an Alexander numbering.} \label{fig_alex_numberings}
\end{figure}

 As mentioned in Section \ref{sec_back}, an equivalent definition can be given in terms of links in thickened surfaces. Recall that every virtual link diagram can be represented by a link diagram on a closed oriented surface (see Carter-Kamada-Saito \cite{CKS}). A virtual link type $L$ is almost classical if and only if it has a diagram $D$ on a surface $\Sigma$ such that $[D]=0 \in H_1(\Sigma;\mathbb{Z})$. Consequently, a link $\mathcal{L} \subset \Sigma \times [0,1]$ with diagram $D$ bounds a Seifert surface $F$ in $\Sigma \times [0,1]$.  An algorithm for drawing such virtual Seifert surfaces was given in \cite{chrisman_vss}. Every classical link is almost classical, but in general almost classical links are rare. Of the 92800 virtual knots having classical crossing number at most six, only 77 are AC \cite{BGHNW_17}.

\subsection{The $\!\!\!\!\Zh$-construction} \label{sec_zh_construction} A \emph{semi-welded link diagram} is an $(r+1)$-component virtual link diagram with one distinguished component. The distinguished component is alternately called the \emph{$(r+1)$-st component}, the \emph{last component},  or the \emph{$\omega$-component}. In figures, the $\omega$-component will be colored dark blue. The \emph{semi-welded move} is shown in Figure \ref{fig_semi_welded_move}. Here the over-crossing arc is from the $\omega$-component and the other two arcs are not in the $\omega$-component. The equivalence relation on semi-welded link diagrams generated by extended Reidemeister moves and semi-welded moves is called \emph{semi-welded equivalence}. Here it is assumed this is an equivalence of labeled link diagrams, with components labeled $1, \ldots,r+1$. In particular, a semi-welded equivalence preserves the $\omega$-component. If $L_1,L_2$ are semi-welded equivalent diagrams, we will write $L_1 \stackrel{sw}{\squigarrowleftright} L_2$.

\begin{figure}[htb]
\begin{tabular}{|ccccc|} 
\hline & & & & \\
& 
\def\svgwidth{1.1in} \tiny
%% Creator: Inkscape 1.0.2-2 (e86c870879, 2021-01-15), www.inkscape.org
%% Accompanies image file 'welded_move.eps' (pdf, eps, ps)
\begingroup%
  \makeatletter%
  \providecommand\color[2][]{%
    \errmessage{(Inkscape) Color is used for the text in Inkscape, but the package 'color.sty' is not loaded}%
    \renewcommand\color[2][]{}%
  }%
  \providecommand\transparent[1]{%
    \errmessage{(Inkscape) Transparency is used (non-zero) for the text in Inkscape, but the package 'transparent.sty' is not loaded}%
    \renewcommand\transparent[1]{}%
  }%
  \providecommand\rotatebox[2]{#2}%
  \newcommand*\fsize{\dimexpr\f@size pt\relax}%
  \newcommand*\lineheight[1]{\fontsize{\fsize}{#1\fsize}\selectfont}%
  \ifx\svgwidth\undefined%
    \setlength{\unitlength}{234.89840782bp}%
    \ifx\svgscale\undefined%
      \relax%
    \else%
      \setlength{\unitlength}{\unitlength * \real{\svgscale}}%
    \fi%
  \else%
    \setlength{\unitlength}{\svgwidth}%
  \fi%
  \global\let\svgwidth\undefined%
  \global\let\svgscale\undefined%
  \makeatother%
  \begin{picture}(1,0.65155053)%
    \lineheight{1}%
    \setlength\tabcolsep{0pt}%
    \put(0,0){\includegraphics[width=\unitlength]{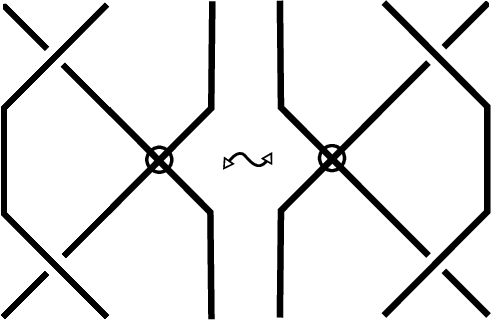}}%
    \put(0.46287291,0.35604462){\color[rgb]{0,0,0}\makebox(0,0)[lt]{\lineheight{40.54999924}\smash{\begin{tabular}[t]{l}$w$\end{tabular}}}}%
  \end{picture}%
\endgroup%
& &
\def\svgwidth{1.1in} \tiny
%% Creator: Inkscape 1.0.2-2 (e86c870879, 2021-01-15), www.inkscape.org
%% Accompanies image file 'semi_welded_move.eps' (pdf, eps, ps)
\begingroup%
  \makeatletter%
  \providecommand\color[2][]{%
    \errmessage{(Inkscape) Color is used for the text in Inkscape, but the package 'color.sty' is not loaded}%
    \renewcommand\color[2][]{}%
  }%
  \providecommand\transparent[1]{%
    \errmessage{(Inkscape) Transparency is used (non-zero) for the text in Inkscape, but the package 'transparent.sty' is not loaded}%
    \renewcommand\transparent[1]{}%
  }%
  \providecommand\rotatebox[2]{#2}%
  \newcommand*\fsize{\dimexpr\f@size pt\relax}%
  \newcommand*\lineheight[1]{\fontsize{\fsize}{#1\fsize}\selectfont}%
  \ifx\svgwidth\undefined%
    \setlength{\unitlength}{234.89887566bp}%
    \ifx\svgscale\undefined%
      \relax%
    \else%
      \setlength{\unitlength}{\unitlength * \real{\svgscale}}%
    \fi%
  \else%
    \setlength{\unitlength}{\svgwidth}%
  \fi%
  \global\let\svgwidth\undefined%
  \global\let\svgscale\undefined%
  \makeatother%
  \begin{picture}(1,0.65155205)%
    \lineheight{1}%
    \setlength\tabcolsep{0pt}%
    \put(0,0){\includegraphics[width=\unitlength]{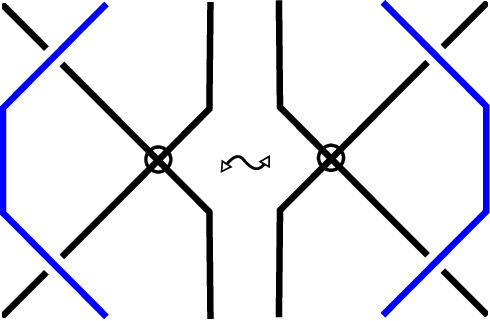}}%
    \put(0.44323135,0.34193985){\color[rgb]{0,0,0}\makebox(0,0)[lt]{\lineheight{40.54999924}\smash{\begin{tabular}[t]{l}$sw$\end{tabular}}}}%
  \end{picture}%
\endgroup%
 & \\
& \underline{Welded move} & & \underline{Semi-welded move} & \\
& (forbidden) & & (permitted) & \\ & & & & \\ \hline
\end{tabular}
\caption{The welded and semi-welded moves.}
\label{fig_semi_welded_move}
\end{figure}

\begin{figure}[htb]
\[
\begin{tabular}{|cc|} \hline & \\
\xymatrix{
\begin{array}{c}
\def\svgwidth{.3in} \tiny 
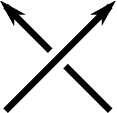
\end{array} \ar[r]^-{\Zh} & \begin{array}{c}\def\svgwidth{1.1in} \tiny 
%% Creator: Inkscape 1.0.2-2 (e86c870879, 2021-01-15), www.inkscape.org
%% Accompanies image file 'zh_pos.eps' (pdf, eps, ps)
\begingroup%
  \makeatletter%
  \providecommand\color[2][]{%
    \errmessage{(Inkscape) Color is used for the text in Inkscape, but the package 'color.sty' is not loaded}%
    \renewcommand\color[2][]{}%
  }%
  \providecommand\transparent[1]{%
    \errmessage{(Inkscape) Transparency is used (non-zero) for the text in Inkscape, but the package 'transparent.sty' is not loaded}%
    \renewcommand\transparent[1]{}%
  }%
  \providecommand\rotatebox[2]{#2}%
  \newcommand*\fsize{\dimexpr\f@size pt\relax}%
  \newcommand*\lineheight[1]{\fontsize{\fsize}{#1\fsize}\selectfont}%
  \ifx\svgwidth\undefined%
    \setlength{\unitlength}{236.86404916bp}%
    \ifx\svgscale\undefined%
      \relax%
    \else%
      \setlength{\unitlength}{\unitlength * \real{\svgscale}}%
    \fi%
  \else%
    \setlength{\unitlength}{\svgwidth}%
  \fi%
  \global\let\svgwidth\undefined%
  \global\let\svgscale\undefined%
  \makeatother%
  \begin{picture}(1,0.6607376)%
    \lineheight{1}%
    \setlength\tabcolsep{0pt}%
    \put(0,0){\includegraphics[width=\unitlength]{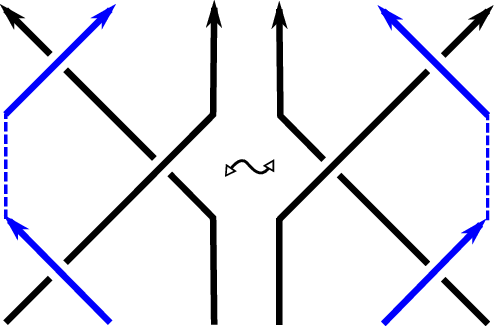}}%
    \put(0.44820787,0.34479521){\color[rgb]{0,0,0}\makebox(0,0)[lt]{\lineheight{40.54999924}\smash{\begin{tabular}[t]{l}$sw$\end{tabular}}}}%
  \end{picture}%
\endgroup%
 \end{array}} & \xymatrix{ \begin{array}{c}
\def\svgwidth{.3in} \tiny 
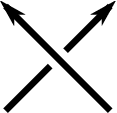
\end{array} \ar[r]^-{\Zh} & \begin{array}{c}\def\svgwidth{1.1in} \tiny 
%% Creator: Inkscape 1.0.2-2 (e86c870879, 2021-01-15), www.inkscape.org
%% Accompanies image file 'zh_neg.eps' (pdf, eps, ps)
\begingroup%
  \makeatletter%
  \providecommand\color[2][]{%
    \errmessage{(Inkscape) Color is used for the text in Inkscape, but the package 'color.sty' is not loaded}%
    \renewcommand\color[2][]{}%
  }%
  \providecommand\transparent[1]{%
    \errmessage{(Inkscape) Transparency is used (non-zero) for the text in Inkscape, but the package 'transparent.sty' is not loaded}%
    \renewcommand\transparent[1]{}%
  }%
  \providecommand\rotatebox[2]{#2}%
  \newcommand*\fsize{\dimexpr\f@size pt\relax}%
  \newcommand*\lineheight[1]{\fontsize{\fsize}{#1\fsize}\selectfont}%
  \ifx\svgwidth\undefined%
    \setlength{\unitlength}{239.48080902bp}%
    \ifx\svgscale\undefined%
      \relax%
    \else%
      \setlength{\unitlength}{\unitlength * \real{\svgscale}}%
    \fi%
  \else%
    \setlength{\unitlength}{\svgwidth}%
  \fi%
  \global\let\svgwidth\undefined%
  \global\let\svgscale\undefined%
  \makeatother%
  \begin{picture}(1,0.6515535)%
    \lineheight{1}%
    \setlength\tabcolsep{0pt}%
    \put(0,0){\includegraphics[width=\unitlength]{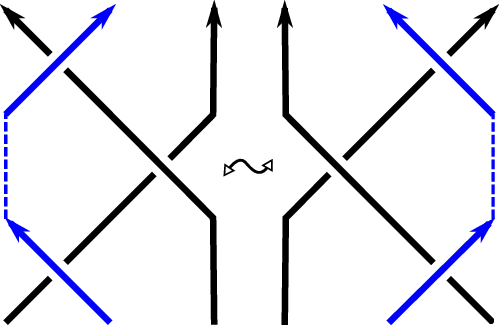}}%
    \put(0.44978272,0.34246882){\color[rgb]{0,0,0}\makebox(0,0)[lt]{\lineheight{40.54999924}\smash{\begin{tabular}[t]{l}$sw$\end{tabular}}}}%
  \end{picture}%
\endgroup%
 \end{array}
} \\ \underline{$\Zh$ for $\oplus$-crossings} & \underline{$\Zh$ for $\ominus$-crossings} \\ & \\ \hline
\end{tabular}
\]
\caption{The Bar-Natan $\Zh$-construction.} \label{fig_zh_construction}
\end{figure}

\begin{figure}[htb]
\[
\begin{tabular}{|ccc|} \hline
 & & \\
& \xymatrix{ \begin{array}{c}
\def\svgwidth{1in} \tiny 
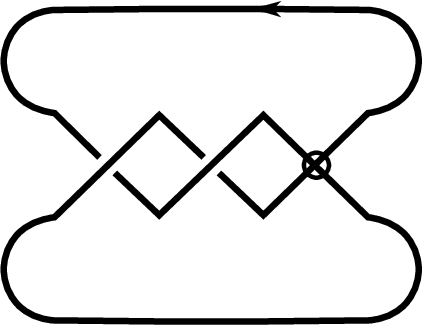
\end{array} \ar[r]^-{\Zh} & \begin{array}{c}\def\svgwidth{2.3in} \tiny 
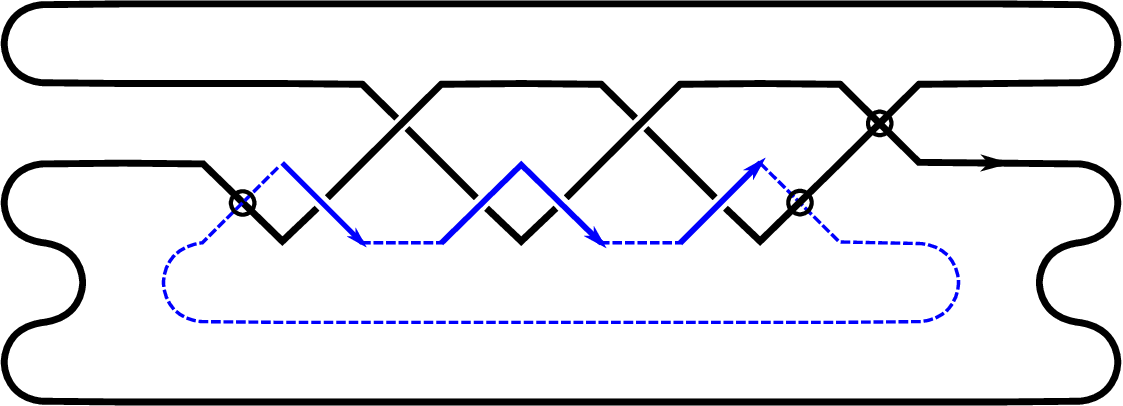 \end{array}
} & \\
& & \\ \hline
\end{tabular}
\]
\caption{Constructing $\Zh(K)$ for $K$ the left virtual trefoil.} \label{fig_zh_left_tref}
\end{figure}

The $\Zh$-construction associates to each $r$-component virtual link diagram $L$ an $(r+1)$-component semi-welded link diagram $L \cup \omega$ as follows. At each classical crossing of $L$, draw two small oriented blue arcs of $\omega$ as shown in Figure \ref{fig_zh_construction}. Then arbitrarily connect the small blue arcs together into a single component. If new crossings are thereby created, mark them as virtual. This arbitrary nature of connecting the small arcs together is emphasized in figures with dashed blue lines. The new virtual link diagram is $\Zh(L)$. Any two ways of forming the $\omega$-component are semi-welded equivalent, and hence the $\Zh$-construction is well-defined (see \cite{boden_chrisman_21,chrisman_22}). Furthermore, we have:

\begin{theorem}[Bar-Natan \cite{bar_natan_talk}] \label{thm_bar_natan_zh} If $L_1 \squigarrowleftright L_2$, then $\Zh(L_1) \stackrel{sw}{\squigarrowleftright} \Zh(L_2)$.
\end{theorem}

Alternatively, the $\Zh$-construction can be defined in terms of almost classical links. Suppose an $r$-component virtual link diagram $L$ is a sublink of an $(r+1)$-component link $L \cup \gamma$ such that the only classical crossings of $\gamma$ with $L$ are over-crossings by $\gamma$. Denote by $\gamma^{op}$ the component $\gamma$ with its orientation reversed. Then the $\Zh$-construction is characterized by the following property: if $L \cup \gamma^{\text{op}}$ is almost classical, then $L \cup \gamma\stackrel{sw}{\squigarrowleftright}\Zh(L)$ (see Chrisman-Todd \cite{chrisman_todd_23}, Theorem A). In other words, after changing the orientation of the $\omega$-component, the $\Zh$-construction turns virtual links into almost classical ones and any other method of doing so is essentially the same as the $\Zh$-construction. An interesting consequence is that if $L$ is already almost classical, then the split link $L \sqcup \textcolor{blue}{\bigcirc}$ must be semi-welded equivalent to $\Zh(L)$ with $\omega=\textcolor{blue}{\bigcirc}$:

\begin{theorem}[\cite{boden_chrisman_21,chrisman_todd_23}] \label{thm_zh_of_ac_splits} If $L$ is almost classical, $\Zh(L)=L \cup \omega$ is semi-welded equivalent to the split diagram $ L \sqcup \textcolor{blue}{\bigcirc}$, where the $\omega$-component is unknotted. 
\end{theorem}

Since the $\Zh$-construction is defined on virtual link diagrams, it may also be applied in the framed, rotational, and framed rotational categories. The \emph{framed semi-welded equivalence relation} is generated by extended Reidemeister moves and the semi-welded move, where the move $R_1^{fr}$ replaces $R_1$ for the non-$\omega$-components only. The \emph{rotational semi-welded equivalence relation} is generated by extended Reidemeister moves and the semi-welded move, where the move $\mathit{VR}_1^{fr}$ replaces $\mathit{VR}_1$ for the non-$\omega$-components only. Similarly, \emph{framed rotational semi-welded equivalence} replaces $R_1, \mathit{VR}_1$ with $R_1^{fr}, \mathit{VR}_1^{fr}$ respectively. The proof that the $\Zh$ construction is well-defined in each of these cases mimics that of the usual semi-welded case (see \cite{boden_chrisman_21,chrisman_22}). Furthermore, if $L_1,L_2$ are equivalent in the framed, rotational, or framed rotational categories, then $\Zh(L_1),\Zh(L_2)$ are equivalent in the corresponding semi-welded category. Details are left as exercises for the reader.

\subsection{The category of semi-welded tangles} \label{sec_cats_defn} To define the extended $U_q(\mathfrak{gl}(m|n))$ functor, it is necessary to define a category of semi-welded tangles to act as the target of the $\Zh$ functor. This is the the aim of the present section. 

First, we describe in more detail the categorical structure for virtual tangles. The category itself will be denoted $\mathcal{VT}$. The objects of $\mathcal{VT}$, are finite sequences from the set of symbols $\{\boxplus,\boxminus\}$. The empty sequence $\varnothing$ of symbols is also an object. The morphisms of $\mathcal{VT}$ correspond to virtual tangles as follows. Suppose $T$ is a virtual tangle diagram with $a$ free ends along its bottom and $b$ free ends along its top. Assume these endpoints are labeled left-to-right as $1,\ldots,a$ and $1,\ldots,b$, respectively. Write $\text{dom}(T)=(\varepsilon_1,\ldots,\varepsilon_a)$ and $\text{cod}(T)=(\eta_1,\ldots,\eta_b)$. Set $\varepsilon_i=\boxplus,\boxminus$ according to whether the arc incident to the $i$-th bottom endpoint is oriented away from or towards the endpoint, respectively. The sign $\eta_i$ is defined oppositely, so that $\eta_i=\boxplus$ if the incident arc terminates at the endpoint and $\eta_i=\boxminus$ if it emanates from the endpoint. Then $T$ is a morphism $T:\text{dom}(T) \to \text{cod}(T)$. For sign sequences $s_1,s_2$, $\text{Hom}_{\mathcal{VT}}(s_1,s_2)$ denotes the set of virtual tangles $T$ with $\text{dom}(T)=s_1,\text{cod}(T)=s_2$, considered equivalent up to virtual isotopy. Hence, $\text{Hom}_{\mathcal{VT}}(\varnothing,\varnothing)$ is in one-to-one correspondence with the set of virtual link types. Composition of virtual tangles (i.e. vertical stacking) defines a composition on $\mathcal{VT}$. The identity map, $1 \in \text{Hom}_{\mathcal{VT}}((\varepsilon_1,\ldots,\varepsilon_a),(\varepsilon_1,\ldots,\varepsilon_a))$ is the virtual tangle $\uparrow^{\varepsilon_1}
\cdots\uparrow^{\varepsilon_2} \uparrow^{\varepsilon_a}$ where $\uparrow^{\boxplus}=\uparrow$ and $\uparrow^{\boxminus}=\downarrow$. Categories $\mathcal{VT}^{\mathit{fr}}$ of framed virtual tangles,  $\mathcal{VT}^{\mathit{rot}}$ of rotational virtual tangles, and $\mathcal{VT}^{\mathit{fr}, \mathit{rot}}$ of framed rotational tangles are similarly defined.

\begin{figure}[htb]
\begin{tabular}{|ccc|} \hline
& & \\
\begin{tabular}{c}\def\svgwidth{1.1in} \tiny
%% Creator: Inkscape 1.0.2-2 (e86c870879, 2021-01-15), www.inkscape.org
%% Accompanies image file 'semi_welded_tangle_1.eps' (pdf, eps, ps)
\begingroup%
  \makeatletter%
  \providecommand\color[2][]{%
    \errmessage{(Inkscape) Color is used for the text in Inkscape, but the package 'color.sty' is not loaded}%
    \renewcommand\color[2][]{}%
  }%
  \providecommand\transparent[1]{%
    \errmessage{(Inkscape) Transparency is used (non-zero) for the text in Inkscape, but the package 'transparent.sty' is not loaded}%
    \renewcommand\transparent[1]{}%
  }%
  \providecommand\rotatebox[2]{#2}%
  \newcommand*\fsize{\dimexpr\f@size pt\relax}%
  \newcommand*\lineheight[1]{\fontsize{\fsize}{#1\fsize}\selectfont}%
  \ifx\svgwidth\undefined%
    \setlength{\unitlength}{238.4722414bp}%
    \ifx\svgscale\undefined%
      \relax%
    \else%
      \setlength{\unitlength}{\unitlength * \real{\svgscale}}%
    \fi%
  \else%
    \setlength{\unitlength}{\svgwidth}%
  \fi%
  \global\let\svgwidth\undefined%
  \global\let\svgscale\undefined%
  \makeatother%
  \begin{picture}(1,0.65733406)%
    \lineheight{1}%
    \setlength\tabcolsep{0pt}%
    \put(0,0){\includegraphics[width=\unitlength]{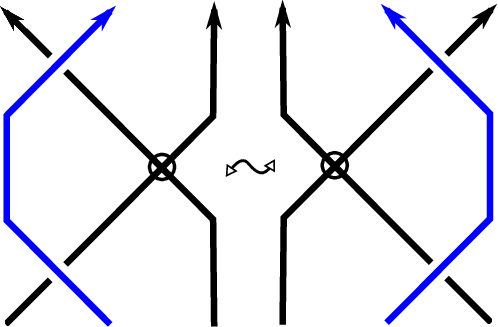}}%
    \put(0.44679041,0.34833479){\color[rgb]{0,0,0}\makebox(0,0)[lt]{\lineheight{40.54999924}\smash{\begin{tabular}[t]{l}$sw$\end{tabular}}}}%
  \end{picture}%
\endgroup%
 \\ \underline{$SW_1$}\end{tabular} & \begin{tabular}{c}\def\svgwidth{1.1in} \tiny
%% Creator: Inkscape 1.0.2-2 (e86c870879, 2021-01-15), www.inkscape.org
%% Accompanies image file 'semi_welded_tangle_2.eps' (pdf, eps, ps)
\begingroup%
  \makeatletter%
  \providecommand\color[2][]{%
    \errmessage{(Inkscape) Color is used for the text in Inkscape, but the package 'color.sty' is not loaded}%
    \renewcommand\color[2][]{}%
  }%
  \providecommand\transparent[1]{%
    \errmessage{(Inkscape) Transparency is used (non-zero) for the text in Inkscape, but the package 'transparent.sty' is not loaded}%
    \renewcommand\transparent[1]{}%
  }%
  \providecommand\rotatebox[2]{#2}%
  \newcommand*\fsize{\dimexpr\f@size pt\relax}%
  \newcommand*\lineheight[1]{\fontsize{\fsize}{#1\fsize}\selectfont}%
  \ifx\svgwidth\undefined%
    \setlength{\unitlength}{238.98495653bp}%
    \ifx\svgscale\undefined%
      \relax%
    \else%
      \setlength{\unitlength}{\unitlength * \real{\svgscale}}%
    \fi%
  \else%
    \setlength{\unitlength}{\svgwidth}%
  \fi%
  \global\let\svgwidth\undefined%
  \global\let\svgscale\undefined%
  \makeatother%
  \begin{picture}(1,0.65467103)%
    \lineheight{1}%
    \setlength\tabcolsep{0pt}%
    \put(0,0){\includegraphics[width=\unitlength]{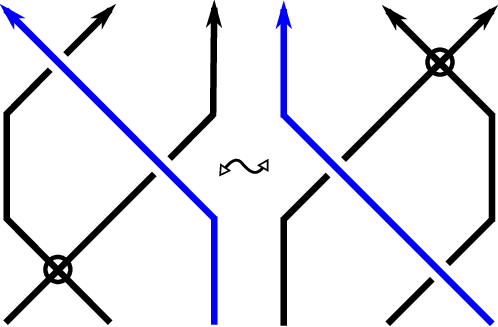}}%
    \put(0.43275433,0.34721751){\color[rgb]{0,0,0}\makebox(0,0)[lt]{\lineheight{40.54999924}\smash{\begin{tabular}[t]{l}$sw$\end{tabular}}}}%
  \end{picture}%
\endgroup%
 \\ \underline{$SW_2$}\end{tabular} & \begin{tabular}{c}\def\svgwidth{1.1in} \tiny
%% Creator: Inkscape 1.0.2-2 (e86c870879, 2021-01-15), www.inkscape.org
%% Accompanies image file 'semi_welded_tangle_3.eps' (pdf, eps, ps)
\begingroup%
  \makeatletter%
  \providecommand\color[2][]{%
    \errmessage{(Inkscape) Color is used for the text in Inkscape, but the package 'color.sty' is not loaded}%
    \renewcommand\color[2][]{}%
  }%
  \providecommand\transparent[1]{%
    \errmessage{(Inkscape) Transparency is used (non-zero) for the text in Inkscape, but the package 'transparent.sty' is not loaded}%
    \renewcommand\transparent[1]{}%
  }%
  \providecommand\rotatebox[2]{#2}%
  \newcommand*\fsize{\dimexpr\f@size pt\relax}%
  \newcommand*\lineheight[1]{\fontsize{\fsize}{#1\fsize}\selectfont}%
  \ifx\svgwidth\undefined%
    \setlength{\unitlength}{238.29727886bp}%
    \ifx\svgscale\undefined%
      \relax%
    \else%
      \setlength{\unitlength}{\unitlength * \real{\svgscale}}%
    \fi%
  \else%
    \setlength{\unitlength}{\svgwidth}%
  \fi%
  \global\let\svgwidth\undefined%
  \global\let\svgscale\undefined%
  \makeatother%
  \begin{picture}(1,0.65679735)%
    \lineheight{1}%
    \setlength\tabcolsep{0pt}%
    \put(0,0){\includegraphics[width=\unitlength]{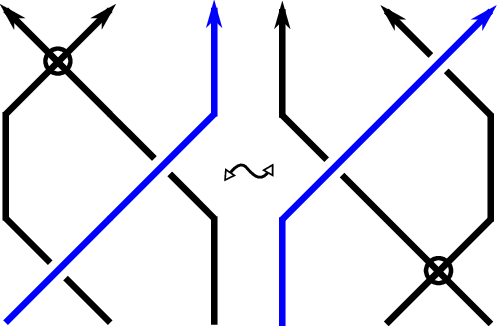}}%
    \put(0.44396233,0.33870451){\color[rgb]{0,0,0}\makebox(0,0)[lt]{\lineheight{40.54999924}\smash{\begin{tabular}[t]{l}$sw$\end{tabular}}}}%
  \end{picture}%
\endgroup%
 \\ \underline{$SW_3$} \end{tabular} \\
& & \\ \hline
\end{tabular}
\caption{Semi-welded moves for virtual tangle diagrams.} \label{fig_semi_welded_tangles}
\end{figure}
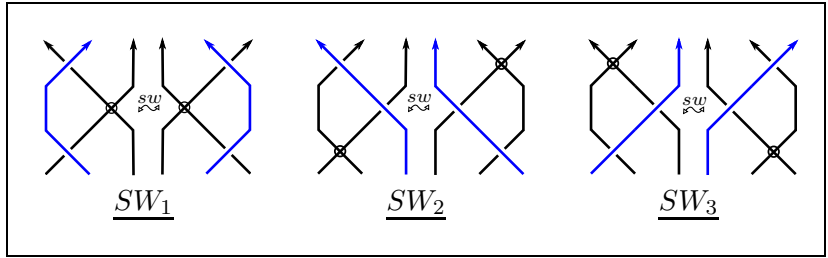

Now we are ready to define the category $\mathcal{SWT}$. The objects of $\mathcal{SWT}$ are sequences coming from the set of symbols $\{\boxplus,\boxminus ,\textcolor{blue}{\boxplus}, \textcolor{blue}{\boxminus}\}$. Note that the color of the symbols $\textcolor{blue}{\boxplus},\textcolor{blue}{\boxminus}$ has been chosen to match the color of the $\omega$-component. By a semi-welded tangle diagram, we mean a virtual tangle diagram $T$ where each of its components is colored blue or black. The set of blue colored components of $T$ will be called its \emph{$\omega$-part}. Likewise, the set of black colored components of $T$ will be called its \emph{$\alpha$-part}. For $T$ a semi-welded tangle diagram, $\text{dom}(T)$ and $\text{cod}(T)$ are defined as the domain and codomain in $\mathcal{VT}$ where each sign symbol matches the color of its incident arc. If $s_1,s_2$ are objects of $\mathcal{SWT}$, then $\text{Hom}_{\mathcal{SWT}}(s_1,s_2)$ is the set of semi-welded tangle diagrams $T$ satisfying $\text{dom}(T)=s_1$, $\text{cod}(T)=s_2$. These morphisms are considered equivalent up the moves $T_0-T_9$ (Figure \ref{fig_classical_tangle_moves}), $VT_1-VT_7$ (Figure \ref{fig_virtual_tangle_moves}) and $SW_1-SW_3$ (Figure \ref{fig_semi_welded_tangles})\footnote{\textbf{Exercise:} Show every semi-welded move from Figure \ref{fig_semi_welded_move} is a consequence of moves from Figures \ref{fig_classical_tangle_moves}, \ref{fig_virtual_tangle_moves}, and \ref{fig_semi_welded_tangles}.}. For moves $T_0-T_9$ and $VT_1-VT_7$, the strands may be either blue or black. For $SW_1-SW_3$, the colors of the strands must be as shown. Only consecutive crossings of blue over black component commute. Likewise, there is a rotational category $\mathcal{SWT}^{\mathit{rot}}$, a framed category $\mathcal{SWT}^{\mathit{fr}}$ and a framed rotational category $\mathcal{SWT}^{\mathit{fr},\mathit{rot}}$.

\section{Functors on Tangle Categories: Virtual, Semi-Welded, and\,$\Zh$} \label{prelim}

This section defines the pair $Q^{m|n}:\mathcal{VT}^{\mathit{fr},\mathit{rot}} \to \textbf{Vect}_{\mathbb{C}(q)}$, $\widetilde{Q}^{m|n}:\mathcal{SWT}^{\mathit{fr},\mathit{rot}} \to \textbf{Vect}_{\mathbb{C}(q,w)}$ of $U_q(\mathfrak{gl}(m|n))$ Reshetikhin-Turaev functors and that $\Zh$ functor. Section \ref{sec_glmn_defn} provides some elementary background on Lie superalgebras and the representation theory of $U_q(\mathfrak{gl}(m|n))$. The virtual $U_q(\mathfrak{gl}(m|n))$ functor $Q^{m|n}$ is defined using the vector representation of $U_q(\mathfrak{gl}(m|n))$. This is accomplished in Section \ref{vectrepmn}. The semi-welded $U_q(\mathfrak{gl}(m|n))$ functor is defined in Section \ref{sec_semi_welded_morphisms}. A proof of invariance under the virtual tangle moves and semi-welded tangle moves is sketched in Section \ref{virtualmovesinvariance}. The $\Zh$ functor is defined in Section \ref{sec_zh_functor}.

\subsection{The quantum supergroup $U_q(\mathfrak{gl}(m|n))$} \label{sec_glmn_defn}
A \emph{Lie superalgebra} $\mathfrak{g}$ is a $\mathbb{Z}_2$-graded vector space $\mathfrak{g}=\mathfrak{g}_{\bar{0}} \oplus \mathfrak{g}_{\bar{1}}$ with a bilinear map $[\cdot,\cdot]: \mathfrak{g} \times \mathfrak{g} \to \mathfrak{g}$, called the the \emph{Lie superbracket}, which satisfies the axioms below. Here, the grading of a homogeneous element $x$ of $\mathfrak{g}$ is denoted by $|x|\in \mathbb{Z}_2$.
\begin{eqnarray*}
    [x,y] &=& -(-1)^{|x||y|}[y,x]\\
    {[x,[y,z]]} &=& [[x,y],z] + (-1)^{|x||y|}[y,[x,z]\,]
\end{eqnarray*}
The subspaces $\mathfrak{g}_{\bar{0}}$ and $\mathfrak{g}_{\bar{1}}$ are called \textit{even} and \textit{odd},  respectively. In the physics literature, these are called the \emph{bosonic} and \emph{fermionic} subspaces, respectively. 

Consider a $\mathbb{Z}_2$-graded vector space $\mathbb{C}^{m|n}=\mathbb{C}^m \oplus \mathbb{C}^n$, where the even space $\mathbb{C}^m$ is $m$ dimensional and the odd space $\mathbb{C}^n$ is $n$ dimensional. Denote the endomorphisms of the superspace $\mathbb{C}^{m|n}$ by $\mathfrak{gl}(m|n)$. The even subspace of $\mathfrak{gl}(m|n)$ consists of block diagonal $(m\times m) \oplus (n \times n)$ matrices and the odd subspace consists of anti-block diagonal matrices of the following form:
\[
\left(\begin{array}{c|c}
    0_{m \times m} & B\\ \hline
    C & 0_{n\times n}
\end{array}\right).
\]
The superbracket operation for matrices in $\mathfrak{gl}(m|n)$ is given by the super commutation relation for homogeneous elements $X$ and $Y$: $[X,Y] = XY-(-1)^{|X||Y|}YX$, where again, $|Z|=0$ if $Z$ is even and $|Z|=1$ if $Z$ is odd. This extends by linearity to all of $\mathfrak{gl}(m|n)$.

% For a matrix $M= \begin{pmatrix} A & B \\ C & D \end{pmatrix}$, its \textit{supertrace} is defined to be $\str(M) = \tr A - \tr D$.

For any two $\mathbb{Z}_2$ graded vector spaces, $U$, $V$, their tensor product inherits the structure of a graded vector space. For $i \in \mathbb{Z}_2$, $W_i$ will denote the  $i^{th}$ graded subspace of the super vector space $W_i$. Then we define:
\[
(U \otimes V)_i = \bigoplus_{j+k = i} U_j \otimes V_k
\] 
The \textit{graded switch map}, $\tau: U\otimes V \to V\otimes U$ is defined by $\tau(x \otimes y)=(-1)^{|x||y|} y \otimes x$. Observe that the above definition implies $\tau$ is a grading-preserving map from $U \otimes V$ to $V \otimes U$.
 
The standard theory of deformation of the universal enveloping algebra (see e.g. Kassel \cite{kassel}) of a Lie algebra can be generalized to the case of Lie superalgebras. For a brief and enlightening exposition of this, see Zhang \cite{zhang_2, zhang}. We obtain in this way a quantized universal enveloping superalgebra $U_q(\mathfrak{gl}(m|n))$ over the field of rational functions $\mathbb{C}(q)$ with $q$ a formal invertible parameter. Furthermore, there is a quasi-triangular Hopf superalgebra structure on $H = U_q(\mathfrak{gl}(m|n))$. The \textit{universal R-matrix}, $R: H\otimes H \to H\otimes H$ is characterized by the following relations:
\[
(\Delta \otimes \id)R = R_{13}R_{23}, \hspace{.1in} (\id \otimes \Delta)R = R_{13}R_{12},
\]
\[
\tau \circ \Delta(h) = R\Delta(h)R^{-1}, \forall h\in H,
\]
where $R_{ij}\in H^{\otimes3}$ is defined by  
$R_{12} = R \otimes 1, R_{23} = 1 \otimes R$, and $R_{13} = (\id \otimes \tau)(R_{12})$.

A ($p|q$)-dimensional representation of $U_q(\mathfrak{gl}(m|n))$ is determined by an action of $U_q(\mathfrak{gl}(m|n))$ on a  $\mathbb{Z}_2$-graded vector space $W$ of dimension $p|q$. This action in turn gives a homomorphism from $U_q(\mathfrak{gl}(m|n))$ to $\End{(W)}$. The \textit{vector representation} (or \textit{defining representation}) is defined by the action of $\mathfrak{gl}(m|n)$ on $\mathbb{C}^{m|n}$. Henceforth, we give $\mathbb{C}^{m|n}$ the standard basis $\{x_i\}_{i=1}^{m+n}$, where $|x_i|=0$ for $1\leq i \leq m$, and $|x_i|=1$ for $m+1\leq i \leq m+n$. Throughout the text, the vector representation of $U_q(\mathfrak{gl}(m|n))$ for fixed values of $m$ and $n$ will be denoted by $V$. We also need the dual representation $V^*$ of $V$. Its $\mathbb{Z}_2$ grading is defined as follows. Set $V^*_{\bar{0}}$ to be the space of dual vectors that vanish on $V_{\bar{1}}$ and $V^*_{\bar{1}}$ to be the space of dual vectors that vanish on $V_{\bar{0}}$. 
% \[
% h \in U_q(\mathfrak{gl}(m|n)) \text{ acts on } f \in V^* \text{ by } x\cdot f(v)= f(S(x)v) \hspace{.1in} \forall v \in V
% \]

\subsection{The virtual $U_q(\mathfrak{gl}(m|n))$ functor $Q^{m|n}$} \label{vectrepmn}  The \emph{virtual $U_q(\mathfrak{gl}(m|n))$ functor}  $Q^{m|n}:\mathcal{VT}^{\mathit{rot},\mathit{fr}} \to \textbf{Vect}_{\mathbb{C}(q)}$ is defined as follows. As in the classical case, set $Q^{m|n}(\boxplus)=V$ and $Q^{m|n}(\boxminus)=V^*$. For an object $s=(\varepsilon_1,\ldots,\varepsilon_N)$, set $Q^{m|n}(s)=\bigotimes_{i=1}^N V^{\varepsilon_i}$, where $V^{\boxplus}=V$ and $V^{\boxminus}=V^*$. If $s=\varnothing$, define $Q^{m|n}(\varnothing)=\mathbb{C}(q)$. 

Having defined $Q^{m|n}$ on all objects, we proceed to the define it on the morphisms of $\mathcal{VT}^{\mathit{fr},\mathit{rot}}$. For the elementary virtual tangle $\uparrow$,  $Q^{m|n}(\uparrow)$ is the identity map $\text{id}:V \to V$ and likewise $Q^{m|n}(\downarrow)$ is the identity map $\text{id}:V^* \to V^*$. Now, let $E_1$, $E_2$ be elementary tangles and assume that $Q^{m|n}(E_1)$, $Q^{m|n}(E_2)$ are already defined. Then the tensor product $E_1 \otimes E_2$ is assigned the tensor product of matrices $Q^{m|n}(E_1) \otimes Q^{m|n}(E_2)$ and a composition $E_1 \circ E_2$ is assigned to the composition $Q^{m|n}(E_1) \circ Q^{m|n}(E_2)$. If $T$ is decomposed into compositions and tensor products of elementary virtual tangles, $Q^{m|n}(T)$ is mapped to the corresponding tensor product and composition of $\mathbb{C}(q)$-vector space homomorphisms. It therefore remains to define $Q^{m|n}$ on the non-trivial elementary tangles. For elementary tangles not involving the virtual crossing, we use the usual values of the $U_q(\mathfrak{gl}(m|n))$ Reshetikhin-Turaev functor with the vector representation. These are explicitly given in Queffelec \cite{queffelec_19}, and these are the values we will use in the present paper (with a few small typos corrected).  Let $R$ denote the $R$-matrix for the vector representation of the ribbon Hopf superalgebra $U_q(\mathfrak{gl}(m|n))$ (see Section \ref{sec_glmn_defn}). Setting $\check{R}=\tau \circ R,\check{R}^{-1}=R^{-1} \circ \tau$, Queffelec gives the following matrices at classical crossings:
    \begin{align*} \label{eqn_r_matrices}
Q^{m|n} \left(\begin{array}{c} \includegraphics[height=.35in]{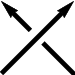} \end{array}\right) &= \check{R} \,\,\,\,\,: x_i \otimes x_j \to \left\{ \begin{array}{cl} 
q x_i \otimes x_i                                  & \text{if } i=j \le m      \\
(-1)^{|i||j|} x_j \otimes x_i+(q-q^{-1}) x_i \otimes x_j                           & \text{if } i<j         \\
(-1)^{|i||j|} x_j \otimes x_i & \text{if } i>j \\
-q^{-1} x_i \otimes x_i                                      & \text{if } i=j>m           
\end{array} \right. \\
Q^{m|n} \left(\begin{array}{c} \includegraphics[height=.35in]{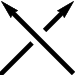} \end{array} \right)&=\check{R}^{-1}: x_i \otimes x_j \to \left\{ \begin{array}{cl} 
q^{-1} x_i \otimes x_i                                  & \text{if } i=j \le m      \\
(-1)^{|i||j|} x_j \otimes x_i                           & \text{if } i<j         \\
(-1)^{|i||j|} x_j \otimes x_i+(q^{-1}-q)x_i \otimes x_j & \text{if } i>j \\
-q x_i \otimes x_i                                      & \text{if } i=j>m           
\end{array} \right. 
\end{align*}
Cups are assigned in \cite{queffelec_19} to the following coevaluation maps $\mathbb{C}(q) \to V \otimes V^*$, $\mathbb{C}(q) \to V^* \otimes V$:
\begin{align*}
Q^{m|n}\left(\begin{array}{c}  \includegraphics[height=.25in]{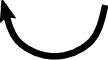} \end{array}\right) &: 1 \to \sum_{k=1}^{m+n} x_k \otimes x_k^*, \\  Q^{m|n}\left(\begin{array}{c}  \includegraphics[height=.25in]{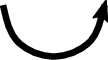} \end{array}\right) & : 1 \to q^{m-n} \left( \sum_{k=1}^{m} q^{1-2k} x_k^* \otimes x_k -\sum_{k=m+1}^{m+n} q^{-4m-1+2k} x_k^* \otimes x_k \right) 
\end{align*}
Caps are assigned in \cite{queffelec_19} to the following evaluation maps $V^* \otimes V \to \mathbb{C}(q)$, $V \otimes V^* \to \mathbb{C}(q)$: 
\[
Q^{m|n}\left(\begin{array}{c}  \includegraphics[height=.25in]{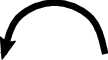} \end{array}\right) : x_k^* \otimes x_k \to 1, Q^{m|n}\left( \begin{array}{c} \includegraphics[height=.25in]{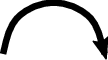} \end{array}\right) : x_k \otimes x_k^* \to \left\{\begin{array}{cl} q^{-m+n-1+2k} & \text{if } k \le m \\ -q^{3m+n+1-2k} & \text{if } k >m \end{array} \right.
\]

 So far, everything coincides by definition with the case of classical links. To define it for virtual links, the value of $Q^{m|n}$ at a virtual crossing must be determined. For this we will use a $q$-\emph{deformed graded switch map} $\tau_q^{m|n}:V \otimes V \to V \otimes V$. This is defined on the basis for $V \otimes V$ by:
\begin{align*} 
Q^{m|n}\left(\begin{array}{c} \includegraphics[height=.35in]{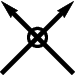} \end{array}\right) &= \tau_q^{m|n}\,\,: x_i \otimes x_j \to \left\{ \begin{array}{cl} x_j \otimes x_i & \text{if } i,j \le m \\ 
 q x_j \otimes x_i & \text{if } j \le m, i>m \\
 q^{-1} x_j \otimes x_i & \text{if } i \le m, j>m \\
-x_j \otimes x_i & \text{if } i,j > m \\
\end{array} \right.
\end{align*}
Section \ref{virtualmovesinvariance} ahead will show that, with this value at a virtual crossing, $Q^{m|n}$ satisfies all of the framed rotational tangle moves. As our definition is different that what occurs elsewhere in the virtual knot theory literature, we will conclude this subsection with some remarks explaining why this choice is a natural one for our setting.

\begin{remark} In the usual semisimple Lie algebra setting (see e.g. \cite{brochier, kauffman_rot, moltmaker2022vassiliev}), virtual crossings are assigned to the switch map $ \sigma: V \otimes V \to V \otimes V$, $\sigma(x_i \otimes x_j)=x_j \otimes x_i$. Observe that if $n=0$, so that $\dim(V)=m$,  we have $\tau^{m|0}_q=\sigma$. Hence we see that the $q$-deformation only appears in the Lie superalgebra case  whereas it is invisible for Lie algebras.
\end{remark}

\begin{remark} Since $U_q(\mathfrak{gl}(m|n))$ is not cocommutative, the graded switch map $\tau_q^{m|n}:V \otimes V \to V \otimes V$ is not a $U_q(\mathfrak{gl}(m|n))$-module homomorphism. Similarly, the switch map $\sigma$ is not a $U_q(\mathfrak{g})$-module homomorphism for $\mathfrak{g}$ a semisimple Lie algebra. This is why the target of the functor $Q^{m|n}$ is $\textbf{Vect}_{\mathbb{C}(q)}$ rather than $U_q(\mathfrak{gl}(m|n))-\textbf{Mod}$, as is the case of the classical $U_q(\mathfrak{gl}(m|n))$ Reshetikhin-Turaev functor. See also Edge \cite{edge2020skein} for a skein theoretic study virtual braiding. 
\end{remark}

\begin{remark} \label{rem_correct_COB} The $q$-deformation $\tau_q^{m|n}$ of $\tau$ is needed to recover the Alexander polynomial and generalized Alexander polynomial. It reflects a consistent change of basis that occurs when passing from the Burau representation to the virtual $U_q(\mathfrak{gl}(m|n))$ functor. The explicit change of basis will be given ahead in Section \ref{sec_alex_gen_alex}. If the graded switch $\tau$ is used instead, the virtual $U_q(\mathfrak{gl}(1|1))$ invariant of a 1-1 AC tangle does not always coincide with the Alexander polynomial. Similarly, $\tau_q^{1|1}$ is needed to ensure that the generalized $U_q(\mathfrak{gl}(1|1))$ polynomial coincides with the GAP.  
\end{remark}

\subsection{The semi-welded $U_q(\mathfrak{gl}(m|n))$ functor $\widetilde{Q}^{m|n}$} \label{sec_semi_welded_morphisms} Let $\mathbb{F}=\mathbb{C}(q,w)$ be the field of rational functions in two invertible variables $q,w$ over $\mathbb{C}$. The vector representation $V$ of $U_q(\mathfrak{gl}(m|n))$ defined in the previous section extends to a representation of vector spaces over $\mathbb{F}$, which will also be denoted by $V$. The \emph{semi-welded $U_q(\mathfrak{gl}(m|n))$ functor}  $\widetilde{Q}^{m|n}:\mathcal{SWT}^{\mathit{fr},\mathit{rot}} \to \textbf{Vect}_{\mathbb{F}}$ can now be defined as follows. First, it is defined on the objects and morphisms of $\mathcal{SWT}^{\mathit{fr},\mathit{rot}}$ having empty $\omega$-part. In this case, $\widetilde{Q}^{m|n}:= Q^{m|n}$. Next, $Q^{m|n}$ is extended to the blue objects $\textcolor{blue}{\boxplus,\boxminus}$ by $\widetilde{Q}^{m|n}(\textcolor{blue}{\boxplus})=\mathbb{F}$ and $\widetilde{Q}^{m|n}(\textcolor{blue}{\boxminus})=\mathbb{F}^* \cong \mathbb{F}$. In other words, the $\omega$-part is colored by the $1$-dimensional vector space $\mathbb{F}$. For $s=(\varepsilon_1,\ldots,\varepsilon_N)$, define $\widetilde{Q}^{m|n}(s)=\bigotimes_{i=1}^N V^{\varepsilon_i}$ where:
\[
V^{\varepsilon_i}=\left\{ \begin{array}{cl} V & \text{if } \varepsilon_i=\boxplus \\ V^* & \text{if } \varepsilon_i=\boxminus \\ \mathbb{F} & \text{if } \varepsilon=\textcolor{blue}{\boxplus,\boxminus} \end{array}  \right..
\]
For a classical crossing of the $\omega$-part over the $\alpha$-part, the assigned morphisms are:
\[
\widetilde{Q}^{m|n}\left(\begin{array}{c} \includegraphics[height=.35in]{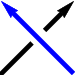} \end{array}\right)\,\, :x_i \to \left\{\begin{array}{cl} x_i & \text{if } i \le m \\ w^{-1} \cdot x_i & \text{if } i>m \end{array}\right., \,\,\widetilde{Q}^{m|n}\left(\begin{array}{c} \includegraphics[height=.35in]{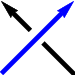} \end{array}\right)\,\, :x_i \to \left\{\begin{array}{cl} x_i & \text{if } i \le m \\ w \cdot x_i & \text{if } i>m \end{array}\right.. 
\]
The inverse map is assigned to a classical crossing of the $\alpha$-part over the $\omega$-part:
\[
\widetilde{Q}^{m|n}\left(\begin{array}{c} \includegraphics[height=.35in]{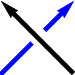} \end{array}\right)\,\, :x_i \to \left\{\begin{array}{cl} x_i & \text{if } i \le m \\ w \cdot x_i & \text{if } i>m \end{array}\right.,\quad \widetilde{Q}^{m|n}\left(\begin{array}{c} \includegraphics[height=.35in]{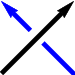} \end{array}\right) \,\,: x_i \to \left\{\begin{array}{cl}  x_i & \text{if } i \le m \\ w^{-1} \cdot x_i & \text{if } i>m \end{array}\right.. 
\]
Any virtual crossing involving the $\omega$-part is assigned an identity matrix:
\[ \widetilde{Q}^{m|n}\left(\begin{array}{c} \includegraphics[height=.35in]{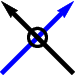} \end{array} \text{ or }\begin{array}{c} \includegraphics[height=.35in]{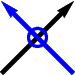} \end{array} \right) \,\,: x_i \to x_i, \quad \widetilde{Q}^{m|n}\left(
\begin{array}{c} \includegraphics[height=.35in]{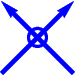} \end{array} \right)\,\, : 1 \to 1
\] 
For evaluation and coevaluation, note that $\mathbb{F} \otimes \mathbb{F}^* \cong \mathbb{F}$. Hence, these are are homomorphisms $\mathbb{F} \to \mathbb{F}$. Naturally, all such maps are assigned to the identity:    \[
\widetilde{Q}^{m|n}\left(\begin{array}{c} \includegraphics[height=.25in]{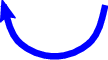} \end{array} \text{ or }  
\begin{array}{c} \includegraphics[height=.25in]{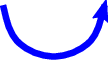} \end{array}\right): 1 \to 1 \otimes 1, \quad \widetilde{Q}^{m|n}\left( \begin{array}{c} \includegraphics[height=.25in]{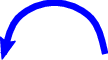} \end{array} \text{ or } 
\begin{array}{c} \includegraphics[height=.25in]{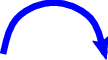} \end{array} \right): 1 \otimes 1 \to 1,
\]
Now that the virtual and semi-welded $U_q(\mathfrak{gl}(m|n))$ functors are defined, we proceed to showing that they are invariant under the appropriate sets of moves on virtual tangle diagrams.

%%%%%%%%%%%%%%%%%%%%%%%%%%%%%%%%%%%%%%%%%%%%%%%%%%%

\subsection{Proof of invariance of $Q^{m|n}$ and $\widetilde{Q}^{m|n}$}\label{geninvariantglmn} \label{virtualmovesinvariance} Consider first the case of $Q^{m|n}$. Clearly, $Q^{m|n}$ is invariant under the classical framed tangle moves. Hence, we need only check the framed rotational moves. As the argument is standard, only proofs for the moves $VT_2$, $VT_4$, and $\mathit{VR}_1^{\mathit{rot}}$ are given.
\newline
\newline
\noindent \underline{$VT_2$:} The left-hand side of this move can be decomposed as the composition of two virtual crossings:
\[
Q^{m|n}\left(\begin{array}{c} \includegraphics[width=.2in]{ovcross.eps} \end{array} \circ \begin{array}{c} \includegraphics[width=.2in]{ovcross.eps} \end{array} \right)=Q^{m|n}\left(\begin{array}{c} \includegraphics[width=.2in]{ovcross.eps} \end{array} \right) \circ Q^{m|n}\left(\begin{array}{c} \includegraphics[width=.2in]{ovcross.eps} \end{array} \right) = \tau_q^{m|n} \circ \tau_q^{m|n}.
\]
It therefore suffices to show that $\tau_q^{m|n}\circ \tau_q^{m|n}= \text{id} \otimes \text{id}$. Let $x_i\otimes x_j \in V\otimes V$ be a basis element. If $i,j \le m$, we have $\tau_q^{m|n} \circ \tau_q^{m|n}(x_i \otimes x_j)=\tau_q^{m|n}(x_j \otimes x_i)=x_i \otimes x_j$. Similarly, if $i,j >m$, we have $\tau_q^{m|n} \circ \tau_q^{m|n}(x_i \otimes x_j)=\tau_q^{m|n}(-x_j \otimes x_i)=x_i \otimes x_j$. If $i \le m$ and $j >m$:
\[
\tau_q^{m|n} \circ \tau_q^{m|n}(x_i \otimes x_j)=\tau_q^{m|n}(q^{-1} \cdot x_j \otimes x_i)=q \cdot q^{-1} \cdot x_i \otimes x_j=x_i \otimes x_j.
\]
Similarly, for $1\leq j \leq m$ and $m+1\leq i \leq m+n$, we get the same composition of maps above after substituting $q^{-1}$ for $q$. Hence, it follows in all cases that $\tau_q^{m|n}\circ \tau_q^{m|n}= \text{id} \otimes \text{id}$. \hfill $\square$
\newline
\newline
\noindent \underline{$VT_4$:} From Figure \ref{fig_virtual_tangle_moves}, we have the following decomposition of the move $VT_4$:
\begin{align*} 
\underline{\text{L.H.S}:} \quad
\left(Q^{m|n}\left(\begin{array}{c}\includegraphics[width=.2in]{pluscross.eps} \end{array}\right)\otimes \text{id} \right)\circ \left(\text{id} \otimes Q^{m|n}\left(\begin{array}{c} 
\includegraphics[width=.2in]{ovcross.eps} \end{array} \right)\right) \circ \left(Q^{m|n}\left(\begin{array}{c} 
\includegraphics[width=.2in]{ovcross.eps} \end{array} \right) \otimes \text{id}\right) &= \\ 
\underline{\text{R.H.S}:} \quad \left(\text{id} \otimes Q^{m|n}\left(\begin{array}{c} 
\includegraphics[width=.2in]{ovcross.eps} \end{array} \right)\right) \circ \left(Q^{m|n}\left(\begin{array}{c} 
\includegraphics[width=.2in]{ovcross.eps} \end{array} \right) \otimes \text{id}\right) \circ \left(\text{id} \otimes Q^{m|n}\left(\begin{array}{c}\includegraphics[width=.2in]{pluscross.eps} \end{array}\right)\right). &
\end{align*}
Consider a basis vector $x_k \otimes x_i \otimes x_j \in V^{\otimes 3}$. It will be shown that $\text{L.H.S}=\text{R.H.S}$ on $x_k \otimes x_i \otimes x_j$. There are eight cases depending on whether each index $i,j,k$ is less than or equal to $m$ or greater than $m$. For the sake of brevity, only two of the non-trivial cases are shown below. The cases are:
\begin{enumerate}
    \item $k,j \leq m$, $m<i$, and
    \item $k,i \leq m$, $m<j$.
\end{enumerate}
\underline{Case (1):} On the left-hand side of the move, we have:
\[
x_k \otimes x_i \otimes x_j \xrightarrow{\tau_q^{m|n} \otimes \id} 
q^{-1} x_i \otimes x_k \otimes x_j \xrightarrow{\id \otimes \tau_q^{m|n}} q^{-1} x_i \otimes x_j \otimes x_k \xrightarrow{\check{R} \otimes \id} q^{-1} x_j\otimes x_i \otimes x_k.
\]
The right-hand side gives the same value, as is seen in the following string of compositions:
\[
  x_k \otimes x_i \otimes x_j \xrightarrow{\id \otimes \check{R}} x_k \otimes x_j \otimes x_i \xrightarrow{\tau_q^{m|n} \otimes \id} x_j \otimes x_k \otimes x_i \xrightarrow{\id \otimes \tau_q^{m|n}} q^{-1} x_j \otimes x_i \otimes x_k.   
\]
\underline{Case (2):} On the left-hand side of the move, we have:
\begin{align*}
    x_k \otimes x_i \otimes x_j &\xrightarrow{\tau_q^{m|n} \otimes \text{id}} x_i \otimes x_k \otimes x_j \\
    &\xrightarrow{\id \otimes \tau_q^{m|n} } q^{-1} x_i \otimes x_j \otimes x_k \\
    &\xrightarrow{\check{R} \otimes \id} q^{-1}(x_j \otimes x_i \otimes x_k+(q-q^{-1})x_i \otimes x_j \otimes x_k).
\end{align*}
On the right-hand side of the move, we get the same value:
\begin{align*}
    x_k \otimes x_i \otimes x_j & \xrightarrow{\id \otimes \check{R}} x_k \otimes x_j \otimes x_i+(q-q^{-1})x_k \otimes x_i \otimes x_j \\
    & \xrightarrow{\tau_q^{m|n} \otimes \id} q^{-1} x_j \otimes x_k \otimes x_i+(q-q^{-1})x_i \otimes x_k \otimes x_j \\
    & \xrightarrow{\id \otimes \tau_q^{m|n}} q^{-1} x_j \otimes x_i \otimes x_k+q^{-1}(q-q^{-1})x_i \otimes x_j \otimes x_k.
\end{align*}
Thus, in both Case (1) and Case (2), $\text{L.H.S}=\text{R.H.S}$. The remaining six cases are similar.  \hfill $\square$
\newline
\newline
\noindent \underline{$\mathit{VR}_1^{\mathit{rot}}$:} The left virtual curl decomposes as:
\[
\left(Q^{m|n}\left(\begin{array}{c}\includegraphics[width=.2in]{cappic_left.eps} \end{array} \right) \otimes \text{id}\right) \circ \left(\text{id} \otimes Q^{m|n}\left(\begin{array}{c} 
\includegraphics[width=.2in]{ovcross.eps} \end{array} \right) \right) \circ \left(Q^{m|n}\left(\begin{array}{c}\includegraphics[width=.2in]{cuppic_right.eps} \end{array} \right) \otimes \text{id}\right).
\]
The right virtual curl can be written as the following decomposition:
\[
\left(\text{id} \otimes Q^{m|n}\left(\begin{array}{c}\includegraphics[width=.2in]{cappic_right.eps} \end{array} \right)\right) \circ \left(Q^{m|n}\left(\begin{array}{c} 
\includegraphics[width=.2in]{ovcross.eps} \end{array} \right) \otimes \text{id} \right) \circ \left(\text{id} \otimes Q^{m|n}\left(\begin{array}{c}\includegraphics[width=.2in]{cuppic_left.eps} \end{array} \right) \right).
\]
The value of the left virtual curl on the basis element $x_i \in V$ with $i \le m$ is:
\begin{align*}
    x_i & \xrightarrow{Q^{m|n}(\raisebox{\depth}{\rotatebox{180}{$\curvearrowleft$}}) \otimes \id}  1 \to q^{m-n} \left( \sum_{k=1}^{m} q^{1-2k} x_k^* \otimes x_k \otimes x_i -\sum_{k=m+1}^{m+n} q^{-4m-1+2k} x_k^* \otimes x_k \otimes x_i \right) \\
    & \xrightarrow{\id \otimes \tau_q^{m|n}} q^{m-n} \left( \sum_{k=1}^{m} q^{1-2k} x_k^* \otimes x_i \otimes x_k -\sum_{k=m+1}^{m+n} q^{-4m-1+2k} \cdot q \cdot x_k^* \otimes x_i \otimes x_k \right) \\
    &  \xrightarrow{Q^{m|n}(\curvearrowleft) \otimes \id} q^{m-n} \cdot q^{1-2i} x_i.
\end{align*}
After making a similar argument for $i>m$, we obtain the full effect of a left virtual curl:
\[
x_i \to \left\{ \begin{array}{cl} q^{m-n+1-2i} \cdot x_i & \text{if } i \le m \\ q^{-n-3m-1+2i} \cdot x_i & \text{if } i>m \end{array} \right..
\]
For $i\le m$, the effect of a right virtual curl on $x_i$ is given by:
\begin{align*}
    x_i & \xrightarrow{\id \otimes Q^{m|n}(\raisebox{\depth}{\rotatebox{180}{$\curvearrowright$}})} \sum_{k=1}^m x_i \otimes x_k \otimes x_k^*+\sum_{k=m+1}^{m+n} x_i \otimes x_k \otimes x_k^* \\
    &\xrightarrow{\tau_q^{m|n}\otimes \id} \sum_{k=1}^m x_k \otimes x_i \otimes x_k^*+\sum_{k=m+1}^{m+n} q^{-1} \cdot x_k \otimes x_i \otimes x_k^* \\
    &\xrightarrow{\id \otimes Q^{m|n}(\curvearrowright)} q^{-m+n-1+2i} \cdot x_i. 
\end{align*}
The argument for $i>m$ is again similar. All together, the right virtual curl is:
\[
x_i \to \left\{ \begin{array}{cl} q^{-m+n-1+2i} \cdot x_i & \text{if } i \le m \\ q^{3m+n+1-2i} \cdot x_i & \text{if } i>m \end{array} \right..
\]
Hence, any composition of a left and right virtual curl yields the identity. \hfill $\square$
\newline
\newline
Now consider the semi-welded $U_q(\mathfrak{gl}(m|n))$ functor $\widetilde{Q}^{m|n}$. Since $Q^{m|n}$ was previously shown to be a framed rotational invariant, it is necessary to check each of the classical and virtual tangle moves having one or more blue component. Most of these can be quickly dispatched by inspection, since a blue arc either acts as the identity or multiplies the ``fermionic'' part by $w^{\pm 1}$. Hence we will content ourselves with proving only the move $\mathit{SW}_1$.
\newline
\newline
\noindent $\underline{\mathit{SW}_1:}$ The left-hand and right-hand sides of the move decompose as: 
\begin{align*} 
\underline{\text{L.H.S}:} \quad
\left(\widetilde{Q}^{m|n}\left(\begin{array}{c}\includegraphics[width=.2in]{omoverplus.eps} \end{array}\right)\otimes \text{id}_V \right)\circ \left(\text{id}_{\mathbb{F}} \otimes \widetilde{Q}^{m|n}\left(\begin{array}{c} 
\includegraphics[width=.2in]{ovcross.eps} \end{array} \right)\right) \circ \left(\widetilde{Q}^{m|n}\left(\begin{array}{c} 
\includegraphics[width=.2in]{omoverminus.eps} \end{array} \right) \otimes \text{id}_V\right) &= \\ 
\underline{\text{R.H.S}:} \quad \left(\text{id}_V \otimes \widetilde{Q}^{m|n}\left(\begin{array}{c} 
\includegraphics[width=.2in]{omoverminus.eps} \end{array} \right)\right) \circ \left(\widetilde{Q}^{m|n}\left(\begin{array}{c} 
\includegraphics[width=.2in]{ovcross.eps} \end{array} \right) \otimes \text{id}_{\mathbb{F}}\right) \circ \left(\text{id}_V \otimes \widetilde{Q}^{m|n}\left(\begin{array}{c}\includegraphics[width=.2in]{omoverplus.eps} \end{array}\right)\right). &
\end{align*}
Consider a basis element $x_i \otimes x_j$. There are four cases all together. The case of $i,j \le m$ and the case of $i, j>m$ are clear. Of the two remaining possibilities, we check only the case that $i\le m$ and $j>m$. First, on the left-hand side, we have:
\begin{align*}
    x_i \otimes x_j & \xrightarrow{\widetilde{Q}^{m|n}\left(\begin{array}{c} 
\includegraphics[width=.15in]{omoverminus.eps} \end{array} \right) \otimes \id_V} x_i \otimes x_j \\
& \xrightarrow{\id_{\mathbb{F}} \otimes \tau_q^{m|n}} q^{-1} \cdot x_j \otimes x_i \\
& \xrightarrow{\widetilde{Q}^{m|n}\left(\begin{array}{c}\includegraphics[width=.15in]{omoverplus.eps} \end{array}\right)\otimes \text{id}_V} w \cdot q^{-1} \cdot x_j \otimes x_i.
\end{align*}
The right-hand side of $\mathit{SW}_1$ is given by:
\begin{align*}
    x_i \otimes x_j & \xrightarrow{\id_V \otimes \widetilde{Q}^{m|n}\left(\begin{array}{c} 
\includegraphics[width=.15in]{omoverplus.eps} \end{array} \right)} w \cdot x_i \otimes x_j \\
& \xrightarrow{ \tau_q^{m|n} \otimes \id_{\mathbb{F}}} q^{-1} \cdot w \cdot x_j \otimes x_i \\
& \xrightarrow{\id_V \otimes \widetilde{Q}^{m|n}\left(\begin{array}{c}\includegraphics[width=.15in]{omoverplus.eps} \end{array}\right)} w \cdot q^{-1} \cdot x_j \otimes x_i.
\end{align*}
Hence, $\widetilde{Q}^{m|n}$ is invariant under the semi-welded move $\mathit{SW}_1$. \hfill $\square$
\newline

Above it was calculated that both the left and right virtual curls evaluate to diagonal matrices. This is true of any 1-1 virtual tangle, as we now show. Theorem \ref{thm_diagonal} will be used ahead in our proof that the virtual $U_q(\mathfrak{gl}(1|1))$ functor is the Alexander polynomial for almost classical 1-1 tangles.

\begin{theorem} \label{thm_diagonal}
    For any 1-1 virtual tangle $T$, $Q^{m|n}(T)$ and $\widetilde{Q}^{m|n}(T)$ are diagonal matrices.
\end{theorem}
\begin{proof} First we will prove the result for $Q^{m|n}$. Note that for an elementary virtual tangle $T$, each of the matrices $Q^{m|n}(T)$ given in Section \ref{vectrepmn} is a grading-preserving map. A composition of grading-preserving maps is grading-preserving. This implies that for any 1-1 virtual tangle, the matrix corresponding to $Q^{m|n}(T):V \to V$ is block diagonal. Now consider the effect of such a matrix $Q^{m|n}(T)$ on a basis vector $x_j$ of $V$. A virtual tangle diagram is a composition and tensor product of caps, cups, classical crossings, virtual crossings, and trivial 1-1 tangles. As $x_j$ moves through this composition, new basis elements will be added. However, the only type of map that can create new basis elements is a cup, which tensors with factors of the form $x_k \otimes x_k^*$ or $x_k^* \otimes x_k$. Furthermore, each cap annihilates a pair of the form $x_k \otimes x_k^*$ or $x_k^* \otimes x_k$ and multiplies by a scalar.  Due to this symmetry of creation and annihilation, there must always be one more $x_j$ than $x_j^*$ in the image at each stage of the composition. Now, since the virtual tangle diagram may be assumed to be a Morse function, the total number of caps must equal to total number of cups.  Hence, every $x_j^*$ created by a cup must later be annihilated by some cap. This leaves only a term $x_j$ at the end, so that the effect of $Q^{m|n}(T):V \to V$ is to multiply $x_j$ by a scalar. The argument for $\widetilde{Q}^{m|n}(T)$ is the same, as every elementary semi-welded tangle is again assigned to a grading-preserving map. 
\end{proof}

\subsection{The \!\!$\Zh$ functor} \label{sec_zh_functor} 
The next step in defining the extended $U_q(\mathfrak{gl}(m|n))$ functor is to generalize the Bar-Natan $\Zh$-construction to a functor between categories of tangles. Denote by $\mathcal{SWT}|_{\alpha}$ the full subcategory of $\mathcal{SWT}$ whose objects are sequences in the (black) symbols $\boxplus,\boxminus$. In other words, the objects of $\mathcal{SWT}|_{\alpha}$ are the objects of $\mathcal{VT}$ but the morphisms may have a non-empty $\omega$-part. In this case, the $\omega$-part consists of closed blue components. Note that composition and tensor product of semi-welded tangles remain well-defined in $\mathcal{SWT}|_{\alpha}$.  The $\Zh$ functor, defined below, will map into a quotient of the category $\mathcal{SWT}|_{\alpha}$. 

Let $\alpha:\mathcal{SWT} \to \mathcal{VT}$ be the functor that maps every object and morphism to its $\alpha$-part. More precisely, if $s$ is a sequence in $\{\boxplus,\boxminus,\textcolor{blue}{\boxplus},\textcolor{blue}{\boxminus}\}$, $\alpha(s)$ is the subsequence of $s$ obtained by deleting all the blue $\textcolor{blue}{\boxplus},\textcolor{blue}{\boxminus}$ terms. If $T$ is a semi-welded tangle, $\alpha(T)$ is the virtual tangle obtained by deleting all blue colored components of $T$. Similarly, there is a functor $\omega:\mathcal{SWT} \to \mathcal{SWT}$ that maps every object and morphism to its $\omega$-part. It is easy to see that $\alpha$ and $\omega$ are monoidal.

Both $\alpha$ and $\omega$ restrict to well-defined functors $\alpha:\mathcal{SWT}|_{\alpha} \to \mathcal{VT}$ and $\omega:\mathcal{SWT}|_{\alpha} \to \mathcal{SWT}|_{\alpha}$. In particular, $\omega$ maps every object $s$ of $\mathcal{SWT}|_{\alpha}$ to the empty sequence $\varnothing$ and every morphism $T:a \to b$ to $\omega(T):\varnothing \to \varnothing$. Thus, $\omega(T)$ consists of some number of \emph{closed} blue components.

Suppose that $a,b$ are objects of $\mathcal{SWT}|_{\alpha}$ and $S,T:a \to b$ are morphisms between them. Furthermore suppose that $\alpha(S)=\alpha(T)$ as diagrams. If the diagrams $\omega(S)$ and $\omega(T)$ have the same classical crossings with $\alpha(S)=\alpha(T)$ (up to planar isotopies), we will say that $S,T$ are \emph{$\omega$-crossing equivalent}. This means that after a planar isotopy, the classical crossings of the $\omega$-part of $S$ with $\alpha(S)=\alpha(T)$ match those of the $\omega$-part of $T$ with $\alpha(T)=\alpha(S)$. See Figure \ref{fig_omega_cross_equiv} for an example.

\begin{figure}[htb]
    \begin{tabular}{|ccccccc|} \hline & & & & & & \\
     & \includegraphics[width=.7in]{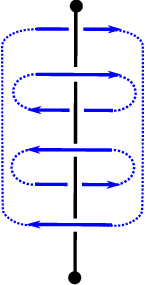} & &
    \includegraphics[width=.7in]{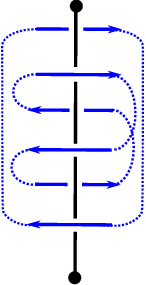} & &
    \includegraphics[width=.7in]{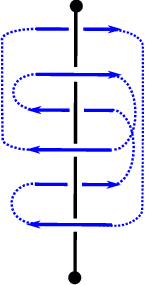} & \\ & & & & & & \\ \hline
    \end{tabular}
    \caption{Three $\omega$-crossing equivalent virtual tangle diagrams.}
    \label{fig_omega_cross_equiv}
\end{figure}

For each Hom-set of $\mathcal{SWT}|_{\alpha}$ let $\thicksim$ denote the smallest equivalence relation generated by the classical tangle moves $T_0-T_9$, the virtual tangle moves $VT_1-VT_7$, the semi-welded moves $SW_1-SW_3$, and $\omega$-crossing equivalence. Finally, let $\widetilde{\mathcal{SWT}}|_{\alpha}$ denote the quotient category $(\mathcal{SWT}|_{\alpha})/\!\thicksim$. 

\begin{definition}[The $\Zh$ Functor] The functor $\Zh:\mathcal{VT} \to \widetilde{\mathcal{SWT}}|_{\alpha}$ is defined as follows. If $a$ is an object of $\mathcal{VT}$, $\Zh(a)=a$. If $T$ is a virtual tangle, perform the $\Zh$-construction as in the virtual link case by flanking each classical crossing of $T$ with two over-crossings by $\omega$ as in Figure \ref{fig_zh_construction}. Then connect the blue arcs into a single component. New crossings created in connecting the over-crossing arcs of $\omega$ together are to be marked as virtual. Define $\Zh(T)$ to be the resulting semi-welded tangle. 
\end{definition}

\begin{theorem} \label{thm_zh_functor} The functor $\Zh:\mathcal{VT} \to \widetilde{\mathcal{SWT}}|_{\alpha}$ is well-defined and monoidal. In particular, for virtual tangles $S,T$, $\Zh(S \otimes T)=\!\!\Zh(S) \otimes \!\!\Zh(T)$ and $\Zh(S \circ T)=\!\!\Zh(S) \circ \!\!\Zh(T)$, if $S \circ T$ is defined.
\end{theorem}

\begin{proof} Well-definedness follows as in the case of virtual links \cite{boden_chrisman_21,chrisman_22}.  A summary is provided to emphasize that the argument does not require $R_1$ and $vR_1$ moves. Let $T$ be a virtual tangle and perform the first step of the $\Zh$ construction on $T$. Let $T_1=T \cup \omega_1,T_2=T \cup \omega_2$ be two ways of completing the $\Zh$-construction on $T$. First note that $\omega(T_1)=\omega_1$ and $\omega(T_2)=\omega_2$ consist of a single component. Thus, traversing $\omega(T_1)$ from a base point, one sees all the crossing of the $\omega$-part with the $\alpha$-part. Using $vR_2$ and $vR_3$ moves, we can assume that all the classical crossings along $\omega$ occur after the all the virtual crossings. Now, any two consecutive under-crossing arcs of the $\alpha$-part may be interchanged by first applying a $vR_2$ move and then applying a semi-welded move (see \cite{chrisman_22}, Figure 9). Since any two under-crossing arcs of the $\alpha$-part can be transposed, and any permutation is generated by transpositions, the ordering of the arcs of $\alpha(T_1)$ along $\omega(T_1)$ can be made to match the ordering of the arcs of $\alpha(T_2)$ along $\omega(T_2)$. Thus, after the virtual isotopy, $T_1$ and $T_2$ are $\omega$-crossing equivalent. This implies that $\Zh(T)$ is well-defined. That $\Zh$ maps equivalent virtual tangles to equivalent semi-welded tangles follows exactly as in \cite{boden_chrisman_21}. 

Now, let $S$ and $T$ be virtual tangles. Apply the first step of the $\Zh$-construction to $S$ and $T$. Next, use this initial placement of the blue over-crossing arcs to build $\Zh(S)$, $\Zh(T)$, and $\Zh(S \otimes T)$. Then $\Zh(S \otimes T)$ is $\omega$-crossing equivalent to $\Zh(S) \otimes \Zh(T)$, since they both have the same initial placement of the blue over-crossing arcs. Similarly, if $S \circ T$ is defined, then $\Zh(S \circ T)$ is $\omega$-crossing equivalent to $\Zh(S) \circ \Zh(T)$. Hence, as claimed, $\Zh(S \otimes T)=\Zh(S) \otimes \Zh(T)$ and $\Zh(S \circ T)=\Zh(S) \circ \Zh(T)$.
\end{proof}

The above definitions all apply to the framed and rotational categories $\mathcal{VT}^{\textit{fr}}$, $\mathcal{VT}^{\textit{rot}}$, $\mathcal{VT}^{\textit{fr,rot}}$. The same procedure gives categories $\widetilde{\mathcal{VT}}|_{\alpha}^{\textit{fr}}$, $\widetilde{\mathcal{VT}}|_{\alpha}^{\textit{rot}}$, and $\widetilde{\mathcal{VT}}|_{\alpha}^{\textit{fr,rot}}$. Furthermore, there are well-defined monoidal functors:
\begin{align*}
    \Zh &:\mathcal{VT}^{\textit{fr}} \to \widetilde{\mathcal{SWT}}|_{\alpha}^{\textit{fr}}, \\
    \Zh &:\mathcal{VT}^{\textit{rot}} \to \widetilde{\mathcal{SWT}}|_{\alpha}^{\textit{rot}}, \\
    \Zh &:\mathcal{VT}^{\textit{fr,rot}} \to \widetilde{\mathcal{SWT}}|_{\alpha}^{\textit{fr,rot}}.
\end{align*}
The proof that each of these functors is well-defined follows exactly as in the case of Theorem \ref{thm_zh_functor}. As no $R_1$ or $vR_1$ moves were used in the proof, no alterations to it are necessary.

\section{Almost classicality for virtual tangles $\&$ braids}

In Section \ref{sec_ac_tangles}, we generalize the notion of \emph{almost classical} from virtual links to virtual tangles. A characterization result (Theorem \ref{thm_zh_func_char}) for the $\Zh$-functor is given in Section \ref{sec_characterize}

\subsection{Almost classical tangles} \label{sec_ac_tangles}  We begin by extending the definition of almost classicality to $\mathcal{VT}$. Let $a,b$ be objects of $\mathcal{VT}$ and $T:a \to b$ a virtual tangle. By a \emph{short arc} of $T$, we mean a path on the diagram extending from one classical crossing of $T$ to the next or a path on the diagram from a string end of $T$ to the next classical crossing of $T$. If $\mathcal{S}(T)$ is the set of short arcs of $T$, an \emph{Alexander numbering} of $T$ is an integer labeling $\Gamma:\mathcal{S}(T) \to \mathbb{Z}$ of the arcs of $T$ that obeys the rule shown on the left in Figure \ref{fig_alex_numberings} at every classical crossing of $T$.  

Let $a=(\varepsilon_1,\ldots,\varepsilon_r)$ be an object of $\mathcal{VT}$ and $T:a \to a$ a virtual tangle diagram from $a$ to itself. For every term $\varepsilon_i$ in the sequence $a$, there is a short arc $\iota_i \in \mathcal{S}(T)$ which ends at $\varepsilon_i$ in the domain of $T$ and a short arc $\tau_i \in \mathcal{S}(T)$ which ends at $\varepsilon_i$ in the codomain of $T$. Now suppose that $T$ has an Alexander numbering $\Gamma:\mathcal{S}(T) \to \mathbb{Z}$. We will say that $\Gamma$ is \emph{conservative} if for $1 \le i \le r$, $\Gamma(\iota_i)=\Gamma(\tau_i)$. If $\Gamma$ is conservative, its \emph{potential} is the integer sequence $V(\Gamma)=(\Gamma(\iota_1),\ldots,\Gamma(\iota_r))=(\Gamma(\tau_1),\ldots,\Gamma(\tau_r))$. If $a=\varnothing$, so that $T$ is a virtual link diagram, an Alexander numbering $\Gamma$ of the tangle $T$ is an Alexander numbering of the virtual link $T$. In this case, we set $V(\Gamma)=0$. Note also that if a virtual tangle $T$ has a conservative Alexander numbering $\Gamma$, then for all $k \in \mathbb{Z}$, then $\Gamma+k$ is a conservative Alexander numbering for $T$.

\begin{definition}[Almost classical tangle]
Let $a \in \text{obj}(\mathcal{VT})$ and $T:a \to a$ a virtual tangle diagram. If $T$ has a conservative Alexander numbering, then $T$ will be called an \emph{almost classical diagram}. If $T$ is  equivalent to an almost classical diagram, it will be called an \emph{almost classical tangle}. Likewise, a tangle in $\mathcal{VT}^{\mathit{fr}}$ or $\mathcal{VT}^{\mathit{fr},\text{rot}}$ will be called \emph{almost classical} if it equivalent in its category to an almost classical diagram.
\end{definition}

\begin{example} \label{ex_conservative} Let $a=(\boxplus,\boxplus)$ and let $T:a \to a$ be the virtual $2$-braid $\chi_1 \sigma_1^{-2}$ (see Figure \ref{fig_conservative}, top). Then $T$ has an Alexander numbering, but it is not almost classical, as there is no conservative Alexander numbering. The long virtual knot $J:(\boxplus) \to (\boxplus$) (see Figure \ref{fig_conservative}, bottom) has a conservative Alexander numbering and, hence, is AC.
\end{example}

\begin{figure}[htb]
    \begin{tabular}{|ccc|}  \hline & & \\
    & \includegraphics[width=2in]{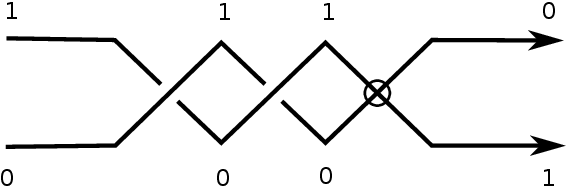} & \\ & & \\ \hline   & & \\ & \includegraphics[width=3.2in]{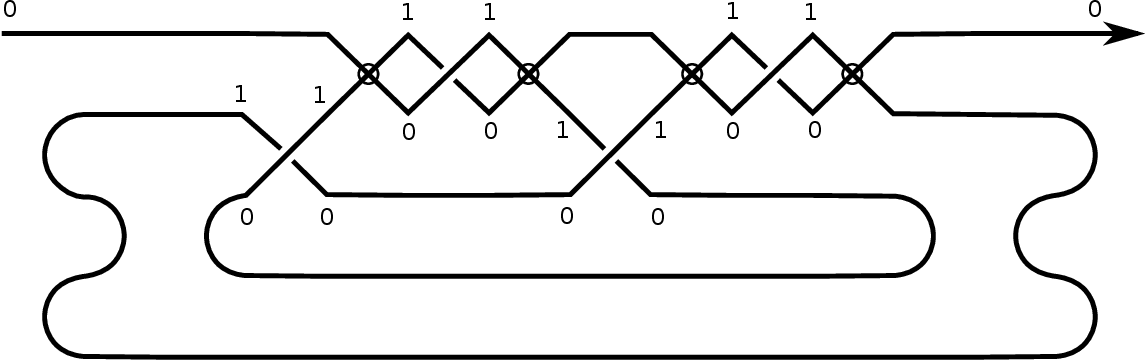} & \\ & & \\ \hline
    \end{tabular}
    \caption{$\circlearrowleft$ The top virtual tangle is Alexander numerable, but not AC. The bottom virtual tangle is AC as it has a conservative Alexander numbering.}   \label{fig_conservative}
\end{figure}

The next two results show that the definition of AC is compatible with the categorical structure of virtual tangles and with the definition of AC links given in Section \ref{sec_ac}. Below, $\widehat{T}$ denotes the right-hand top-bottom closure of $T:a \to a$ (analogous to Figure \ref{fig_closures}).

\begin{theorem} \label{thm_ac_makes_sense} If $T_1,T_2$ are AC tangles, then $T_1 \otimes T_2$ is AC. If $T_1,T_2:a \to a$ are AC tangle diagrams whose potentials differ by a constant vector, then $T_1 \circ T_2$ is AC.
\end{theorem}
\begin{proof} The first claim follows immediately from the definitions. For the second claim, suppose that $T_1:a \to a$ and $T_2:a \to a$ are diagrams of $T_1,T_2$ that have conservative Alexander numberings $\Gamma_1,\Gamma_2$, respectively, such that $V(\Gamma_2)-V(\Gamma_1)=(k,\ldots,k)$. Then $\Gamma_1+k$ is a an Alexander numbering of $T_1$ such that $V(\Gamma_2)=V(\Gamma_1+k)$. Each of the terminal arcs in $T_2$ has the same numbering as the corresponding initial arc in $T_1$, so that $T_1 \circ T_2$ has a conservative Alexander numbering. 
\end{proof}

\begin{theorem} A virtual tangle diagram $T:a \to a$ is an almost classical diagram if and only if $\widehat{T}$ is an almost classical link diagram.
\end{theorem}

\begin{proof} Observe that  the short arcs of $T$ and those of $\widehat T$ are in one-to-one correspondence, except if $T$ has are free ends. An Alexander numbering $\Gamma$ of $T$ will yield an Alexander numbering of $\widehat{T}$ if and only if $\Gamma$ matches on corresponding free ends. This occurs if and only if $\Gamma$ is conservative.
\end{proof}

\subsection{Characterization of the\,\,$\Zh$-functor} \label{sec_characterize} To analyze the effect of the $\Zh$ functor on AC tangles, we will generalize the characterization theorem of the $\Zh$-construction given in \cite{chrisman_todd_23}. First, we extend the notion of Alexander systems (see \cite{chrisman_todd_23}, Definition 3.1.1) to virtual tangles. Let $T:a \to a$ be a virtual tangle diagram. Suppose that $T$ is a sub-tangle of a tangle $T \cup \gamma$ where $\gamma$ is a single closed component. Every short arc of $T \cup \gamma$ lies in either $\gamma$ or in $T$. Denote by $S_{\gamma}(T)\subseteq \mathcal{S}(T \cup \gamma)$ the subset of short arcs of $T \cup \gamma$ contained $T$.

\begin{figure}[htb]
\begin{tabular}{|ccc|} \hline 
& & \\
\begin{tabular}{c} \includegraphics[scale=.5]{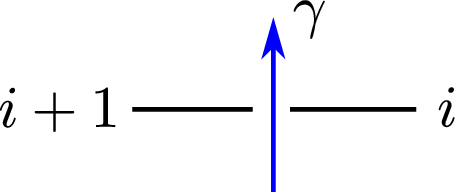} \end{tabular} & \begin{tabular}{c}
\includegraphics[scale=.4]{alex_numb_def_nice_II.eps} \end{tabular} & \begin{tabular}{c} \includegraphics[scale=.4]{alex_numb_def_nice.eps} \end{tabular} \\ & & \\\hline 
\end{tabular}
\caption{An Alexander sub-numbering $\Gamma:\mathcal{S}_{\gamma}(T) \to \mathbb{Z}$ obeys these rules at every classical crossing of $T \cup \gamma$.} \label{fig_alex_sub_numb}
\end{figure}

\begin{definition}[Alexander systems] \label{defn_alex_system} An \emph{Alexander system} for a virtual tangle $T:a \to a$ is a triple $(T, \gamma,\Gamma)$, where $T \cup \gamma$ is a virtual tangle diagram, $\gamma$ has exactly one closed component, and $\Gamma:\mathcal{S}_{\gamma}(T) \to \mathbb{Z}$ is an integer labeling of $S_{\gamma}(T)$, such that the following three conditions are satisfied. 
\begin{enumerate}
    \item For every classical crossing involving $\gamma$, the over-crossing arc lies in $\gamma$ and the under-crossing arcs lie in $T$. Otherwise, all crossings involving $\gamma$ are virtual. 
    \item The function $\Gamma$ obeys the rules shown in Figure \ref{fig_alex_sub_numb} at every classical crossing of $T \cup \gamma$. In this case, $\Gamma$ is called an \emph{Alexander sub-numbering}. 
    \item The Alexander sub-numbering $\Gamma$ is conservative (i.e. $\Gamma$ matches on corresponding free ends).
\end{enumerate}
 A conservative Alexander sub-numbering $\Gamma$ has a potential $V(\Gamma)$, which is defined analogously to potentials of Alexander numberings. Two Alexander systems $(T,\gamma_1,\Gamma_1)$, $(T,\gamma_2,\Gamma_2)$ for the same virtual tangle $T$ are said to be \emph{equivalent} if $V(\Gamma_1)=V(\Gamma_2)$ and $T \cup \gamma_1 \stackrel{sw}{\squigarrowleftright} T \cup \gamma_2$.
\end{definition}

If $s=\varnothing$, $T$ is a virtual link diagram. In this case, if $(T,\gamma,\Gamma)$ is an Alexander system for $T$, then $V(\Gamma)=0$ and $\Gamma$ is automatically conservative. In this case, Definition \ref{defn_alex_system} is identical to Definition 3.1.1 from \cite{chrisman_todd_23}. The main technical lemma for Alexander systems is the following.

\begin{lemma} \label{lemma_robbs_lemma} Any two Alexander systems $(T,\gamma_1,\Gamma_1)$, $(T,\gamma_2,\Gamma_2)$ for $T:a \to a$ having the same potential $V(\Gamma_1)=V(\Gamma_2)$ are equivalent. 
\end{lemma}
\begin{proof} The special case of $a=\varnothing$ is the content of \cite{chrisman_todd_23}, Lemmas 3.2.1, 3.2.2, 3.2.3. The argument in this case can be briefly summarized as follows. First, one makes $\Gamma_1$ and $\Gamma_2$ agree at each classical crossing of $T$ by using the $R_3$ move to insert loops of $\gamma_1$ around the classical crossing. Then given a short arc $\nu$ between two classical crossings, one can reduce the crossings of $\gamma_1$ along $\nu$ so that they all point in the same direction. By the definition of an Alexander sub-numbering (see Figure \ref{fig_alex_sub_numb}, left), it follows that $\Gamma_1$ changes monotonically along $\nu$. Now, $\Gamma_1$ and $\Gamma_2$ agree at the ends of each short arc. This means $\gamma_2$ can be positioned so that the sub-numbering changes in the same monotonic fashion along $\nu$ as $\gamma_1$. In particular, $\gamma_2$ must cross $\nu$ the same number of times and point in the same direction as $\gamma_1$. Performing this along every short arc between classical crossings, it follows that $T \cup \gamma_1$ and $T\cup \gamma_2$ are semi-welded equivalent. 

Now suppose $a \ne \varnothing$, so that $T$ has some free ends. Proceeding as before, one can ensure that $\Gamma_1$ and $\Gamma_2$ agree near every classical crossing of $T$. By hypothesis, $\Gamma_1$ and $\Gamma_2$ agree at the free ends. Applying the same argument as in the case that $a=\varnothing$, it can be ensured that $\gamma_1$ and $\gamma_2$ have the same classical crossings with $T$ along any short arc of the virtual tangle $T$. Hence, the given Alexander systems are equivalent.
\end{proof}

Denote by $\Zh^{\text{op}}(T)=T \cup \omega^{\text{op}}$ the semi-welded tangle obtained from $\Zh(T)$ by changing the direction of the $\omega$-component. This will be called the $\Zh^{\text{op}}$-construction.

\begin{figure}[htb]
\begin{tabular}{|cccc|} \hline & & & \\
& \includegraphics[scale=.4]{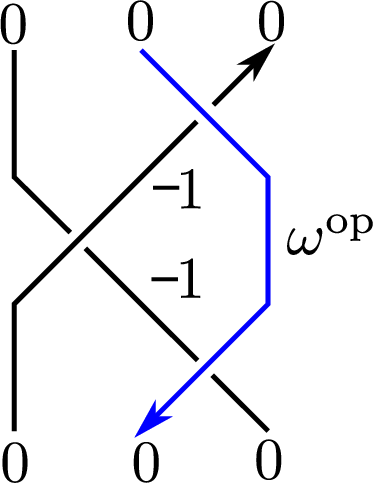} & 
\includegraphics[scale=.4]{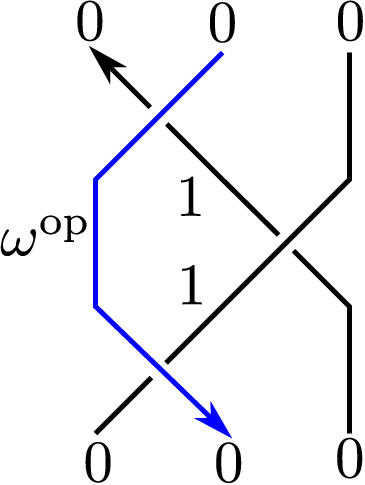} & \\ & & & \\ \hline 
\end{tabular}
\caption{The $\Zh^{\text{op}}$-constructions defines an Alexander sub-numbering $\Omega:\mathcal{S}_{\omega^{\text{op}}}(T) \to \mathbb{Z}$ having potential $V(\Omega)=(0,\ldots,0)$.} \label{fig_zh_is_conservative}
\end{figure}

\begin{lemma} \label{lemma_zh_op_has_potential_0} For any virtual tangle $T:a \to a$, $\Zh^{\text{op}}(T)$ defines an Alexander system $(T,\omega^{\text{op}}, \Omega)$, where $V(\Omega)=(0,\ldots,0)$ (or $V(\Omega)=0$ if $s=\varnothing$). Consequently, any Alexander system for $T$ with potential $(0,\ldots,0)$ is equivalent to the  $\Zh^{\text{op}}$-construction.
\end{lemma}
\begin{proof} Construct $\Zh^{\text{op}}(T)$ in the usual fashion. The short arcs near the classical crossings of $T \cup \omega^{\text{op}}$ can then be labeled as in Figure \ref{fig_zh_is_conservative}. Note that all the entering and exiting strands at each crossing are labeled $0$. Hence, it may be assumed that away from the classical crossings, all the short arcs may also be labeled $0$. In particular, all the initial and terminal ends of $T$ are labeled $0$. Thus, we have a conservative Alexander sub-numbering $\Omega:\mathcal{S}_{\omega^{\text{op}}}(T) \to \mathbb{Z}$ with potential $V(\Omega)=(0,\ldots,0)$. Then by Lemma \ref{lemma_robbs_lemma}, $(T,\omega^{\text{op}},\Omega)$ is equivalent to Alexander system having potential $(0,\ldots,0)$.
\end{proof}

\begin{figure}[htb]
\begin{tabular}{|c|} \hline \\
\xymatrix{ & & \begin{array}{c} \includegraphics[width=1.5in]{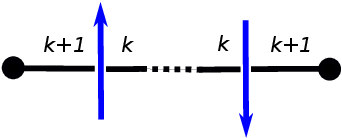} \end{array} \\
\begin{array}{c} \includegraphics[width=1.5in]{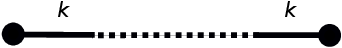} \end{array} \ar[urr]^-{\circlearrowleft \text{ loop at } \infty} \ar[rr]_-{\circlearrowright \text{ loop at } \infty} & & \begin{array}{c} \includegraphics[width=1.5in]{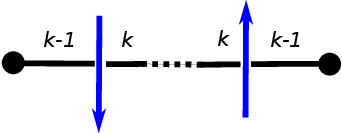} \end{array} \\
} \\ \hline
\end{tabular}
\caption{Changing the potential by adding loops at $\infty$.}
\label{fig_loops_infty}
\end{figure}

Now, given any vector $(v_1,\ldots,v_r) \in \mathbb{Z}^r$, there is an Alexander system $(T,\gamma,\Gamma)$ for $T$ whose potential $V(\Gamma)$ is $(v_1,\ldots,v_r)$. To see this, we describe the process of \emph{adding loops at $\infty$}. Begin with an arbitrary Alexander system $(T,\gamma_0,\Gamma_0)$ for $T$ and consider the potential $\Gamma(\iota_i)=g_i$ of the $i$-th initial arc. If $g_i=v_i$, there's nothing to do. If $g_i<v_i$, the potential of the $i$-th initial arc can be raised by adding two arcs to $\gamma_0$ as shown in the top branch of Figure \ref{fig_loops_infty}. Iterating this procedure, the $i$-th potential can be raised to $v_i$. Similarly, if $g_i>v_i$, the $i$-th potential can be lowered to $v_i$ by successively adding pairs of blue arcs to $\gamma$, as shown in the bottom branch of Figure \ref{fig_loops_infty}. Do this for all $i$, $1 \le i \le n$, and reconnect the blue over-crossing arcs into a single component $\gamma$. This gives an Alexander system $(T,\gamma,\Gamma)$ such that $V(\Gamma)=(v_1,\ldots,v_r)$. Viewing $\varepsilon_i$ as the point at $\infty$ for the $i$-th strand, each added pair of blue arcs can be interpreted as adding a clockwise ($\circlearrowright$ ) or counterclockwise ($\circlearrowleft$) loop orbiting around $\infty$. After adding a loop at infinity, one then reconnects the blue arcs into a single component.

The $\Zh$ functor can now be characterized in a manner parallel to that of the $\Zh$-construction. The theorem below states that the $\Zh^{\text{op}}$ functor makes virtual tangles of the form $T:a \to a$ almost classical and any other method of doing so by adding an over-crossing component is equivalent to the $\Zh^{\text{op}}$-construction, up to adding loops at $\infty$.

\begin{theorem} \label{thm_zh_func_char} Up to adding loops at $\infty$, every Alexander system for $T$ is equivalent to the $\Zh^{\text{op}}$-construction.
\end{theorem}

\begin{proof} Let $(T,\gamma_0,\Gamma_0)$ be an Alexander system. By adding loops at $\infty$, we obtain a new Alexander system $(T,\gamma,\Gamma)$ such that $V(\Gamma)=(0,\ldots,0)$. By Lemma \ref{lemma_zh_op_has_potential_0}, $(T,\gamma,\Gamma)$ is equivalent to the $\Zh^{\text{op}}$-construction. \end{proof}

\begin{corollary} \label{cor_AC_11_tangles_split} If $T:(\boxplus) \to (\boxplus)$ is an almost classical 1-1 tangle, then $\Zh(T)$ splits as $T \sqcup \textcolor{blue}{\bigcirc}$.     
\end{corollary}

\begin{proof} Since $\Zh$ maps equivalent virtual tangles to equivalent semi-welded tangles, we may assume without loss of generality that $T$ is already Alexander numerable. Let $\Gamma:\mathcal{S}(T) \to \mathbb{Z}$ be any Alexander numbering. Then $V(\Gamma)=(k)$ for some $k \in \mathbb{Z}$. Furthermore, $\Gamma-k$ is also an Alexander numbering of $T$ and $V(\Gamma-k)=(0)$. Let $\gamma=\textcolor{blue}{\bigcirc}$ be an unknotted blue circle disjoint from $T$ in the diagram. Clearly, $(T,\gamma, \Gamma)$ is an Alexander system for $T$. Since its potential is $(0)$, Lemma \ref{lemma_zh_op_has_potential_0} implies that $(T,\gamma,\Gamma)$ is equivalent to the $\Zh^{\text{op}}$-construction. Hence, $\Zh(T)$ splits as $T \sqcup \textcolor{blue}{\bigcirc}$. \end{proof}

\section{Invariants of virtual braids and rotational links} \label{sec_invariants}

This section defines the extended $U_q(\mathfrak{gl}(m|n))$ functor and applies it to virtual braids and links. Section \ref{sec_compose} discusses the composition of $\widetilde{Q}^{m|n} \circ \Zh$ and proves Theorem \ref{thm_A}. Representations of the virtual braid group are discussed in Section \ref{sec_braids}. Deframing the composition gives invariants of rotational links and tangles, as shown in Section \ref{sec_rot_invar}. Some technical details for computing the invariants via trace formulas are discussed in Section \ref{sec_trace}. 

\subsection{Proof of Theorem \ref{thm_A}} \label{sec_compose} First we show that $\widetilde{Q}^{m|n} \circ \Zh$ is well-defined.

\begin{lemma} \label{lemma_comp_is_ok} The functor $\widetilde{Q}^{m|n}:\mathcal{SWT}^{\text{fr},\text{rot}} \to \textbf{Vect}_{\mathbb{C}(q,w)}$ descends to a well-defined functor $\widetilde{Q}^{m|n}:\widetilde{\mathcal{SWT}}|_{\alpha}^{\text{fr},\text{rot}} \to \textbf{Vect}_{\mathbb{C}(q,w)}$. Hence, $\widetilde{Q}^{m|n} \circ \Zh$ is a well-defined monoidal functor.
\end{lemma}
\begin{proof} Let $a,b$ be objects of $\mathcal{SWT}|_{\alpha}^{\mathit{fr},\mathit{rot}}$ and let $T_1,T_2:a \to b$ be morphisms between them. By Section \ref{virtualmovesinvariance}, $\widetilde{Q}^{m|n}(T_1)=\widetilde{Q}^{m|n}(T_2)$ whenever $T_1,T_2$ are equivalent as framed rotational semi-welded tangles. Therefore, it suffices to show that $\widetilde{Q}^{m|n}(T_1)=\widetilde{Q}^{m|n}(T_1)$ whenever $T_1,T_2$ are $\omega$-crossing equivalent. Refer to the definition of $\widetilde{Q}^{m|n}$ in Section \ref{sec_semi_welded_morphisms}. Assume then that $T_1,T_2$ are given by tangle decompositions and that $T_1,T_2$ are $\omega$-crossing equivalent. Vertical blue strands are assigned to the $1$-dimensional vector space $\mathbb{C}(q,w)$. Hence, any tensor product or composition involving a vertical blue arc has no affect on the value of $\widetilde{Q}^{m|n}$. Similarly, blue cups and caps have no affect on the invariant. At a blue virtual crossing, we again have a $1$-dimensional identity map, so that tensoring or composing a blue virtual crossing has no effect. A virtual crossing having one blue and one black arc corresponds to the identity map on $V$. At a classical crossing of a blue are with another, we again obtain the identity map on $\mathbb{C}(q,w)$. Thus, the values of $\widetilde{Q}^{m|n}(T_1)$ and $\widetilde{Q}^{m|n}(T_2)$ are entirely determined by the $\alpha$-part of $T_1,T_2$ and the position of the classical crossings of their $\omega$-parts with their $\alpha$-parts. But since $T_1,T_2$ are $\omega$-crossing equivalent, this implies that $\widetilde{Q}^{m|n}(T_1)=\widetilde{Q}^{m|n}(T_2)$. \end{proof}

Finally, we may formally define the extended $U_q(\mathfrak{gl}(m|n))$ functor and prove Theorem \ref{thm_A}.

\begin{definition}[Extended $U_q(\mathfrak{gl}(m|n))$ functor]  The \emph{extended} $U_q(\mathfrak{gl}(m|n))$ \emph{functor} of framed rotational tangles is the composite functor $\widetilde{Q}^{m|n} \circ \Zh: \mathcal{VT}^{\mathit{fr},\mathit{rot}} \to \textbf{Vect}_{\mathbb{C}(q,w)}$. 
\end{definition}

\begin{theorem}[Theorem \ref{thm_A}] \label{thm_ext_cong_to_virt_on_AC} If $T:a \to a$ is an AC tangle diagram, then $\widetilde{Q}^{m|n}\circ \Zh(T)$ is a conjugate of $Q^{m|n}(T)$, with conjugation determined by a conservative Alexander numbering of $T$.   
\end{theorem}
\begin{proof} Write $a=(\varepsilon_1, \ldots, \varepsilon_r)$. Let $\Gamma:\mathcal{S}(T) \to \mathbb{Z}$ be a conservative Alexander numbering of $T$. Let $V(\Gamma)=(k_1,\ldots,k_r)$ be the potential of $\Gamma$. If $k_i >0$, add clockwise loops at $\infty$ (see Figure \ref{fig_loops_infty}) until the labels near $\varepsilon_i$ are reduced to $0$. If $k_i<0$, instead add counter-clockwise loops at $\infty$ until the labels near $\varepsilon_i$ are $0$. Now connect all these added blue over-crossing arcs to create a single blue component. Let $\omega^{\text{op}}$ denote this single component and let $\Omega:\mathcal{S}_{\omega^{\text{op}}}(T) \to \mathbb{Z}$ denote the Alexander sub-numbering we just constructed from $\Gamma$. Then $(T,\omega^{\text{op}},\Omega)$ is a conservative Alexander numbering with potential $V(\Omega)=(0,\ldots,0)$. By Lemma \ref{lemma_zh_op_has_potential_0}, $(T,\omega^{\text{op}},\Omega)$ is equivalent to the $\Zh^{\text{op}}$-construction. 

To calculate $\widetilde{Q}^{m|n} \circ \Zh(T)$, we need only evaluate $\widetilde{Q}^{m|n}$ on $T \cup (\omega^{\text{op}})^{\text{op}}=T \cup \omega$. Set $X$ to be the matrix assigned by $\widetilde{Q}^{m|n}$ to a positive over-crossing of a black arc by a blue arc. In block form, with $I_k$ the $k \times k$ identity matrix and $0$ an appropriately sized matrix of zeros, $X$ is given by:
\[
X=\left(\begin{array}{c|c} I_m & 0 \\ \hline 0 & w \cdot I_n \end{array}\right)
\]
Let $\epsilon_i=1,-1$ according to whether $\varepsilon_i=\boxplus,\boxminus$, respectively. It can be quickly verified that the effect of adding loops at $\infty$ for $\varepsilon_i$ is to compose with $X^{\epsilon_i k_i}$ and precompose with $X^{-\epsilon_i k_i}$. The cumulative effect is given by:
\[
\widetilde{Q}^{m|n}\circ \Zh(T)=\left( X^{\epsilon_n k_n} \otimes \cdots \otimes X^{\epsilon_1 k_1} \right) \circ Q^{m|n}(T) \circ \left( X^{-\epsilon_n k_n} \otimes \cdots \otimes X^{-\epsilon_1 k_1} \right)  
\]
Hence, $\widetilde{Q}^{m|n}\circ \Zh(T)$ is a conjugate of $Q^{m|n}(T)$ when $T$ has a conservative Alexander numbering.
\end{proof}

\subsection{Representations of the virtual braid group} \label{sec_braids} The representation is just $\widetilde{Q}^{m|n} \circ \Zh$:

\begin{theorem} For all possible $m,n,N$, the assignment $\widetilde{\rho}^{\,m|n}_N(\beta)=\widetilde{Q}^{m|n}\circ \Zh(\beta)$ is a representation $\widetilde{\rho}_N^{\,m|n}:\mathit{VB}_N \to GL((m+n)^N,\mathbb{C}(q,w))$ of the virtual braid group $\mathit{VB}_N$.
\end{theorem}
\begin{proof} By Theorem \ref{thm_zh_functor}, if $\beta_1,\beta_2 \in \mathit{VB}_N$, then $\Zh(\beta_1 \circ \beta_2)=\Zh(\beta_1) \circ \Zh(\beta_2)$. By Lemma \ref{lemma_comp_is_ok},  $\widetilde{\rho}^{\,m|n}_N(\beta_1 \circ \beta_2)=\widetilde{\rho}^{\,m|n}_N(\beta_1) \circ \widetilde{\rho}^{\,m|n}_N(\beta_2)$. Since virtual braid equivalence does not use $R_1$ or $vR_1$ moves, equivalent virtual braids are also equivalent as framed rotational tangles. Hence, $\widetilde{\rho}_N^{\,m|n}$ satisfies the virtual braid relations \ref{rel_vbr1}, \ref{rel_vbr2}, and \ref{rel_vbr3}.   
\end{proof}
 
It will be shown in Lemma \ref{lemma_exterior_for_gen_func} that, after passing to the exterior algebra, $\widetilde{\rho}_N^{\,1|1}$ is equivalent to the generalized Burau representation. Setting $w=1$ gives a representation that is equivalent to the Burau representation (again after passing to the exterior algebra). For $m,n \ne 1$, the representations are not in general the same as the generalized Burau representation. The representation $\widetilde{\rho}_N^{\,m|n}$ can be easily calculated on generators. For the reader's convenience, the calculation for a negative crossing $\sigma_1^{-1} \in \mathit{VB}_2$ is given below. The value at a positive crossing is the inverse of this matrix and the value at a virtual crossing is the $q$-deformed graded switch given in Section \ref{vectrepmn}. 

\[
\widetilde{\rho}^{\,\,m|n}_2(\sigma_1^{-1}):  \quad x_k \otimes x_l  \to \left\{ 
\begin{array}{cl}
q^{-1} \cdot x_l \otimes x_l & \text{if } l=k \le m \\
-q \cdot x_l \otimes x_l & \text{if } l=k > m \\
x_l \otimes x_k  & \text{if } k<l \le m \\
w \cdot x_l \otimes x_k & \text{if } k \le  m <l \\
-x_l \otimes x_k & \text{if } m < k <l \\
(-1)^{|l||k|} \cdot x_l \otimes x_k+(q^{-1}-q) \cdot x_k \otimes x_l & \text{if } \begin{array}{l} m<l<k \text{ or,} \\ l <k \le m \end{array} \\
w^{-1} \cdot x_l \otimes x_k+(q^{-1}-q) \cdot x_k \otimes x_l & \text{if } l \le  m <k \\
\end{array}
\right.
\]

\begin{example} \label{example_braid_not_AC} Let $\beta_1=\sigma_1\chi_1 \sigma_1^{-1} \in \mathit{VB}_2$. Clearly, $\widehat{\beta}_1=\widehat{\chi_1\sigma_1^{-1} \sigma_1}$, which is equivalent to the unknot in the virtual category. Although the closure is almost classical, $\beta_1$ itself is not almost classical. This can be proved using Theorem \ref{thm_ext_cong_to_virt_on_AC}. Suppose that $\beta_1$ is equivalent to a virtual braid $\beta_2$ having a conservative Alexander numbering $\Gamma$ with potential $V(\Gamma)=(k_1,k_2)$. Then $Q^{1|1}(\beta_1)=Q^{1|1}(\beta_2)$ and Theorem \ref{thm_ext_cong_to_virt_on_AC} gives:
\[
    \widetilde{Q}^{1|1}\circ \Zh(\beta_1) = \left(\begin{pmatrix}
      1 & 0 \\
      0 & w^{k_1}
    \end{pmatrix} \otimes \begin{pmatrix}
      1 & 0 \\
      0 & w^{k_2} 
    \end{pmatrix}\right) \circ Q^{1|1}(\beta_2) \circ \left(\begin{pmatrix}
      1 & 0 \\
      0 & w^{-k_1}
    \end{pmatrix} \otimes \begin{pmatrix}
      1 & 0 \\
      0 & w^{-k_2} 
    \end{pmatrix}\right).
\]    
Calculating the left and right hand side yields the equation:    
    \[
    \left(
\begin{array}{cccc}
 1 & 0 & 0 & 0 \\
 0 & q^2 w-w & -q^3+\frac{1}{q w^2}+2 q-\frac{1}{q} & 0 \\
 0 & q w^2 & w-q^2 w & 0 \\
 0 & 0 & 0 & -1 \\
\end{array}
\right)=
   \left(
\begin{array}{cccc}
 1 & 0 & 0 & 0 \\
 0 & q^2-1 & 2 q w^{k_2-k_1}-q^3 w^{k_2-k_1} & 0 \\
 0 & q w^{k_1-k_2} & 1-q^2 & 0 \\
 0 & 0 & 0 & -1 \\
\end{array}
\right). 
\]
As no choice of $k_1,k_2$ satisfies this equation, it follows that $\beta_1$ is not almost classical. $\hfill \square$ 
\end{example}

\subsection{Unframed rotational links} \label{sec_rot_invar} Polynomial invariants are obtained in the usual way. First we apply the extended $U_q(\mathfrak{gl}(m|n))$ invariants and then we adjust for the classical writhe.

\begin{definition}[Virtual $\&$ generalized $U_q(\mathfrak{gl}(m|n))$ polynomials] Let $L$ be a virtual link diagram. The \emph{virtual $U_q(\mathfrak{gl}(m|n))$ polynomial} $f_L^{\,m|n}(q)$ of $L$ and the \emph{generalized $U_q(\mathfrak{gl}(m|n))$ polynomial}  $\widetilde{f}^{\,m|n}_L(q,w)$ of $L$ are defined to be:
\begin{align*}
f^{m|n}_L(q) &:=q^{(n-m) \emph{wr}(L)} Q_{L}^{\,m|n}(q),  \\
\widetilde{f}^{\,m|n}_L(q,w) &:=q^{(n-m) \emph{wr}(L)} \widetilde{Q}_{\zh(L)}^{\,m|n}(q,w),
\end{align*} 
where $\emph{wr}(L)$ denotes the classical writhe of $L$. Note that $f^{m|n}_L(q)=\widetilde{f}^{\,m|n}_L(q,1)$, since setting $w=1$ assigns all elementary semi-welded tangles involving the $\omega$-part to an identity matrix. 
\end{definition}

\begin{theorem} If $L_1,L_2$ are rotationally equivalent virtual link diagrams, then: 
\[
f_{L_1}^{\,m|n}(q)=f_{L_2}^{\,m|n}(q), \text{ and } \widetilde{f}_{L_1}^{\,m|n}(q,w)=\widetilde{f}_{L_2}^{\,m|n}(q,w).
\]
In particular, $\widetilde{f}^{\,1|1}_L(q,w)$ is an invariant of virtual links, up to multiples of $q^{p}$ for $p \in \mathbb{Z}$.
\end{theorem}
\begin{proof} The $\Zh$ functor is invariant under all extended Reidemeister moves. The functors $Q^{m|n}$, $\widetilde{Q}^{m|n}$ are invariant under all moves except the $R1$ and $vR1$ moves on the $\alpha$-parts of semi-welded links.  Suppose then that $C$ is a $1$-$1$ virtual tangle corresponding to a classical curl. Then $\Zh(C)$ splits as $C \sqcup \textcolor{blue}{\bigcirc}$. Hence, the value on a classical curl agrees with the classical case: $\widetilde{Q}^{\,m|n}_{\zh(C)}(q,w)=Q^{m|n}_C(q)=q^{(m-n)\cdot \emph{wr}(C)}I_{m+n}$. If $m=n=1$ and $C$ is a virtual curl, then we have by Section \ref{virtualmovesinvariance} that $\widetilde{Q}^{\,1|1}_{\zh(C)}(q,w)=Q^{1|1}_C(q)=q^{\pm 1} \cdot I_2$. Here the sign of the exponent is $-1$ for a left virtual curl and $+1$ for the right virtual curl. This proves the second claim.
\end{proof}

 The next result shows $\widetilde{f}^{\,m_1|n_1}_L(q,w)$ and $\widetilde{f}^{\,m_2|n_2}_L(q,w)$ satisfy the same skein relation whenever $m_1-n_1=m_2-n_2$. However, as will be seen in Section \ref{sec_calc}, it is generally not true that $m_1-n_1=m_2-n_2$ implies $\widetilde{f}^{\,m_1|n_1}_L(q,w)$ and $\widetilde{f}^{\,m_2|n_2}_L(q,w)$ are equal. Below we follow the usual conventions: for a virtual link diagram $L=L_+$ with a fixed $\oplus$ crossing $x$ of $L$, let $L_-$ denote the virtual link diagram obtained by changing $x$ to an $\ominus$ crossing, and let $L_0$ the virtual link diagram obtained by applying the oriented smoothing at $x$. See Figure \ref{fig_skein}. For $z \in \mathbb{Z}$, write $[z]_q:=\dfrac{q^z-q^{-z}}{q-q^{-1}}$, for the quantum integer. 

\begin{figure}[htb]
\begin{tabular}{|ccccccc|} \hline & & & & & & \\
 & \includegraphics[width=1in]{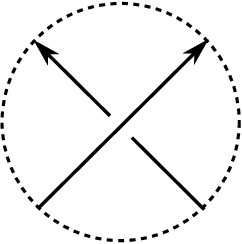} & & \includegraphics[width=1in]{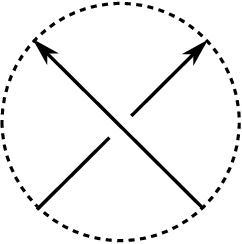} & & \includegraphics[width=1in]{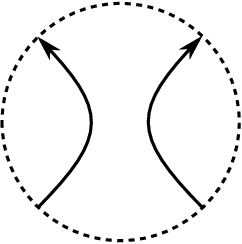} & \\ & & & & & & \\ & \underline{$L_+$:} & & \underline{$L_-$:} & & \underline{$L_0$:} & \\ & & & & & & \\ \hline
\end{tabular}
\caption{Local changes to a virtual link diagram $L$ in a skein relation.} \label{fig_skein}
\end{figure}

\begin{theorem} \label{thm_skein_relation} The invariants $f_L^{m|n}(q),\widetilde{f}_L^{\,m|n}(q,w)$ satisfy the following skein relations.
\begin{enumerate}
 \item $f_{\bigcirc}^{m|n}(q)=[m-n]_q$, $q^{m-n} f_{L_+}^{m|n}(q)-q^{n-m} f_{L_-}^{m|n}(q) = (q-q^{-1}) f_{L_0}^{m|n}(q)$.
 \newline
\item $\widetilde{f}_{\bigcirc}^{\, m|n}(q,w)=[m-n]_q$, $q^{m-n} \widetilde{f}_{L_+}^{\,m|n}(q,w)-q^{n-m} \widetilde{f}_{L_-}^{\,m|n}(q,w) = (q-q^{-1}) \widetilde{f}_{L_0}^{\,m|n}(q,w)$.
\end{enumerate}
\end{theorem}
\begin{proof} Since $f^{m|n}_L(q)$ is a specialization of $\widetilde{f}^{m|n}(q,w)$, it suffices to prove $(2)$. The value of an unknot can be calculated in the usual fashion. First decompose $\bigcirc$ as $\curvearrowright \circ \raisebox{\depth}{\rotatebox{180}{$\curvearrowright$}}$ and then multiply the cup and cap maps for $U_q(\mathfrak{gl}(m|n))$. Since $\Zh(\bigcirc)$ splits, it follows that $\widetilde{Q}_{\zh(\bigcirc)}^{m|n}(q,w)=Q_{\bigcirc}^{m|n}(q)=[m-n]_q$. This proves the first claim. Next, note that for the positive crossing $ \sigma_1 \in \mathit{VB}_2$, and negative crossing $ \sigma_1^{-1}$, we have that:
\[
\widetilde{Q}^{m|n}_{\sigma_1}(q,w)-\widetilde{Q}^{m|n}_{\sigma_1^{-1}}(q,w)=(q-q^{-1})\widetilde{Q}^{m|n}_{\uparrow \otimes \uparrow} (q,w).
\]
This can be seen by a direct calculation with the $R$-matrices for $U_q(\mathfrak{gl}(m|n))$ given in Section \ref{vectrepmn}.  Using the fact that $\emph{wr}(L_+)-2=\emph{wr}(L_-)$ and $\emph{wr}(L_-)+1=\emph{wr}(L_0)$, we have that:
\begin{align*}
    q^{2(m-n)} \widetilde{f}^{\,m|n}_{L_+}(q,w) &=q^{-(m-n)(wr(L_+)-2)} \widetilde{Q}_{\zh(L_+)}^{m|n}(q,w) \\
    &=q^{-(m-n)wr(L_-)} \left(\widetilde{Q}_{\zh(L_-)}^{m|n}(q,w)+(q-q^{-1})\widetilde{Q}_{\zh(L_0)}^{m|n}(q,w) \right) \\
    &= \widetilde{f}^{\,m|n}_{L_-}(q,w)+(q-q^{-1}) q^{-(m-n)(wr(L_0)-1)}\widetilde{Q}_{\zh(L_0)}^{m|n}(q,w) \\
    &=\widetilde{f}^{\,m|n}_{L_-}(q,w)+(q-q^{-1}) q^{m-n}\widetilde{f}^{\,m|n}_{L_0}(q,w).
\end{align*}
The proof is completed by multiplying the above equation through by $q^{n-m}$.
\end{proof}

The last result of this subsection shows how Theorem \ref{thm_A} specializes to the case of AC links. Here the rotational category is required whenever $(m,n) \ne (1,1)$. 

\begin{corollary} If $L$ is rotationally almost classical, then $\widetilde{f}^{\,\,m|n}_L(q,w)=f_L^{m|n}(q)$.
\end{corollary}
\begin{proof} Suppose that a diagram of $L$ is equivalent as a rotational link to an almost classical diagram $K$. Viewing $K$ as a virtual tangle diagram, $K$ is a morphism $\varnothing \to \varnothing$ that has a conservative Alexander numbering. By Theorem \ref{thm_ext_cong_to_virt_on_AC}, $\widetilde{Q}^{m|n}\circ \Zh(K)$ is conjugate to $Q^{m|n}(K)$. But since these are both $1 \times 1$ matrices, it must be that $\widetilde{Q}^{m|n}\circ \Zh(K)=Q^{m|n}(K)$.
\end{proof}

\subsection{Trace formulas} \label{sec_trace} As in the classical case, the generalized $U_q(\mathfrak{gl}(m|n))$ polynomial can be calculated by a trace formula. The trace formulas will typically result in a faster computation time when the braid width is large. First we recall the basic ingredients of the classical construction.  

For a fixed ribbon Hopf algebra $\mathfrak{A}$ and $\rho:\mathfrak{A} \to \text{End}(U)$ a finite-dimensional representation of $\mathfrak{A}$, let $Q$ denote the Reshetikhin-Turaev invariant associated to $\mathfrak{A}$. The $R$-matrix of $\mathfrak{A}$ and $\rho$ give a representation $\varphi:\mathit{Br}_N \to \text{End}(U^{\otimes N})$ of the braid group on $N$-strands. Then for any $\beta \in \mathit{Br}_N$, the value of $Q$ on the link diagram $\widehat{\beta}$ is given by:
\begin{align}\label{eqn_diagrammatic_trace}
Q(\widehat{\beta}) &=\text{tr} \left( \varphi(\beta) \circ \mu^{\otimes N}\right).
\end{align}
See, for example, \cite{jackson_moffatt}, Theorem 7.31. This is shown diagrammatically in Figure \ref{fig_diagrammatic_trace}. The map $\mu:U \to U$ is completely determined by the cup and cap maps of $Q$, and can be easily calculated from standard techniques (see \cite{jackson_moffatt}, Lemma B.4). For the vector representation of $U_q(\mathfrak{gl}(m|n))$, the $i,j$ entry $\mu_{ij}$ of $\mu:V \to V$ works out to be:
\begin{align} \label{eqn_mu_mn}
    \mu_{i j}&=\left\{\begin{array}{cl}  q^{-m+n-1+2i} & \text{if } i=j \le m \\ -q^{3m+n+1-2i} & \text{if } i=j>m \\ 0 & \text{else} \end{array} \right..
\end{align}
\begin{example} \label{ex_mu_11} For future use, we compute $\mu$ and $\mu^{\otimes N}$ for the vector representation $V$ of $U_q(\mathfrak{gl}(1|1))$. With respect to the basis $\{x_1,x_2\}$, $\mu$ is given by:
\[
 \begin{pmatrix}
    q & 0 \\ 0 & -q 
\end{pmatrix}
\]
Write $|i_j|$ for the grading $|x_{i_j}|$ of $x_{i_j}$. Then $N \ge 1$, $\mu^{\otimes N}:V^{\otimes N} \to V^{\otimes N}$ acts on the basis by:
 \[
 \mu^{\otimes N}(x_{i_1} \otimes \cdots \otimes x_{i_N})=q^N (-1)^{\sum_j |i_j|}x_{i_1} \otimes \cdots \otimes x_{i_N}
 \]
\end{example}

A sketch of the proof of Equation \ref{eqn_diagrammatic_trace} for classical braids will be useful for the virtual case. Let $\beta \in \mathit{Br}_N$ and write $\widehat{\beta}=\curvearrowright^N \circ (\beta \otimes \downarrow^N) \circ \raisebox{\depth}{\rotatebox{180}{$\curvearrowright$}}^N$, where $\square^N$ denotes $N$ parallel copies of $\square$. Then since $Q$ is a monoidal functor, we have:
\begin{align*}
    Q(\widehat{\beta}) & =Q(\curvearrowright^N) \circ Q(\beta \otimes \downarrow^N) \circ Q(\raisebox{\depth}{\rotatebox{180}{$\curvearrowright$}}^N) \\
    &=Q (\curvearrowright^N) \circ \left(Q(\beta) \otimes \text{id}^{\otimes N} \right)\circ Q(\raisebox{\depth}{\rotatebox{180}{$\curvearrowright$}}^N) \\
    &=Q (\curvearrowright^N) \circ \left(\varphi(\beta) \otimes \text{id}^{\otimes N} \right)\circ Q(\raisebox{\depth}{\rotatebox{180}{$\curvearrowright$}}^N)
\end{align*}
Then it is shown that $Q (\curvearrowright^N) \circ \left(- \otimes \text{id}^{\otimes N} \right)\circ Q(\raisebox{\depth}{\rotatebox{180}{$\curvearrowright$}}^N)$ coincides with $\text{tr}(- \otimes \mu^{\otimes N})$.

\begin{figure}[htb]
\[
\begin{tabular}{|ccc|} \hline
 & & \\
& \xymatrix{ \begin{array}{c}
\includegraphics[width=1in]{trace_full_1.eps}
\end{array} \ar[rr]^-{\varphi(-) \circ \mu^{\otimes N}} & & \begin{array}{c}\includegraphics[width=1.01in]{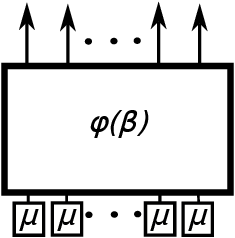} \end{array} \ar[r]^-{\text{tr}}& \begin{array}{c}\includegraphics[width=2.01in]{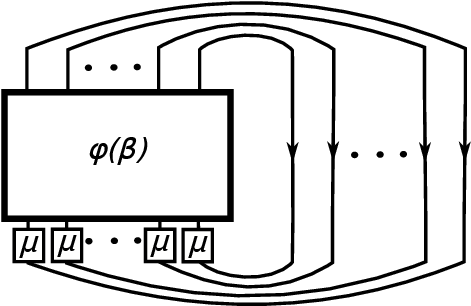} \end{array}
} & \\
& & \\ \hline
\end{tabular}
\]
\caption{Diagrammatic version of (\ref{eqn_diagrammatic_trace}).} \label{fig_diagrammatic_trace}
\end{figure}

\begin{theorem} \label{thm_trace_for_closure} For $\beta \in \mathit{VB}_N$, $\widetilde{Q}^{m|n} \circ \Zh(\widehat{\beta}) =\text{tr} \left(\widetilde{\rho}^{\,\,m|n}_N(\beta)  \circ \mu^{\otimes N}\right)$, where $\mu$ is given by $\ref{eqn_mu_mn}$.
\end{theorem}

 \begin{proof} This essentially follows from the functoriality of $\Zh$ and the proof of Equation \ref{eqn_diagrammatic_trace} for classical braids.  As above, write $\widehat{\beta}=\curvearrowright^N \circ (\beta \otimes \downarrow^N) \circ \raisebox{\depth}{\rotatebox{180}{$\curvearrowright$}}^N$. Then since $\Zh$ is a monoidal functor:
\begin{align*}
    \widetilde{Q}^{m|n} \circ \Zh(\widehat{\beta})& =\widetilde{Q}^{m|n}(\Zh(\curvearrowright^N)) \circ \widetilde{Q}^{m|n}(\Zh(\beta \otimes \downarrow^N))\circ \widetilde{Q}^{m|n}(\Zh(\raisebox{\depth}{\rotatebox{180}{$\curvearrowright$}}^N)) \\
    &=\widetilde{Q}^{m|n}(\curvearrowright^N) \circ \left(\widetilde{Q}^{m|n}(\Zh(\beta)) \otimes \text{id}^{\otimes N} \right)\circ \widetilde{Q}^{m|n}(\raisebox{\depth}{\rotatebox{180}{$\curvearrowright$}}^N) \\
    &=Q^{m|n}(\curvearrowright^N) \circ \left(\widetilde{Q}^{m|n}(\Zh(\beta)) \otimes \text{id}^{\otimes N} \right) \circ Q^{m|n}(\raisebox{\depth}{\rotatebox{180}{$\curvearrowright$}}^N)
\end{align*}
The last equality follows from the fact that $\curvearrowright^N$ and $\raisebox{\depth}{\rotatebox{180}{$\curvearrowright$}}^N$ contain no crossings. In this factorization, we see that the effect of any virtual crossings is contained in the representation $\widetilde{Q}^{m|n} \circ \Zh: \mathit{VB}_N \to GL((m+n)^N,\mathbb{C}(q,w))$ of the virtual braid group. Furthermore, the proof of \ref{eqn_diagrammatic_trace} for a classical braid $\beta$ involves only the parts $\curvearrowright^N$, $\raisebox{\depth}{\rotatebox{180}{$\curvearrowright$}}^N$, and $\text{id}^{\otimes N}$.  Hence, the proof of \ref{eqn_diagrammatic_trace} applies exactly as in the classical case and the proof is complete.
\end{proof}
Similarly, if $\beta \in \mathit{Br}_N$ and $\beta'$ is the $1$-$1$ tangle obtained from $\beta$ by closing the strands $2,\ldots, N$, $Q(\beta')$ is given by a composition of partial (or operator) traces (see e.g. \cite{jackson_moffatt}, Section 6.2):
\begin{align} \label{eqn_diagrammatic_part_trace}
Q_{\beta'}(q)=\text{tr}_2 \circ \text{tr}_3 \circ \cdots \circ \text{tr}_{N}\left(\varphi(\beta) \circ (\text{id} \otimes \mu^{\otimes N-1}\right)).
\end{align}
This is depicted diagrammatically in Figure \ref{fig_diagrammatic_part_trace}. The same formula holds for $\mathit{VB}_N$. 

\begin{theorem} \label{thm_trace_for_11_tangles} For $\beta \in \mathit{VB}_N$ and $\mu:V \to V$ as given in $\ref{eqn_mu_mn}$,
\[
\widetilde{Q}^{m|n} \circ \Zh(\beta') =\text{tr}_2 \circ \text{tr}_3 \circ \cdots \circ \text{tr}_{N} \left(\widetilde{\rho}^{\,\,m|n}_N(\beta)  \circ (\text{id} \otimes \mu^{\otimes N-1})\right),
\] 
\end{theorem}

\begin{proof} Write $\beta'=(\uparrow \otimes \curvearrowright^{N-1}) \circ (\beta \otimes \downarrow^{N-1}) \circ (\uparrow \otimes \raisebox{\depth}{\rotatebox{180}{$\curvearrowright$}}^{N-1})$. Then since $\Zh$ is a monoidal functor:
\begin{align*}
    \widetilde{Q}^{m|n} \circ \Zh(\beta')& =\widetilde{Q}^{m|n}(\Zh(\uparrow \otimes \curvearrowright^{N-1}))\circ \widetilde{Q}^{m|n}(\Zh(\beta \otimes \downarrow^{N-1})) \circ \widetilde{Q}^{m|n}(\Zh(\uparrow \otimes \raisebox{\depth}{\rotatebox{180}{$\curvearrowright$}}^{N-1})) \\
    &=\widetilde{Q}^{m|n}(\uparrow \otimes \curvearrowright^{N-1}) \circ \left(\widetilde{Q}^{m|n}(\Zh(\beta)) \otimes \text{id}^{\otimes N-1} \right)\circ \widetilde{Q}^{m|n}(\uparrow \otimes \raisebox{\depth}{\rotatebox{180}{$\curvearrowright$}}^{N-1}) \\
    &=Q^{m|n}(\uparrow \otimes \curvearrowright^N) \circ \left(\widetilde{\rho}^{\,\,m|n}_N(\beta) \otimes \text{id}^{\otimes N-1} \right)\circ Q^{m|n}(\uparrow \otimes \raisebox{\depth}{\rotatebox{180}{$\curvearrowright$}}^{N-1})
\end{align*}
The rest of the proof now follows as in the classical case.
\end{proof}

\begin{figure}[htb]
\small
\[
\begin{tabular}{|ccc|} \hline
 & & \\
& \xymatrix{ \begin{array}{c}
\includegraphics[width=1in]{trace_full_1.eps}
\end{array} \ar[rr]^-{\varphi(-) \circ (\text{id} \otimes \mu^{\otimes N-1})} & & \begin{array}{c}\includegraphics[width=1.01in]{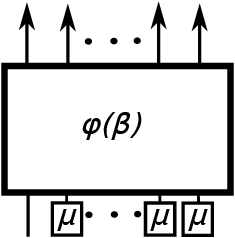} \end{array} \ar[r]^-{\text{tr}_2 \circ \cdots \circ \text{tr}_N}& \begin{array}{c}\includegraphics[width=1.95in]{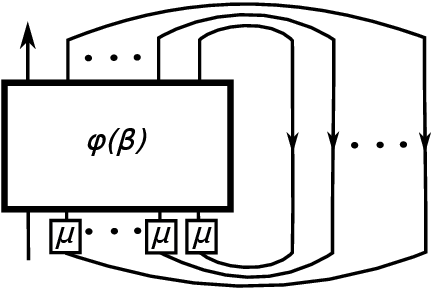} \end{array}
} & \\
& & \\ \hline
\end{tabular}
\]
\normalsize
\caption{Diagrammatic version of (\ref{eqn_diagrammatic_part_trace})}\label{fig_diagrammatic_part_trace}
\end{figure}

\section{The Alexander $\&$ Generalized Alexander Polynomials} \label{sec_alex_gen_alex}

The goal of this section is to show that the generalized Alexander polynomial of a virtual link and the Alexander polynomial of an almost classical link can be recovered from the extended $U_q(\mathfrak{gl}(1|1))$ functor. It should be reemphasized that these facts do not follow immediately from the skein relations in Theorem \ref{thm_skein_relation}. Not all virtual knots can be unknotted by crossing changes and hence the skein relation does not uniquely determine the invariants. Instead, these results are obtained using a argument similar to the one given by Kauffman-Saleur \cite{kauffman_saleur_91} for Alexander polynomial of classical links. First, we review the Burau and generalized Burau representations in Section \ref{sec_burau}. In Section \ref{sec_exterior}, we find another representation of $\mathit{VB}_N$ by applying the exterior algebra functor. With these preparations in place, the generalized Alexander polynomial is recovered in Section \ref{sec_recover_gen_alex}. Finally, the Alexander polynomial of an AC link is realized in Section \ref{sec_recover_alex}.

\subsection{The Burau $\&$ generalized Burau representations} \label{sec_burau} Recall the Burau representation $\rho_N:\mathit{VBr}_N \to \text{GL}(N,\mathbb{C}(t))$ of the virtual braid group $\mathit{VB}_N$. First, $\rho_2$ is defined on generators by:
\[
\rho_2(\sigma_1)=\begin{pmatrix}
    1-t & t \\ 1 & 0
\end{pmatrix}, \quad \rho_2(\sigma_1^{-1})=\begin{pmatrix}
    0 & 1 \\ \tfrac{1}{t} & 1-\tfrac{1}{t}
\end{pmatrix}, \quad \rho_2(\chi_1)=\begin{pmatrix}
    0 & 1 \\ 1 & 0
\end{pmatrix}.
\]
For $N>2$, the Burau representation is defined on generators by the following block matrices:
\[
\rho_{N}(\sigma_i^{\pm 1})=\left(\begin{array}{c|c|c}
    I_{i-1} & 0 & 0 \\ \hline
    0 & \rho_2(\sigma_i^{\pm 1}) & 0 \\ \hline
    0 & 0 & I_{N-i-1}
\end{array} \right), \quad \rho_{N}(\chi_i)=\left(\begin{array}{c|c|c}
    I_{i-1} & 0 & 0 \\ \hline
    0 & \rho_2(\chi_i) & 0 \\ \hline
    0 & 0 & I_{N-i-1}
\end{array} \right).
\]

For an almost classical link $L$, the Alexander polynomial of $L$ is obtained from the Burau representation as follows. Recall that every virtual link $L$ has a fundamental group $\pi_1(L)$ obtained via the Wirtinger presentation \cite{kauffman_vkt}. Applying the Fox calculus to the relations at each classical crossing yields the Alexander module of $L$ over $\mathbb{Z}[t^{\pm}]$. If $L$ is AC, the first elementary ideal $\mathcal{E}_1$ is principal (see e.g. \cite{NNST_12,BGHNW_17}) and a generator of $\mathcal{E}_1$ is the Alexander polynomial $\Delta_L(t)$, which is well-defined up to a unit of $\mathbb{Z}[t^{\pm}]$. Now, if $L=\widehat{\beta}$ for $\beta \in \mathit{VB}_N$, an argument by Fox calculus shows that $\rho_N(\beta)-I_N$ is also a presentation of the Alexander module of $L$. Since $\mathcal{E}_1$ is principal, the determinant of any $(N-1) \times (N-1)$ minor of $\rho_N(\beta)-I_N$ coincides with $\Delta_L(t)$, up to a multiple of $\pm t^k$ for $k \in \mathbb{Z}$. For definitiveness, let $A_{11}$ denote the matrix obtained from $\rho_N(\beta)$ by deleting the $N$-th row and column. Then, up to units:
\begin{align} \label{eqn_alex_defn}
    \Delta_L(t) &=\det(A_{11}-I_{N-1}) 
\end{align}
This formulation will be used ahead. Another well-known fact will be needed below. Observe that $\det(\rho_N(\beta)-I_N)=0$ for all $\beta \in \mathit{VB}_N$. Indeed, $\vec{v}=(1,\ldots,1)^{\intercal}$ is an eigenvector with eigenvalue $1$ for $\rho_N(\sigma_i^{\pm})$ and $\rho_N(\chi_i)$ for all $1 \le i \le N-1$. Then it is also an eigenvector for $\rho_N(\beta)$ for all $\beta \in \mathit{VB}_N$ and hence $(\rho_N(\beta)-I_{N}) \vec{v}=0$.   

%\begin{remark}
%    It looks like I am writing my braid words backwards, so that $\rho_N$ is really a representation of the opposite group.
%\end{remark}

The generalized Burau representation $\widetilde{\rho}_N:\mathit{VB}_N \to \text{GL}(N,\mathbb{C}(s,t))$ for virtual braids was first defined in Kauffman-Radford \cite{kauffman_radford_00}. It is naturally obtained from the theory of biquandles, in the same way that Kauffman-Saleur \cite{kauffman_saleur_92} initially defined the JKS-invariant for knots in thickened surfaces. For $N=2$, it is defined on generators by: 
\[
\widetilde{\rho}_2(\sigma_1)=\begin{pmatrix}
    1-st & s \\ t & 0
\end{pmatrix}, \quad \widetilde{\rho}_2(\sigma_1^{-1})=\begin{pmatrix}
    0 & \tfrac{1}{t} \\ \tfrac{1}{s} & 1-\tfrac{1}{st}
\end{pmatrix}, \quad \widetilde{\rho}_2(\chi_1)=\begin{pmatrix}
    0 & 1 \\ 1 & 0
\end{pmatrix}
\]
For $N>2$, the generalized Burau representation is defined on generators by the block matrices:
\[
\widetilde{\rho}_{N}(\sigma_i^{\pm})=\left(\begin{array}{c|c|c}
    I_{i-1} & 0 & 0 \\ \hline
    0 & \widetilde{\rho}_2(\sigma_i^{\pm}) & 0 \\ \hline
    0 & 0 & I_{N-i-1}
\end{array} \right), \quad \widetilde{\rho}_{N}(\chi_i)=\left(\begin{array}{c|c|c}
    I_{i-1} & 0 & 0 \\ \hline
    0 & \widetilde{\rho}_2(\chi_i) & 0 \\ \hline
    0 & 0 & I_{N-i-1}
\end{array} \right)
\]

Using the $\Zh$-construction, the generalized Alexander polynomial $G_L(s,t)$ can be obtained from $\widetilde{\rho}_N$ in the same way that $\Delta_L(t)$ is obtained from $\rho_N$. In \cite{boden_chrisman_21}, it was proved that (after a change of variables) $G_L(s,t)$ is the multi-variable Alexander polynomial of the $(r+1)$-component virtual link $\Zh(L)$. In this case, the components of $L$ are assigned to the variable $s$ while the $\omega$ component goes to the variable $t$. Again using Fox calculus, one arrives at a presentation matrix for the Alexander module over $\mathbb{Z}[s^{\pm 1},t^{\pm 1}]$. The principal minor which ignores the column corresponding to the generator of the $\omega$ component then gives the generalized Alexander polynomial. If $L=\widehat{\beta}$ for some $\beta \in \mathit{VB}_N$, this coincides by Fox calculus with $\widetilde{\rho}_N(\beta)-I_N$. Hence we have: 
\begin{align} \label{thm_gap_det}
 G_L(s,t) &=\det(\widetilde{\rho}_N(\beta)-I_N),   
\end{align}
which is well-defined up to multiples of $\pm s^j t^k$ for $j,k \in \mathbb{Z}$. For further details on the algebraic aspects, see e.g. \cite{BDGGHN_15}, Theorem 4.4 and Corollary 4.8. The above discussion suffices for what is needed in the subsequent sections.

\subsection{The exterior algebra} \label{sec_exterior} Consider the case of the Alexander polynomial. Denote the vector space over $\mathbb{C}(t)$ on which the Burau representation acts by $U$, with ordered basis $u_1,\ldots, u_N$. Recall that the exterior algebra $\bigwedge ^* U$ has basis consisting of $1$ and elements of the form $u_{a_1} \wedge u_{a_2} \wedge \cdots \wedge u_{a_k}$ where $a_1<a_2<\cdots <a_k$ and $k \le N$. Denote by $\bigwedge^*(f)$ the induced map of a linear transformation $f:U \to U$. Since $\mathit{VB}_N$ is generated by $\sigma_i^{\pm}$, $\chi_i$ for $1 \le i \le N-1$ and $\bigwedge^*$ is functorial, it is enough to determine $\bigwedge^* (\rho_2(\sigma_1))$ and $\bigwedge^* (\rho_2(\chi_1))$. First, perform a change of variables $t\to q^{-2}$ and then make a consistent change of basis. The change of basis matrix is given by:
\begin{equation} \label{eqn_COB}
    \begin{pmatrix}
    q^{-1/2} & 0         \\
        0    & q^{1/2}    
\end{pmatrix}.
\end{equation}
Then the Burau representation $\rho_2:VB_2 \to GL(2,\mathbb{C}(q))$ is given on generators by:
\[
\rho_2(\sigma_1)=\begin{pmatrix}
    1-\tfrac{1}{q^2} & \tfrac{1}{q} \\
    \tfrac{1}{q} & 0
    \end{pmatrix}, \quad \rho_2(\sigma_1^{-1})=\begin{pmatrix}
    0 & q \\
    q & 1-q^2
    \end{pmatrix}, \quad \rho_2(\chi_1)=\begin{pmatrix}
    0 & q \\
    \tfrac{1}{q} & 0
    \end{pmatrix}. 
\]
The induced maps $\bigwedge\!\!\,^* (\rho_2(\sigma_1))$, $\bigwedge\!\!\,^* (\rho_2(\sigma_1^{-1}))$, $\bigwedge\!\!\,^* (\rho_2(\chi_1))$ are, respectively, given by:
\[
\begin{pmatrix}
    1 & 0 & 0 & 0 \\
    0 & 1-\tfrac{1}{q^2} & \tfrac{1}{q} & 0 \\
    0 & \tfrac{1}{q} & 0 & 0 \\
    0 & 0 & 0 & -\tfrac{1}{q^2}
\end{pmatrix}, \quad \begin{pmatrix}
    1 & 0 & 0 & 0 \\
    0 & 0 & q & 0 \\
    0 & q & 1-q^2 & 0 \\
    0 & 0 & 0 & -q^2
\end{pmatrix}, \quad
\begin{pmatrix}
    1 & 0 & 0 & 0 \\
    0 & 0 & q & 0 \\
    0 & \tfrac{1}{q} & 0 & 0 \\
    0 & 0 & 0 & -1
\end{pmatrix},
\]
\normalsize
\begin{remark} \label{rem_near_COB}Observe that $\bigwedge\!\!\,^* (\rho_2(\chi_1))$ is exactly the value of $Q^{1|1}\left(\begin{array}{c} \includegraphics[width=.15in]{ovcross.eps}\end{array} \right)$. As explained in Remark \ref{rem_correct_COB}, our choice for the value of the virtual braiding is forced by the Burau representation and the change of basis matrix (\ref{eqn_COB}).
\end{remark}

For the generalized Alexander polynomial, the required change of variables is $s \to t^{-1} q^{-2}$ followed by $t \to w$. Together with the change of basis given in (\ref{eqn_COB}), this yields:
\[
\widetilde{\rho}_2(\sigma_1)=\begin{pmatrix}
    1-\tfrac{1}{q^2} & \tfrac{1}{wq} \\
    \tfrac{w}{q} & 0
    \end{pmatrix}, \quad \widetilde{\rho}_2(\sigma_1^{-1})=\begin{pmatrix}
    0 & \tfrac{1}{w} q \\
    w q & 1-q^2
    \end{pmatrix}, \quad \widetilde{\rho}_2(\chi_1)=\begin{pmatrix}
    0 & q \\
    \tfrac{1}{q} & 0
    \end{pmatrix}. 
\]
The induced maps $\bigwedge\!\!\,^* (\widetilde{\rho}_2(\sigma_1))$, $\bigwedge\!\!\,^* (\widetilde{\rho}_2(\sigma_1^{-1}))$, and $\bigwedge\!\!\,^* (\widetilde{\rho}_2(\chi_1))$ are, respectively, given by:

\[
\begin{pmatrix}
    1 & 0 & 0 & 0 \\
    0 & 1-\tfrac{1}{q^2} & \tfrac{1}{wq} & 0 \\
    0 & \tfrac{w}{q} & 0 & 0 \\
    0 & 0 & 0 & -\tfrac{1}{q^2}
\end{pmatrix}, \quad \begin{pmatrix}
    1 & 0 & 0 & 0 \\
    0 & 0 & \tfrac{1}{w} q & 0 \\
    0 & wq & 1-q^2 & 0 \\
    0 & 0 & 0 & -q^2
\end{pmatrix},\quad
\begin{pmatrix}
    1 & 0 & 0 & 0 \\
    0 & 0 & q & 0 \\
    0 & \tfrac{1}{q} & 0 & 0 \\
    0 & 0 & 0 & -1
\end{pmatrix}.
\]

Next, $\bigwedge ^* U$ is identified with $V^{\otimes N}$, where $V$ is again the vector representation of $U_q(\mathfrak{gl}(1|1))$. Recall that $V$ is generated by $x_1,x_2$ with $x_1$ even and $x_2$ odd. First, $1$ is identified with $x_1 \otimes x_1 \otimes \cdots \otimes x_1$. The vector $u_i$ is identified with the basis vector for $V^{\otimes N}$ having $x_2$ in the $i+1 \pmod N$ position and $x_1$ elsewhere. Hence, $u_1$ is identified with $x_1 \otimes x_2 \otimes x_1 \otimes \cdots \otimes x_1$, $u_2$ is identified with $x_1 \otimes x_1 \otimes x_2 \otimes x_1 \cdots \otimes x_1$, and $u_N$ is identified with $x_2 \otimes x_1 \otimes x_1 \otimes \cdots \otimes x_1$. In general, the basis element $u_{i_1} \wedge u_{i_2} \wedge \cdots \wedge u_{i_k}$ where $i_1 < \ldots < i_N$ is identified with the basis element having $x_2$ in the the $i_j+1 \pmod{N}$ position for every $1 \le j \le k$ and $x_1$ in all other positions. Note that $V^{\otimes N}$ is $\mathbb{Z}_2$-graded so that $x_{a_1} \otimes x_{a_2} \otimes \cdots \otimes x_{a_N}$ is odd if $x_2$ appears an odd number of times in the $x_{a_i}$ and even otherwise (see Section \ref{sec_glmn_defn}). Hence, $u_{i_1} \wedge u_{i_2} \wedge \cdots \wedge u_{i_k}$ is odd exactly when $k$ is odd.

For $N=2$, take as an ordered basis for $V^{\otimes 2}$ the set $\{x_1 \otimes x_1, x_1 \otimes x_2, x_2 \otimes x_1, x_2 \otimes x_2\}$. With respect to this basis, the matrix $\check{R}$ for $U_q(\mathfrak{gl}(1|1))$ is given by (see Section \ref{vectrepmn}): 
\[
\check{R}=\begin{pmatrix}
    q & 0 & 0 & 0 \\
    0 & q-\tfrac{1}{q} & 1 & 0 \\
    0 & 1 & 0 & 0 \\
    0 & 0 & 0 & -\tfrac{1}{q}
\end{pmatrix}. 
\]
Hence, $\check{R}=q \cdot \bigwedge^*(\rho_2(\sigma_1))$. Similarly, $\check{R}^{-1}=q^{-1} \cdot \bigwedge^*(\sigma_1^{-1})$ and $Q^{1|1}_{\chi_1}(q)=\bigwedge^*(\rho_2(\chi_1))$. This proves the following lemma relating the Burau representation and the virtual $U_q(\mathfrak{gl}(1|1))$ functor.

\begin{lemma} \label{lemma_exterior_for_virtual_functor}  For any $\beta \in \mathit{VB}_N$, $Q_{\beta}^{1|1}(q)=q^{wr(\beta)}\cdot \bigwedge^* (\rho_N(\beta))$.
\end{lemma}

The extended $U_q(\mathfrak{gl}(1|1))$ functor is connected to the generalized Burau representation via the $\Zh$-construction. From Figure \ref{fig_zh_construction} and the definition of $\widetilde{Q}^{1|1}$, we compute:
\begin{align*}
    \widetilde{Q}_{\zh(\sigma_1)}^{1|1}(q,w) &= \left(\widetilde{Q}^{m|n}\left(\begin{array}{c}\includegraphics[width=.2in]{omoverplus.eps} \end{array}\right)\otimes \text{id}_V \right)\circ \left(\text{id}_{\mathbb{F}} \otimes \widetilde{Q}^{m|n}\left(\begin{array}{c} 
\includegraphics[width=.2in]{pluscross.eps} \end{array} \right)\right) \circ \left(\widetilde{Q}^{m|n}\left(\begin{array}{c} 
\includegraphics[width=.2in]{omoverminus.eps} \end{array} \right) \otimes \text{id}_V\right) \\ 
    &= \left(\begin{pmatrix}
        1 & 0 \\ 0 & w   \end{pmatrix} \otimes I_2 \right)
        \circ \left(I_1 \otimes \begin{pmatrix}
    q & 0 & 0 & 0 \\
    0 & q-\tfrac{1}{q} & 1 & 0 \\
    0 & 1 & 0 & 0 \\
    0 & 0 & 0 & -\tfrac{1}{q}
\end{pmatrix} \right) \circ \left(\begin{pmatrix}
            1 & 0 \\ 0 & w^{-1} 
        \end{pmatrix} \otimes I_2 \right) \\
        &=\begin{pmatrix}
    q & 0 & 0 & 0 \\
    0 & q-\tfrac{1}{q} & \tfrac{1}{w} & 0 \\
    0 & w & 0 & 0 \\
    0 & 0 & 0 & -\tfrac{1}{q}
\end{pmatrix}\\
&= q \cdot \bigwedge\!\!\,^*(\widetilde{\rho}_2(\sigma_1))
\end{align*}
Likewise, $\widetilde{Q}^{1|1}_{\zh(\sigma_1^{-1})}(q,w)=q^{-1} \bigwedge^*(\widetilde{\rho}_2(\sigma_1))$ and $\widetilde{Q}^{1|1}_{\zh(\chi_1)}(q,w)=\bigwedge^*(\widetilde{\rho}_2(\chi_1))$. This proves:
\begin{lemma} \label{lemma_exterior_for_gen_func}
 For any $\beta \in \mathit{VB}_N$, $\widetilde{Q}_{\zh(\beta)}^{1|1}(q,w)=q^{wr(\beta)}\cdot \bigwedge^* (\widetilde{\rho}_N(\beta))$.   
\end{lemma}

\subsection{Recovery of the generalized Alexander polynomial} \label{sec_recover_gen_alex} First recall the following well-known formula for an $r \times r$ matrix $A$:
\begin{equation} \label{eqn_det_trace}
    \det(z \cdot I_r-A)=\sum_{k=0}^r z^{r-k}(-1)^k \text{tr}\left(\bigwedge \! ^k A \right).
\end{equation}
For $\beta \in \mathit{VB}_N$, set $z=1$ and $A=\widetilde{\rho}_N(\beta)$ in the above formula. After rearranging, the result is: 
\[
\det(\widetilde{\rho}_N(\beta)-I_N)=(-1)^N \sum_{k=0}^N (-1)^k \text{tr}\left(\bigwedge \! ^k \widetilde{\rho}_N(\beta) \right).
\]
By Equation (\ref{thm_gap_det}), the left-hand side is $G_{\widehat{\beta}}(w^{-1}q^{-2},w)$. In block form, the right-hand side is:
\begin{equation} \label{eqn_trace_rho}
(-1)^N \text{tr}\left(
\begin{array}{c|c|c|c|c} 
1 & 0 & 0 & \cdots & 0 \\ \hline
0 & -\widetilde{\rho}_N(\beta) & 0 & \cdots & 0 \\ \hline
0 & 0 & \bigwedge^2 \widetilde{\rho}_N(\beta) & \cdots & 0 \\ \hline
\vdots & \vdots & \vdots & \ddots & \vdots \\ \hline
0 & 0 & 0 & \cdots & (-1)^N \bigwedge^N \widetilde{\rho}_N(\beta)
\end{array} \right)
\end{equation}
Recall from Example \ref{ex_mu_11} that $\mu^{\otimes N}(x_{i_1} \otimes \cdots \otimes x_{i_N})=q^N (-1)^{\sum_j |i_j|}x_{i_1} \otimes \cdots \otimes x_{i_N}$. Then (\ref{eqn_trace_rho}) is:
\begin{equation} \label{eqn_trace_final}
(-1)^N q^{-N} \text{tr}\left(\bigwedge \! ^* \widetilde{\rho}_N(\beta) \circ \mu^{\otimes N}\right)
\end{equation}
Combining (\ref{eqn_trace_final}) with Lemma \ref{lemma_exterior_for_gen_func} and Theorem \ref{thm_trace_for_closure} gives the following result.

\begin{theorem} \label{thm_recover_GAP} For $\beta \in \mathit{VB}_N$ and $L=\widehat{\beta}$, the following formula holds for the GAP:
\[
G_{\widehat{\beta}}(w^{-1}q^{-2},w)=\det(\widetilde{\rho}_N(\beta)-I_N)=(-1)^N q^{-N-wr(\beta)}\widetilde{Q}_{\zh(\widehat{\beta})}^{1|1}(q,w). 
\]
\end{theorem}

\subsection{Recovery of the Alexander polynomial} \label{sec_recover_alex} By Theorem \ref{thm_diagonal}, we know that $Q^{1|1}_T(q)$ is a diagonal matrix for any $1$-$1$-tangle $T$. The first step is to show that $Q_T^{1|1}(q)$ is actually scalar multiple of the identity matrix. This is also true in the classical case, but unfortunately the classical proof does not apply. This is again due to the fact that the matrix assigned to a virtual crossing is not a $U_q(\mathfrak{gl}(1|1))$-module homomorphism. Our argument proceeds by a sequence of lemmas.

\begin{lemma} Let $T$ be 1-1 virtual tangle $T$ and write $Q_T^{1|1}(q)=\begin{pmatrix} \alpha_1 & 0 \\ 0 & \alpha_2  \end{pmatrix}$. Define a virtual link diagram $S:\varnothing \to \varnothing$ by $S=\curvearrowleft \circ (\downarrow \otimes \,\, T) \circ \raisebox{\depth}{\rotatebox{180}{$\curvearrowleft$}}$. Then:
\[
Q_{S}^{1|1}(q)=\frac{1}{q}(\alpha_1-\alpha_2).
\]
\end{lemma}
\begin{proof}  The proof is by direct calculation:
\begin{align*}
    Q^{1|1}_{S}(q) &= Q^{1|1}_{\rotatebox{360}{$\curvearrowleft$}}(q) \circ Q^{1|1}_{\downarrow \,\otimes \, T}(q) \circ Q_{\raisebox{\depth}{\rotatebox{180}{$\curvearrowleft$}}}^{1|1}(q) \\
    &= \begin{pmatrix}
        1 & 0 & 0 & 1
    \end{pmatrix} \begin{pmatrix}
        \alpha_1 & 0 & 0 & 0 \\
        0 & \alpha_2 & 0 & 0 \\
        0 & 0 & \alpha_1 & 0 \\
        0 & 0 & 0 & \alpha_2 
    \end{pmatrix}
    \begin{pmatrix}
        q^{-1} \\ 0 \\ 0 \\ -q^{-1}
    \end{pmatrix} \\
    &= \frac{1}{q}(\alpha_1-\alpha_2).
\end{align*}
\end{proof}

\begin{lemma} For any virtual link $L$, $Q^{1|1}_L(q)=0$.
\end{lemma}
\begin{proof} Write $L=\widehat{\beta}$ for some $\beta \in VB_N$. Recall that $\widetilde{Q}_{\zh(L)}^{1|1}(q,1)=Q_T^{1,1}(q)$. By Theorem \ref{thm_recover_GAP},
\[
\det(\rho_N(\beta)-I_N)=(-1)^Nq^{-N-wr(\beta)}Q_{\widehat{\beta}}^{1|1}(q).
\]
As discussed in Section \ref{sec_burau}, $\det(\rho_N(\beta)-I_N)=0$ for all $\beta \in \mathit{VB}_N$. Hence, $Q^{1|1}_L(q)=0$ as well. 
\end{proof}

\begin{theorem} \label{thm_scalar}
    For any 1-1 virtual tangle, $Q_T^{1|1}(q)$ is a scalar multiple of the identity matrix.
\end{theorem}
\begin{proof} Suppose that $Q_T^{1|1}(q)=\begin{pmatrix} \alpha_1 & 0 \\ 0 & \alpha_2  \end{pmatrix}$. Combining the two previous lemmas gives $0=\tfrac{1}{q}(\alpha_1-\alpha_2)$, and  hence $\alpha_1=\alpha_2$. 
\end{proof}

In particular, this shows that there is a well-defined AP for any virtual 1-1 tangle. Another corollary is yet another proof of the Silver-Williams Theorem.

\begin{corollary}
    If $L$ is almost classical, then $\widetilde{Q}_{\zh(L)}^{1|1}(q,w)=0$.
\end{corollary}
\begin{proof} $\Zh(L)$ is the split link $L \sqcup \textcolor{blue}{\bigcirc}$. Hence, $\widetilde{Q}_{\zh(L)}^{1|1}(q,w)=Q_{L}^{1|1}(q)=0$. \end{proof}

With these lemmas in place, we are ready to prove the main result of this subsection.

\begin{theorem} \label{thm_gen_ac_eq_AP} Let $L=\widehat{\beta}$ for some almost classical $\beta \in \mathit{VB}_N$. Then:
\[
    \widetilde{Q}_{\beta'}^{1|1}(q,w) = Q_{\beta'}^{1|1}(q)
           =q^{wr(\beta)+N-1}(-1)^{N-1} \Delta_L(q^{-2}) 
            \begin{pmatrix} 1 & 0 \\ 0 & 1  
            \end{pmatrix}.
\]
\end{theorem}

\begin{proof} Since $\beta$ is AC, $\beta'$ is an AC 1-1 virtual tangle. Then by Corollary \ref{cor_AC_11_tangles_split}, $\Zh(\beta')$ splits as $\beta' \sqcup \textcolor{blue}{\bigcirc}$. This proves the first equality. For the second equality, Theorem \ref{thm_scalar} states that $Q_{\beta'}^{1|1}(q)$ is a scalar multiple of $I_2$. Hence, by  Equation (\ref{eqn_alex_defn}),  it suffices to show that the $(1,1)$ entry of $Q_{\beta'}^{1|1}(q)$ is the first principal minor of $\rho_N(\beta)-I_N$. By Theorem \ref{thm_trace_for_11_tangles} and Lemma \ref{lemma_exterior_for_virtual_functor},
\begin{align} \label{eqn_alex_real_first_step}
 Q_{\beta'}^{1|1}(q)&=q^{wr(\beta)} \cdot \text{tr}_2 \circ \text{tr}_3 \circ \cdots \circ \text{tr}_N \left( \bigwedge\! ^* \rho_N(\beta) \circ (\text{id} \otimes \mu^{\otimes (N-1)}) \right).
\end{align}
The composition of trace maps can be calculated as follows. Let $C \in \text{End}(V^{\otimes N}) \cong (V^{\otimes N})^* \otimes V^{\otimes N}$.  In tensor notation, the entries of the matrix of $C$ with respect to the bases for $V^{\otimes N},(V^{\otimes N})^*$ are:
\[
C^{i_1 i_2 \cdots i_N}_{j_1 j_2 \cdots j_N},
\]
where $i_k,j_k \in \{1,2\}$ for all $k$. The above isomorphism identifies $C$ with the tensor:
\[
\sum_{i_k,j_k} C^{i_1 i_2 \cdots i_N}_{j_1 j_2 \cdots j_N} \cdot x^{j_1} \otimes \cdots \otimes x^{j_N} \otimes x_{i_1} \otimes \cdots \otimes x_{i_N},
\]
where $x^{1},x^2$ is the dual basis for $V^*$. Then $\text{tr}_2 \circ \cdots \circ \text{tr}_N(C)$ is given by:
\[
\sum_{i_1,j_1} \sum_{i_2=1}^2 \cdots \sum_{i_N=1}^2 C^{i_1 i_2 \cdots i_N}_{j_1 i_2 \cdots i_N} \cdot x^{j_1} \otimes x_{i_1}.
\]
The $(1,1)$ entry of the corresponding matrix in $\text{End}(V)$ is:
\begin{equation} \label{eqn_part_trace_as_tensor}
\sum_{i_2=1}^2 \cdots \sum_{i_N=1}^2 C^{1 \, i_2 \cdots i_N}_{1 \, i_2 \cdots i_N}.
\end{equation}

Now, set $C=\bigwedge\! ^* \rho_N(\beta) \circ (\text{id} \otimes \mu^{\otimes (N-1)})$. The terms of the above sum correspond to certain diagonal entries of $\bigwedge\! ^* \rho_N(\beta) \circ (\text{id} \otimes \mu^{\otimes (N-1)})$. By \ref{ex_mu_11}, $\text{id} \otimes \mu^{\otimes (N-1)}$ is a diagonal matrix whose entries are $\pm q^{N-1}$, so that the diagonal entries of $\bigwedge\! ^* \rho_N(\beta) \circ (\text{id} \otimes \mu^{\otimes (N-1)})$ are those of $\bigwedge\! ^* \rho_N(\beta)$ multiplied by $\pm q^{N-1}$. The diagonal entries of $\bigwedge\! ^* \rho_N(\beta)$ are the principal $k \times k$ minors of $\rho_N(\beta)$. The identification of $\bigwedge \! ^* U$ with $V^{\otimes N}$ implies that (\ref{eqn_part_trace_as_tensor}) adds up only those principal $k \times k$ minors not involving the last row or column of $\rho_N(\beta)$. Indeed, the indices of the term $C^{1 \, i_2 \cdots i_N}_{1 \, i_2 \cdots i_N}$ imply that $u_N$ is excluded, since the basis elements $u_{a_1} \wedge \cdots \wedge u_N$ of $\bigwedge \! ^* U$ correspond to basis elements of $V^{\otimes N}$ that start with $x_2$. Hence, (\ref{eqn_part_trace_as_tensor}) is $q^{N-1}$ times the sum of ($\pm$) the principal $k \times k$ minors of $\rho_N(\beta)$ that do not involve the last row or column of $\rho_N(\beta)$.  The minors that do not involve the last row or column are the same as the $k \times k$ minors of the $(N-1) \times (N-1)$ submatrix obtained from $\rho_N(\beta)$ by deleting its last row and column. Recall from Section \ref{sec_burau} that this submatrix is denoted $A_{1 1}$.

It only remains to determine the $\pm$ signs in the sum (\ref{eqn_part_trace_as_tensor}). First observe that:
\[
(\text{id} \otimes \mu^{\otimes N-1}) (x_1 \otimes x_{a_1} \otimes \cdots \otimes x_{a_{N-1}})=(-1)^{\sum_{i=1}^{N-1} |a_i|} q^{N-1}x_1 \otimes x_{a_1} \otimes \cdots \otimes x_{a_{N-1}}.
\]
Now, a principal $k \times k$ minor of $A_{1 1}$ corresponds to a choice of $k$ basis vectors $u_{b_1},\ldots,u_{b_k}$ with $b_1< \cdots <b_k$ and $k<N$. If the corresponding basis element of $V^{\otimes N}$ is denoted by $x_{a_1} \otimes \cdots \otimes x_{a_{N-1}}$, then $(-1)^{\sum_{i=1}^{N-1} |a_i|}=-1$ when $k$ is odd and $1$ when $k$ is even. This implies that (\ref{eqn_part_trace_as_tensor}) is equal to:
\[
q^{N-1} \left(1-\text{tr}(A_{1 1})+\text{tr}\left(\bigwedge \! ^2 A_{1 1}\right)-\cdots+(-1)^{N-1} \text{tr}\left(\bigwedge \! ^{N-1} A_{1 1}\right)\right).
\]
By Equation (\ref{eqn_det_trace}), it follows that:
\[
 q^{N-1}\sum_{k=1}^{N-1}(-1)^k \text{tr}\left(\bigwedge \! ^k A_{1 1} \right)=(-1)^{N-1}q^{N-1}\det(A_{1 1}-I_{N-1})=(-1)^{N-1}q^{N-1} \Delta_{\widehat{\beta}}(q^{-2}).
\]
Combining this expression with Equation \ref{eqn_alex_real_first_step} gives the desired conclusion. 
\end{proof}

\section{Calculations $\&$ Examples}
\label{sec_calc}
Code for computing the virtual $U_q(\mathfrak{gl}(m|n))$ functor and semi-welded $U_q(\mathfrak{gl}(m|n))$ functor of virtual tangle diagrams has been implemented in \emph{Mathematica} \cite{mathematica}. A notebook explaining the code and containing all of the following examples is available here:
\begin{center}
\url{https://github.com/micah-chrisman/virtual-knots/blob/main/gen_glmn_final.nb}
\end{center}
\subsection{GAP of the virtual trefoil} The $\Zh$-construction was illustrated in Figure \ref{fig_zh_left_tref} with the left-handed virtual trefoil. Let's verify Theorem \ref{thm_recover_GAP} in this simplest case. As a virtual braid, the virtual trefoil is $\beta=\chi_1 \sigma_1^{-2} \in \emph{VB}_2$. The middle expression in Theorem \ref{thm_recover_GAP} is:
\[
\det(\widetilde{\rho}(\beta)-I_2)=-\frac{1}{s^2 t^2}+\frac{1}{s^2 t}+\frac{1}{s t^2}-\frac{1}{s}-\frac{1}{t}+1,
\]
which agrees exactly with Green's table \cite{green}. The change of variables $s\to w^{-1} q^{-2}$, $t \to w$ gives:
\[
-q^4+q^4w+\frac{q^2}{w}-q^2w-\frac{1}{w}+1.
\]
This matches the right hand side of the formula in Theorem \ref{thm_recover_GAP}, which is given by:
\[
(-1)^2q^{-2-(-2)}\widetilde{Q}^{1|1}_{\zh(\widehat{\beta})}(q,w)=q^4 w-q^4-q^2 w+\frac{q^2}{w}-\frac{1}{w}+1.
\]
%\[
%(-1)^2q^{-2-(-2)}\widetilde{Q}^{1|1}_{\zh(\widehat{\beta})}(q,a,b)=\frac{b q^4}{a}-\frac{b q^2}{a}+\frac{a q^2}{b}-\frac{a}{b}-q^4+1.
%\]
\subsection{Almost classical knots $\&$ the Alexander polynomial} The virtual knot $K=4.105$ from Green's table is almost classical. A diagram of $K$ is shown at the top of Figure \ref{fig_4_105}. As a closed braid, $K=\widehat{\beta}$ where $\beta=\chi_1 \sigma_1^{-1} \chi_1 \sigma_2^{-1}\chi_1 \sigma_1^{-1} \chi_1 \sigma_2^{-1} \in \emph{VB}_3$. Calculating the virtual and generalized $U_q(\mathfrak{gl}(1|1))$ polynomials directly gives the expected values:
\[
\widetilde{Q}_{\zh(K)}^{1|1}(q,w)=Q_K^{1|1}(q)=0.
\]
The bottom of Figure \ref{fig_4_105} shows a 1-1 virtual tangle $T$ whose left closure is 4.105. Calculating the virtual and extended $U_q(\mathfrak{gl}(1|1))$ invariants directly from the definition yields:
\[
\widetilde{Q}_{\zh(K)}^{1|1}(t^{-1/2},w)=Q_K^{1|1}(t^{-1/2})= (1-\tfrac{2}{t}+\tfrac{2}{t^2}) \begin{pmatrix}
    1 & 0 \\ 0 & 1
\end{pmatrix}
\]
According to \cite{BGHNW_17}, Table 2, the Alexander polynomial of 4.105 is $\Delta_K(t)=1-2t+2t^2$. Here we have followed the orientation conventions of Kauffman-Saleur \cite{kauffman_saleur_91}, which are opposite of that in \cite{BGHNW_17}. This has the effect of switching the positive and negative push-offs when calculating the Seifert matrices, which in turn has the effect of exchanging $t$ with $t^{-1}$ in the Alexander polynomial. The reader can verify this by checking the Seifert matrices for 4.105 given in \cite{bcg2}, Table 3. 

\begin{figure}[htb]
\begin{tabular}{|c|} \hline
 \\
 \begin{tabular}{c}
\def\svgwidth{3in} \tiny 
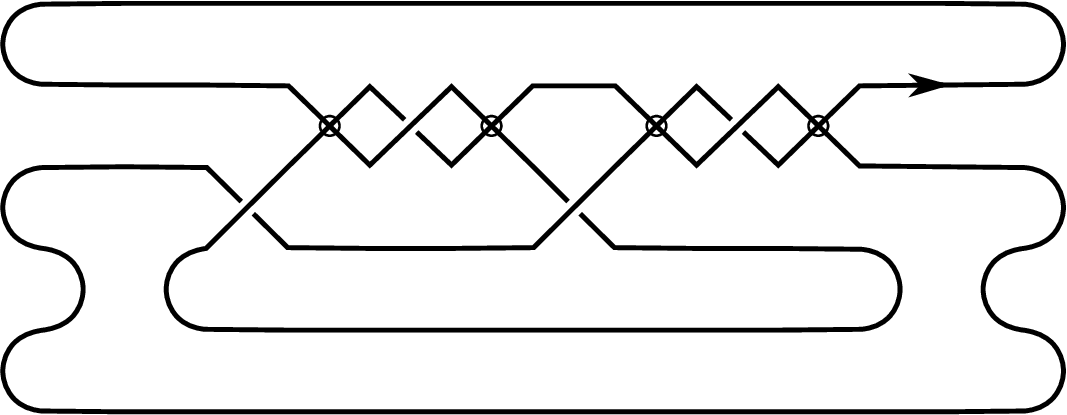
\end{tabular} \\ \\ \hline \\
\begin{tabular}{c}
\def\svgwidth{3in} \tiny 
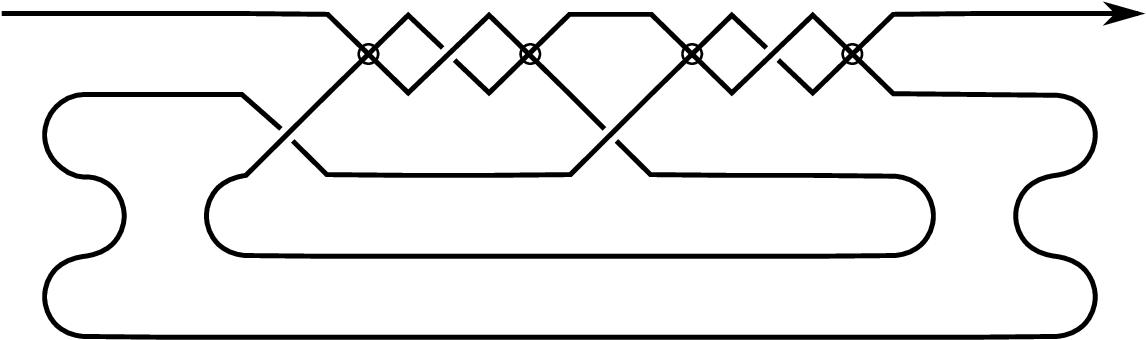
\end{tabular}
 \\ \\ \hline
\end{tabular}
\caption{$\circlearrowleft$ The AC knot 4.105 (top) and a 1-1 AC tangle whose closure is 4.105 (bottom).} \label{fig_4_105}
\end{figure}

\subsection{Same $\widetilde{f}^{\,\,1|1}$, different $\widetilde{f}^{\,\,2|2}$} The virtual knots $K_1=3.5$ and $K_2=4.106$ have the same generalized Alexander polynomial and hence the same generalized $U_q(\mathfrak{gl}(1|1))$ polynomials. Tangle decompositions of $3.5$ and $4.106$ are given in Figure \ref{fig_same_11_diff_22}.  Then:
\[
\widetilde{Q}^{1|1}_{\zh(K_1)}(q,w)=q^3 w^2-\frac{1}{q^3 w^2}-q^3+\frac{1}{q^3}-q w^2+\frac{1}{q w^2}+q-\frac{1}{q}=\widetilde{Q}^{1|1}_{\zh(K_2)}(q,w).
\]
However, calculating their extended $U_q(\mathfrak{gl}(2|2))$ invariants gives: 
\begin{align*}
    \widetilde{Q}^{2|2}_{\zh(K_1)}(q,w) &= q^5 w^2-\frac{1}{q^5 w^2}+q^3 w^2-\frac{1}{q^3 w^2}-4 q^3+\frac{4}{q^3}-q w^2+\frac{q}{w^2}+\frac{1}{q w^2}-\frac{w^2}{q}+4 q-\frac{4}{q}, \\
    \widetilde{Q}^{2|2}_{\zh(K_2)}(q,w) &= -q^9w^2+q^9+q^7 w^2-2 q^7+3 q^5 w^2-\frac{q^5}{w^2}+3 q^5-q^3 w^2+\frac{q^3}{w^2}-\frac{2}{q^3 w^2}-8 q^3 \\
    &+\frac{1}{q^3}-2 q w^2+\frac{3 q}{w^2}-\frac{1}{q w^2}+3 q+\frac{2}{q}.
\end{align*}
\begin{figure}[htb]
\begin{tabular}{|c|} \hline \\ 
    \includegraphics[width=1.5in]{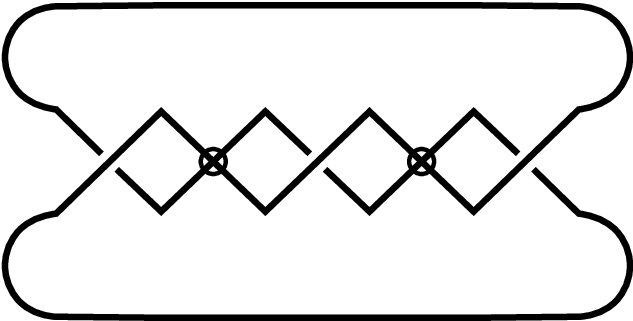} \\  \underline{3.5:} \\ \\ \includegraphics[width=3in]{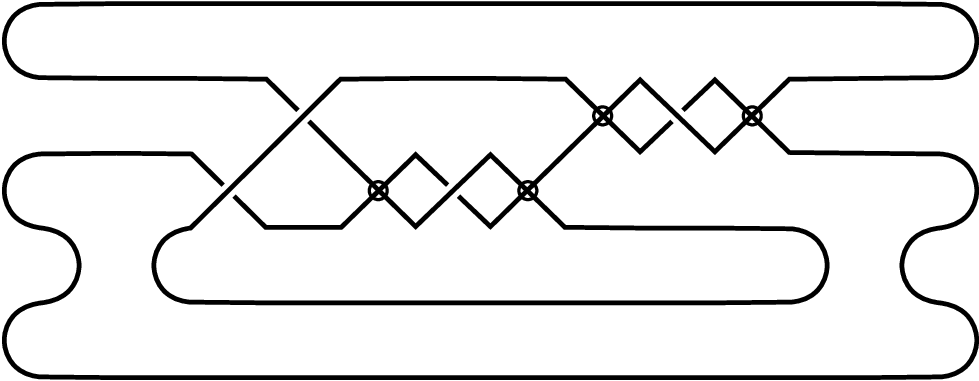} \\ \underline{4.106:} \\ \hline
    \end{tabular}
    \caption{$\circlearrowleft$ Two knots having the same generalized $U_q(\mathfrak{gl}(1|1))$ polynomials but different generalized $U_q(\mathfrak{gl}(2|2))$ polynomials.} 
    \label{fig_same_11_diff_22}
\end{figure}

\subsection{Same $\widetilde{f}^{\,\,2|0}$, but different $\widetilde{f}^{\,\,3|1}$} The phenomenon exhibited in the previous section for $m=n=1$ also holds for various values of $m-n$. For $K$ a classical link diagram, $f^{2|0}_K(q)/[2]_q$ is the Jones polynomial of $K$, and hence, so is $f^{m|n}_K(q)/[2]_q$ whenever $m-n=2$. However, this is not generally case for virtual link diagrams. As a first example, consider the 2-component virtual links 2-4-48 and 2-4-91 from Bartholomew's virtual link table \cite{bartholomew}. Braid words $\beta_1, \beta_2$, respectively, for these virtual links are given at the top of Figure \ref{fig_same_20_diff_31}. We have:
\[
    \widetilde{f}^{\,\,2|0}_{\widehat{\beta}_1}(q,w)=q^8-q^7-2 q^6+q^5+q^4-q^3+\frac{1}{q^2}+2 q+\frac{1}{q}+1=\widetilde{f}^{\,\,2|0}_{\widehat{\beta}_2}(q,w)
\]
However, the $U_q(\mathfrak{gl}(3|1))$ invariants are different:
\begin{align*}
    \widetilde{f}^{\,\,3|1}_{\widehat{\beta_1}}(q,w) &=-q^{18}w+q^{18}-\frac{q^{16}}{w}+q^{16}+q^{15}-q^{14} w^2+q^{14} w-\frac{q^{14}}{w}+q^{13}-\frac{2 q^{12}}{w}+4 q^{11} w-q^{11}+\frac{q^{10}}{w^2} \\
    &-\frac{q^{10}}{w}-q^{10}-2 q^9+q^8 w^2+\frac{q^8}{w^2}-q^8 w+q^7+\frac{q^6}{w^2}+\frac{q^6}{w}-6 q^6-4 q^5 w+2 q^5-\frac{q^4}{w^2}-q^4 w\\
    &+\frac{q^4}{w}+6 q^4-\frac{q^2}{w^2}-q^2 w+q^2+\frac{1}{q^2}+3 q+\frac{1}{q}-\frac{1}{w^2}+1 \\
    \widetilde{f}^{\,\,3|1}_{\widehat{\beta_2}}(q,w) &=-q^{18} w+q^{18}-\frac{q^{16}}{w}+q^{16}+q^{15}-\frac{q^{14}}{w}+q^{13}-\frac{q^{12}}{w}+2 q^{11} w-q^{11}-q^{10} w-\frac{q^{10}}{w}+\frac{2 q^9}{w}\\ 
    &-2 q^9+\frac{q^8}{w}+q^8+q^7-6 q^6-2 q^5 w+2 q^5+\frac{q^4}{w}+5 q^4-\frac{2 q^3}{w}-q^2 w-\frac{q^2}{w}+\frac{1}{q^2}+3 q+\frac{1}{q}+1
\end{align*}
Note that the same thing happens for 2-4-67 and 2-4-100. Braid words $\beta_3,\beta_4$ for these virtual links can be found in the accompanying \emph{Mathematica} above. The common value of the $U_q(\mathfrak{gl}(2|0))$ is:
\[
\widetilde{f}^{\,\,2|0}_{\widehat{\beta_3}}(q,w)=-\frac{1}{q^8}-\frac{2}{q^7}+\frac{3}{q^6}+\frac{4}{q^5}-\frac{2}{q^4}+q^3+\frac{1}{q^3}+\frac{1}{q^2}+q-\frac{3}{q}+1 =\widetilde{f}^{\,\,2|0}_{\widehat{\beta_4}}(q,w)
\]
Again, the $U_q(\mathfrak{gl}(3|1))$ invariants are distinct:
\begin{align*}
\widetilde{f}^{\,\,3|1}_{\widehat{\beta_3}}(q,w) &=-q^{15}w^3+q^{15} w^2+q^{14} w^2-q^{14} w+q^{12} w^2-2 q^{12} w+q^{12}-q^{11} w^2-q^{11}\\
    &+q^{10} w^2+6 q^{10} w+q^{10}+q^9 w^3+\frac{q^9}{w^2}-3 q^9 w-\frac{2 q^9}{w}-3 q^9-3 q^8 w^2-4 q^8 w\\
    &+q^8-\frac{1}{q^8}+2 q^7 w^2+\frac{q^7}{w^2}+7 q^7 w-\frac{q^7}{w}+3 q^7-\frac{2}{q^7}-3 q^6 w^2-4 q^6 w-3 q^6 \\ 
    &+\frac{1}{q^6}-4 q^5 w^2+\frac{1}{q^5 w^2}+\frac{2 w}{q^5}+3 q^5+\frac{2}{q^5}+5 q^4 w^2+12 q^4 w+\frac{2 w}{q^4}-q^4+2 q^3 w^2\\
    &-\frac{2 q^3}{w^2}+\frac{1}{q^3 w^2}-7 q^3 w+\frac{q^3}{w}-\frac{2 w}{q^3}-q^3-\frac{2}{q^3}-q^2 w^2-5 q^2 w+\frac{4 w}{q^2}+3 q^2+\frac{2}{q^2}\\
    &+q w^2-\frac{2 q}{w^2}-\frac{w^2}{q}+3 q w-\frac{1}{q w}-\frac{3 w}{q}+\frac{5}{q}-w^2-8 w+2
    \end{align*}
    \begin{align*}
    \widetilde{f}^{\,\,3|1}_{\widehat{\beta_4}}(q,w) &= -q^{13}w+q^{13}-q^{12} w-\frac{q^{12}}{w}+q^{12}+3 q^{11}+\frac{q^{10}}{w^2}+2 q^{10} w-\frac{q^{10}}{w}-\frac{3 q^9}{w^2}-2 q^9 w\\ &-\frac{2 q^9}{w}
    -3 q^9+\frac{2 q^8}{w^2}-\frac{1}{q^8 w^2}-2 q^8 w+\frac{4 q^8}{w}+3 q^8-\frac{1}{q^8}+\frac{1}{q^7 w^2}+4 q^7 w-\frac{2 q^7}{w}\\&+\frac{2}{q^7 w}+4 q^7-\frac{2}{q^7}-\frac{q^6}{w^2}-2 q^6 w-\frac{2 q^6}{w}+\frac{1}{q^6 w}-5 q^6+\frac{2}{q^6}-\frac{q^5}{w^2}-\frac{1}{q^5 w^2}+\frac{4 q^5}{w}\\&-\frac{1}{q^5 w}+q^5+\frac{1}{q^5}-\frac{q^4}{w^2}+\frac{1}{q^4 w^2}+4 q^4 w-\frac{2 q^4}{w}+\frac{w}{q^4} 
    +\frac{3}{q^4 w}-q^4+\frac{2 q^3}{w^2}+\frac{1}{q^3 w^2}\\&-3 q^3 w+\frac{q^3}{w}-\frac{w}{q^3}-\frac{3}{q^3 w}-8 q^3-\frac{1}{q^3}-\frac{3 q^2}{w^2}+\frac{1}{q^2 w^2}-q^2 w+\frac{8 q^2}{w}+\frac{w}{q^2}-\frac{6}{q^2 w}\\&+7 q^2+\frac{q}{w^2}-\frac{4 q}{w}+\frac{2}{q w}+2 q+\frac{6}{q}+\frac{1}{w^2}-2 w-\frac{4}{w}
\end{align*}
\begin{figure}[htb]
    \begin{tabular}{|cc|}  \hline
    &  \\
    \includegraphics[width=2in]{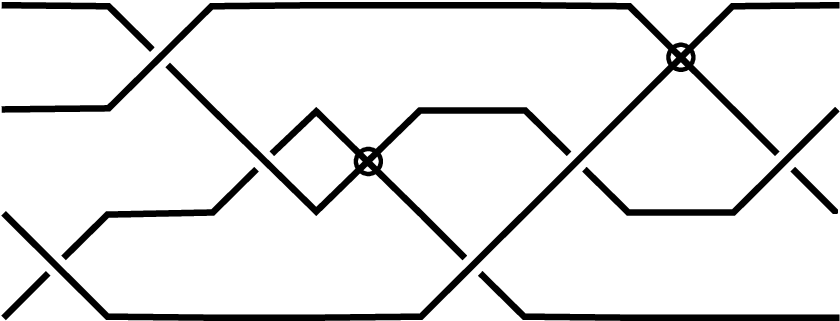} & \includegraphics[width=2in]{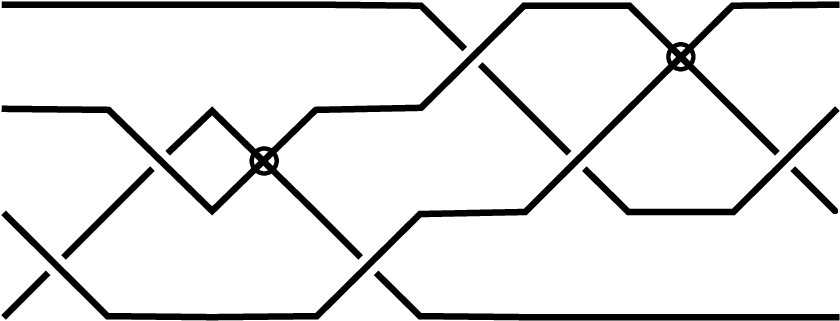} \\ \underline{2-4-48:} & \underline{2-4-91:}\\ \hline
    \end{tabular}
    \caption{Braids whose closures have the same generalized $U_q(\mathfrak{gl}(2|0))$ polynomials but different generalized $U_q(\mathfrak{gl}(3|1))$ polynomials.}
    \label{fig_same_20_diff_31}
\end{figure}

\subsection{Trivial GAP, but non-trivial $\widetilde{f}^{\,\,2|2},\widetilde{f}^{\,\,3|3}$} The virtual knot $K=5.114$ is not almost classical but has trivial generalized Alexander polynomial. A virtual braid word $\beta$ for 5.114 is shown in Figure \ref{fig_5_114}. The extended $U_q(\mathfrak{gl}(m|m)$ invariants for $m>1$ can detect that $K$ is not almost classical in the rotational category:

\begin{align*}
    \widetilde{f}^{\,\,2|2}_{\widehat{\beta}}(q,w) &=2 q^{14} w-2 q^{14}-4 q^{12} w+\frac{2 q^{12}}{w}+4 q^{12}+q^{11} w-5 q^{11}-\frac{4 q^{10}}{w}\\
    &-2 q^9 w+\frac{q^9}{w}+14 q^9+4 q^8 w-6 q^8-\frac{2 q^7}{w}-12 q^7-2 q^6 w+\frac{4 q^6}{w}\\
    &+8 q^6+2 q^5 w+2 q^5-\frac{2 q^4}{w}-6 q^4-q^3 w+\frac{2 q^3}{w}+q^3+2 q^2-\frac{q}{w} 
\end{align*}
\begin{align*}
    \widetilde{f}^{\,\,3|3}_{\widehat{\beta}}(q,w) &=2 q^{22} w-2 q^{22}-2 q^{20} w+\frac{2 q^{20}}{w}+2 q^{20}-4 q^{19}+2 q^{18} w-\frac{2 q^{18}}{w}\\&+4 q^{17}-6 q^{16} w+\frac{2 q^{16}}{w}+4 q^{16}+q^{15} w-5 q^{15}+2 q^{14} w-\frac{6 q^{14}}{w}\\
    &-4 q^{14}-q^{13} w+\frac{q^{13}}{w}+22 q^{13}+\frac{2 q^{12}}{w}-6 q^{12}-q^{11} w-\frac{q^{11}}{w}-21 q^{11}\\&+4 q^{10} w+14 q^{10}-\frac{q^9}{w}+4 q^9+\frac{4 q^8}{w}-12 q^8+q^7 w-3 q^7-2 q^6 w\\&+6 q^6+q^5 w+\frac{q^5}{w}+2 q^5-\frac{2 q^4}{w}-4 q^4-q^3 w+\frac{q^3}{w}+q^3+2 q^2-\frac{q}{w}.
\end{align*}

It is interesting to note that $5.114$ also has trivial graded genus. However, 5.114 is not slice as it has nontrivial Rasmussen invariant \cite{boden_chrisman_21}. The next subsection we give further evidence that $\widetilde{f}^{m|m}$ is a virtual slice obstruction. 

\begin{figure}[htb]
\begin{tabular}{|c|} \hline 
\\
\includegraphics[width=2.2in]{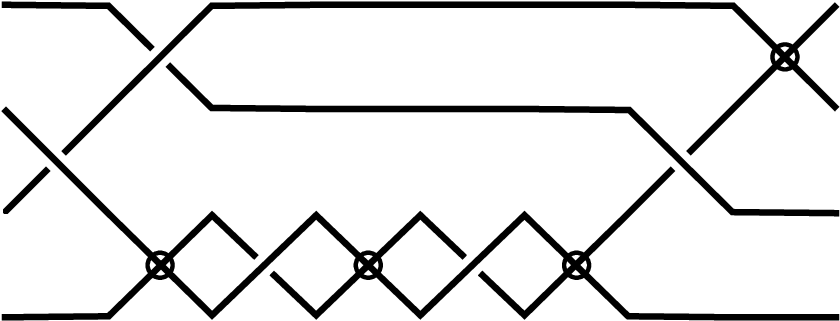}\\
\underline{5.114:} \\ \hline
\end{tabular}
\caption{$\circlearrowleft$ A knot with trivial generalized Alexander polynomial but nontrivial generalized $U_q(\mathfrak{gl}(2|2))$ and $U_q(\mathfrak{gl}(3|3))$ polynomials.}
\label{fig_5_114}
\end{figure}

\subsection{Is $\widetilde{f}^{\,m|m}$ a virtual slice obstruction?} \label{sec_slice} Recall two virtual knots $K,K^*$ are said to be \emph{(virtually) concordant} if they may be obtained from one another by a connected sequence of extended Reidemeister moves, births, deaths, saddles such that:
\[
\#(\text{births})-\#(\text{saddles})+\#(\text{deaths})=0
\]
A virtual knot that is concordant to the unknot is said to be \emph{(virtually) slice}. For many examples of slice virtual knots, see \cite{bcg1,chrisman_22}. Some virtual slice obstructions are the odd writhe, the various index polynomials, the graded genus, the generalized Alexander polynomial, the Rasmussen invariant, and the extended Milnor invariants. There are exactly thirteen slice virtual knots having classical crossing number at most four \cite{bcg1}. Our next result follows from direct calculation, which is included in the accompanying \emph{Mathematica} notebook. To conserve space, the virtual tangle decompositions of these virtual knots has been omitted in this version. They can be found in the \emph{Mathematica} notebook, or alternatively, in the first arXiv version of this paper.

%\begin{figure}[htb]
%    \begin{tabular}{|c|c|} \hline &  \\
%    \includegraphics[width=2in]{4_8_rot.eps} & \includegraphics[width=2in]{4_55_rot.eps} \\ \underline{4.8:} & \underline{4.55:}\\ \hline
%    &  \\
%    \includegraphics[width=2in]{4_75_rot.eps} %& \includegraphics[width=2in]{4_77_rot.eps} \\ \underline{4.75:} & \underline{4.77:}\\ \hline
%    \multicolumn{2}{|c|}{} \\
%    \multicolumn{2}{|c|}{\includegraphics[width=2.8in]{4_56_rot.eps}} \\ 
%    \multicolumn{2}{|c|}{\underline{4.56:}}
%    \\ \hline
%    \multicolumn{2}{|c|}{} \\
%    \multicolumn{2}{|c|}{\includegraphics[width=2.8in]{4_76_rot.eps}} \\ 
 %   \multicolumn{2}{|c|}{\underline{4.76:}}
%    \\ \hline
%    \end{tabular}
%    \caption{$\circlearrowleft$ Slice virtual knots up to classical crossing number 4 (Part 1 of 2)}
%    \label{fig_slice_I}
%\end{figure}

\begin{theorem} If $K$ is any one of the slice virtual knots 4.8, 4.55, 4.56, 4.58, 4.59, 4.71, 4.72, 4.75, 4.76, 4.77, 4.90, 4.98, 4.99, then: 
\[
\widetilde{Q}^{1|1}_{\zh(K)}(q,w)=\widetilde{Q}^{2|2}_{\zh(K)}(q,w)=\widetilde{Q}^{3|3}_{\zh(K)}(q,w)=\widetilde{Q}^{4|4}_{\zh(K)}(q,w)=\widetilde{Q}^{5|5}_{\zh(K)}(q,w)=0.
\]
\end{theorem}
On the other hand, the generalized $U_q(\mathfrak{gl}(m|n))$ polynomials are generally non-zero when $m \ne n$ and $n \ge 1$.  For example, this is the case for the $U_q(\mathfrak{gl}(2|1))$ invariants of all slice virtual knots having crossing number 4. We give just two illustrations:
\begin{align*}
\widetilde{f}^{\,\,2|1}_{4.8}(q,w) &= -q^4 w+2 q^3 w+\frac{2}{q^3 w}-\frac{2}{q^3}+2 q^2 w-\frac{q^2}{w}-\frac{1}{q^2 w}-4 q w\\&+\frac{2 q}{w}-\frac{4}{q w}+\frac{2 w}{q}-2 q+\frac{4}{q}-w+\frac{2}{w}+1\\
\widetilde{f}^{\,\,2|1}_{4.72}(q,w) &=q^7-\frac{2 q^5}{w}-q^5+\frac{4 q^3}{w}-\frac{q^2}{w}-\frac{1}{q^2 w}-\frac{2 q}{w}-q+\frac{1}{q}+\frac{2}{w}+1
\end{align*}
Since the generalized Alexander polynomial $\widetilde{f}^{\,\,1|1}(q,w)$ is a virtual slice obstruction, the observations in this subsection and the previous one lead to the following conjecture.

\begin{conjecture} \label{conjecture} For all $m \ge 1$, if $K$ is a slice virtual knot, then $\widetilde{f}^{\,\,m|m}_K(q,w)=0$.
\end{conjecture}

\begin{remark} It is important to emphasize that here we mean it is a slice obstruction in the virtual category, not in the rotational category, even though $\widetilde{f}^{m|m}$ is a  rotational invariant for $m>1$. 
\end{remark}

%\begin{figure}[htb]
%    \begin{tabular}{|c|c|} \hline &  \\
%    \includegraphics[width=2in]{4_58_rot.eps} & \includegraphics[width=2in]{4_59_rot.eps} \\ \underline{4.58:} & \underline{4.59:}\\ \hline
%    &  \\
%    \includegraphics[width=2in]{4_71_rot.eps} & \includegraphics[width=2in]{4_72_rot.eps} \\ \underline{4.71:} & \underline{4.72:}\\ \hline & \\
%    \includegraphics[width=2in]{4_90_rot.eps} & \includegraphics[width=2in]{4_98_rot.eps} \\ \underline{4.90:} & \underline{4.98:}\\ \hline
%    \multicolumn{2}{|c|}{} \\
%    \multicolumn{2}{|c|}{\includegraphics[width=2in]{4_99_rot.eps}} \\ 
%    \multicolumn{2}{|c|}{\underline{4.99:}}
%    \\ \hline
%    \end{tabular}
%    \caption{$\circlearrowleft$ Some slice virtual knots up to classical crossing number 4 (Part 2 of 2).}
%   \label{fig_slice_II}
%\end{figure}

\subsection{Virtual knots of unknown slice status} \label{sec_unknown} There are 92800 virtual knots having classical crossing number at most six. Of these, 1321 are slice, 91475 are not slice, and four are of unknown slice status \cite{chrisman_22}. Gauss diagrams for these last four are shown in Figure \ref{fig_4_unknown}. We will show that if Conjecture \ref{conjecture} holds up to $m=2$, then two of these, $K_1=6.31445$ and $K_2=6.62002$, are not slice.

Consider first the virtual knot $K_1=6.31445$. A diagram is given at the top of Figure \ref{fig_first_horseman}. This virtual knot has trivial graded genus, generalized Alexander polynomial, and Rasmussen invariant. Furthermore, the extend Milnor invariant vanish up to order at least nine \cite{chrisman_22}. We will use the skein relation in Theorem \ref{thm_skein_relation} to show that $\widetilde{f}^{\,\,2|2}_K(q,w)$ is non-vanishing. Observe that changing the circled $\ominus$ crossing in Figure \ref{fig_first_horseman} to a $\oplus$ crossing yields a diagram rotationally equivalent to the unknot. The switched crossing is depicted as a pink arrow in Figure \ref{fig_4_unknown}. Smoothing at the circled crossing yields the $2$-component link $L$ shown at the bottom of Figure \ref{fig_first_horseman}. Then the skein relation implies that:
\begin{align*}
    \widetilde{f}^{\,\,2|2}_{K_1}(q,w) &= \widetilde{f}^{\,\,2|2}_{\bigcirc}(q,w)-\left(q-\tfrac{1}{q}\right)\widetilde{f}^{\,\,2|2}_L(q,w) \\
    &=q^{-9} w^{-2}(q-1)^3 (q+1)^3 \left(q^2+1\right) (1-q w) \left(q^7 w-q^6+4 q^4 w-2 q^4-2 q^2+q w-1\right)
\end{align*}

\begin{figure}[htb]
\begin{tabular}{|c|} \hline
\def\svgwidth{4in}
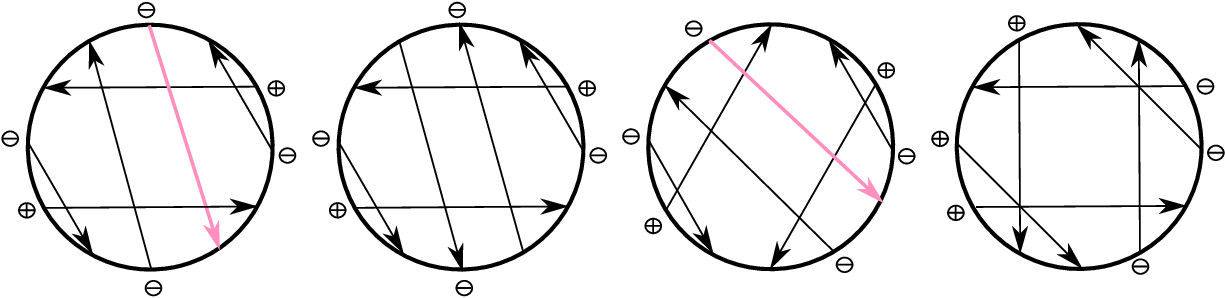\\ \\ \hline
\end{tabular}
\caption{Four virtual knots of unknown slice status.} \label{fig_4_unknown}
\end{figure}

\begin{figure}[htb]
\begin{tabular}{|c|} \hline \\ 
    \includegraphics[width=3in]{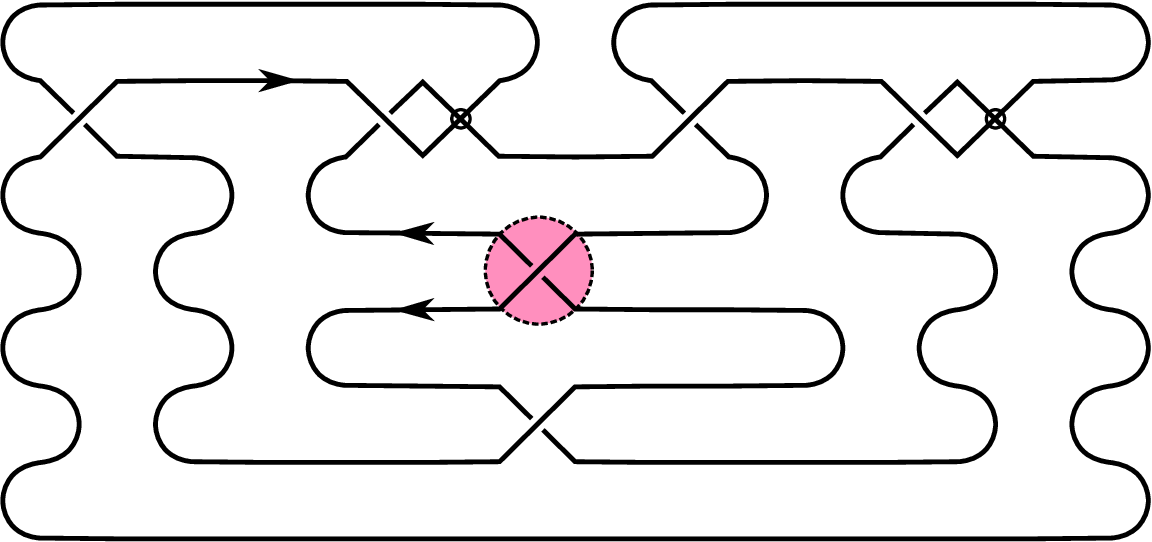} \\  \underline{6.31445:} \\ \hline \\ \includegraphics[width=3in]{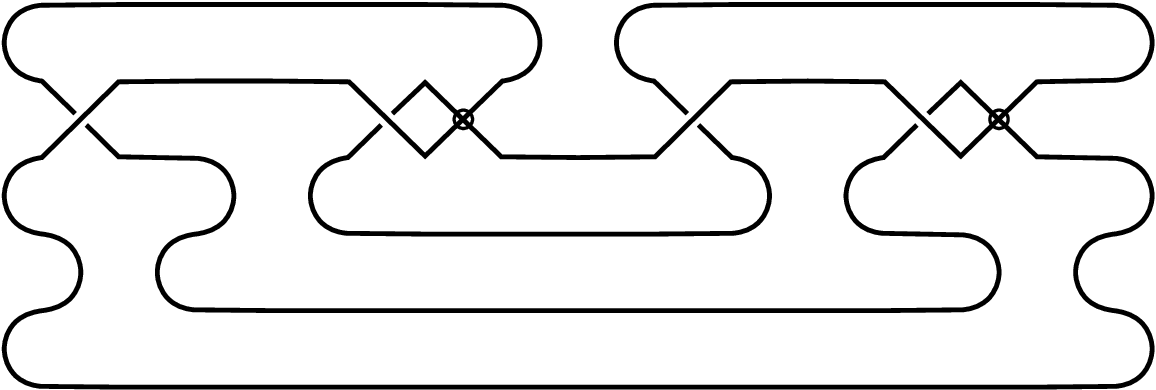} \\ 6.31445 (smoothed): \\ \hline
    \end{tabular}
    \caption{$\circlearrowleft$ The virtual knot 6.31445 and the same virtual knot with the pink circled crossing smoothed and simplified.}
    \label{fig_first_horseman}
\end{figure}

Next, consider the virtual knot $K_2=6.62002$. A virtual tangle decomposition is shown in Figure \ref{fig_6_62002}. This knot also has trivial graded genus, generalized Alexander polynomial, and Rasmussen invariant. The extended Milnor invariants of $K_2$ vanish up to order at least seven. To show that its generalized $U_q(\mathfrak{gl}(2|2))$ invariant is nontrivial, we again use the skein relation in Theorem \ref{thm_skein_relation}. Let $K_2^{\oplus}$ be the virtual knot diagram obtained from $K_2$ by switching the circled $\ominus$ crossing to a $\oplus$ crossing. In Figure \ref{fig_4_unknown}, this corresponds to switching the direction and sign of the pink arrow. It is easy to see that $K_2^{\oplus}$ is equivalent in the virtual category to $4.77$. Indeed, after performing the obvious Reidemeister 2 move, the resulting diagram has the same Gauss diagram as $4.77$. Hence, if Conjecture \ref{conjecture} holds up to $m=2$, then $\widetilde{f}_{K_2^{\oplus}}^{\,\,2|2}(q,w)=0$. Smoothing at the crossing returns a link that is rotationally equivalent to the link $L$ shown at the bottom of Figure \ref{fig_first_horseman}. Hence:
\begin{align*}
    \widetilde{f}^{\,\,2|2}_{K_2}(q,w) &= \widetilde{f}^{\,\,2|2}_{K_2^{\oplus}}(q,w)-\left(q-\tfrac{1}{q}\right)\widetilde{f}^{\,\,2|2}_L(q,w) \\
    &=q^{-9} w^{-2}(q-1)^3 (q+1)^3 \left(q^2+1\right) (1-q w) \left(q^7 w-q^6+4 q^4 w-2 q^4-2 q^2+q w-1\right).
\end{align*}

\begin{figure}[htb]
\begin{tabular}{|c|} \hline \\
\includegraphics[width=4in]{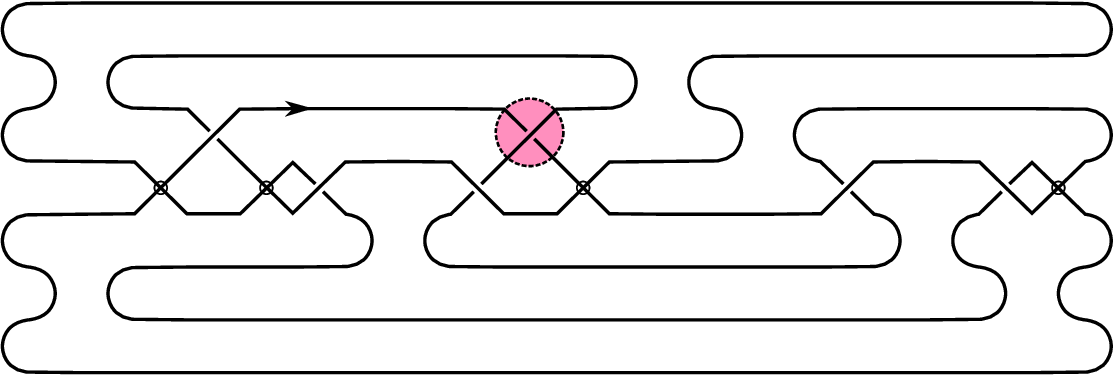} \\ \underline{6.62002:} \\ \hline
\end{tabular}
\caption{$\circlearrowleft$ The virtual knot 6.62002. The pink circled crossing is used the skein relation calculation of its generalized $U_q(\mathfrak{gl}(2|2))$ polynomial.}\label{fig_6_62002}
\end{figure}

For the remaining two virtual knots in Figure \ref{fig_4_unknown}, 6.52373 and 6.86951, the skein relation does not reduce the calculation of $\widetilde{f}^{\,\,2|2}$ to an easily computable virtual link. In each case, we obtained a virtual tangle diagram having width 8 after either smoothing or switching a crossing. Hence, the calculation involves some matrices of dimension $4^8 \times 4^8=65536 \times 65536$. Up to this point in the paper, computation time has ranged from nearly instantaneous to approximately 15 minutes. However, 6.52373 and 6.86951 proved to be beyond our immediately available resources.

\subsection*{Acknowledgments} The authors would like to express their thanks to K. Davis, C. Frohman, P. Pongtanapaisan, and R. Todd for numerous helpful conversations. The first named author was partially supported by funds and release time by The Ohio State University at Marion.

\bibliographystyle{amsplain}

\bibliography{0_super_bib}

\end{document}